\documentclass[11pt]{article}
\usepackage[margin=0.7in]{geometry}
\usepackage{amssymb, comment,hyperref,slashed,tensor, amsmath, amsthm}
\usepackage{hyperref, mathrsfs}

\usepackage{mathtools}
\mathtoolsset{showonlyrefs}

\usepackage{authblk}

\newtheorem{theorem}{Theorem}[section]
\newtheorem{lemma}[theorem]{Lemma}

\newtheorem{definition}[theorem]{Definition}

\newtheorem{remark}[theorem]{Remark}

\newcommand{\R}{\mathbb{R}}

\newcommand{\N}{\mathbb{N}}

\newcommand{\p}{\partial}

\allowdisplaybreaks				
\usepackage{xcolor} 
\usepackage{enumitem}

\definecolor{green}{rgb}{0,0.8,0} 

\newcommand{\relphantom}[1]{\mathrel{\phantom{#1}}}

\makeatletter
\newcommand{\nrm}{\@ifstar{\nrmb}{\nrmi}}
\newcommand{\nrmi}[1]{\Vert{#1}\Vert}
\newcommand{\nrmb}[1]{\left\Vert{#1}\right\Vert}
\newcommand{\abs}{\@ifstar{\absb}{\absi}}
\newcommand{\absi}[1]{\vert{#1}\vert}
\newcommand{\absb}[1]{\left\vert{#1}\right\vert}
\newcommand{\brk}{\@ifstar{\brkb}{\brki}}
\newcommand{\brki}[1]{\langle{#1}\rangle}
\newcommand{\brkb}[1]{\left\langle{#1}\right\rangle}
\newcommand{\set}{\@ifstar{\setb}{\seti}}
\newcommand{\seti}[1]{\{#1\}}
\newcommand{\setb}[1]{\left\{ #1\right\}}
\makeatother

\newcommand{\br}[1]{\overline{#1}}

\newcommand{\VERT}[1]{{\left\vert\kern-0.25ex\left\vert\kern-0.25ex\left\vert #1 
    \right\vert\kern-0.25ex\right\vert\kern-0.25ex\right\vert}}

\DeclareMathOperator{\supp}{supp}

\newcommand{\aeq}{\simeq}
\newcommand{\aleq}{\lesssim}
\newcommand{\ageq}{\gtrsim}

\newcommand{\lap}{\Delta}

\newcommand{\ud}{\mathrm{d}}
\newcommand{\rd}{\partial}

\newcommand{\imp}{\Rightarrow}

\newcommand{\peq}{\relphantom{=}}			
\newcommand{\pleq}{\relphantom{\leq}}			

\newcommand{\alp}{\alpha}
\newcommand{\bt}{\beta}
\newcommand{\gmm}{\gamma}
\newcommand{\Gmm}{\Gamma}
\newcommand{\dlt}{\delta}
\newcommand{\Dlt}{\Delta}
\newcommand{\eps}{\epsilon}

\newcommand{\kpp}{\kappa}
\newcommand{\lmb}{\lambda}

\newcommand{\sgm}{\sigma}

\newcommand{\Tht}{\Theta}

\newcommand{\Ups}{\Upsilon}

\newcommand{\bfPhi}{\boldsymbol{\Phi}}

\newcommand{\bbR}{\mathbb R}

\newcommand{\bbZ}{\mathbb Z}

\newcommand{\calC}{\mathcal C}

\newcommand{\calF}{\mathcal F}

\newcommand{\calH}{\mathcal H}

\newcommand{\calL}{\mathcal L}

\newcommand{\calO}{\mathcal O}

\newcommand{\calT}{\mathcal T}
\newcommand{\calU}{\mathcal U}

\title{Gradient blow-up for dispersive and dissipative perturbations \newline of the Burgers equation}
\author{Sung-Jin Oh\thanks{Department of Mathematics, UC Berkeley and School of Mathematics, Korea Institute for Advanced Study. \newline E-mail: sjoh@math.berkeley.edu} \space and Federico Pasqualotto\thanks{Department of Mathematics, UC Berkeley. E-mail: fpasqualotto@math.berkeley.edu}}
\date{}

\begin{document}

\maketitle

\begin{abstract}
We consider a class of dispersive and dissipative perturbations of the inviscid Burgers equation, which includes the fractional KdV equation of order $\alpha$, and the fractal Burgers equation of order $\beta$, where $\alpha, \beta \in [0,1)$, and the Whitham equation. For all $\alpha, \beta \in [0,1)$, we construct solutions whose gradient blows up at a point, and whose amplitude stays bounded, which therefore display a ``shock-like'' singularity. We moreover provide an asymptotic description of the blow-up. To our knowledge, this constitutes the first proof of gradient blow-up for the fKdV equation in the range $\alpha \in [2/3, 1)$, as well as the first description of explicit blow-up dynamics for the fractal Burgers equation in the range $\bt \in [2/3, 1)$. 

Our construction is based on modulation theory, where the well-known smooth self-similar solutions to the inviscid Burgers equation are used as profiles. A somewhat amusing point is that the profiles that are less stable under initial data perturbations (in that the number of unstable directions is larger) are more stable under perturbations of the equation (in that higher order dispersive and/or dissipative terms are allowed) due to their slower rates of concentration. Another innovation of this article, which may be of independent interest, is the development of a streamlined weighted $L^{2}$-based approach (in lieu of the characteristic method) for establishing the sharp spatial behavior of the solution in self-similar variables, which leads to the sharp H\"older regularity of the solution up to the blow-up time.
\end{abstract}

\section{Introduction}

In this article, we construct and describe the dynamics of solutions with smooth decaying initial data that exhibit gradient blow-up (while the solution itself remains bounded) for a wide class of perturbations of the inviscid Burgers equation 
\begin{equation} \label{eq:burgers} \tag{Burgers}
	\rd_{t} u + u \rd_{x} u = 0.
\end{equation}
Examples covered by our result include \emph{the fractional KdV equations}
\begin{equation} \label{eq:fkdv-0} \tag{fKdV}
	\rd_{t} u +u \rd_{x} u + \abs{D_{x}}^{\alp-1} \rd_{x} u = 0 \qquad \hbox{ for any } 0 \leq \alp < 1,
\end{equation}
\emph{Whitham's equation},
\begin{equation}  \label{eq:whitham} \tag{Whitham}
	\rd_{t} u +u \rd_{x} u + \Gmm(D_{x}) \rd_{x} u = 0 \qquad \hbox{ where } \Gmm(\xi) = \sqrt{\frac{\tanh \xi}{\xi}},
\end{equation}
as well as the \emph{fractal Burgers equations} 
\begin{equation} \label{eq:fburgers} \tag{fBurgers}
	\rd_{t} u +u \rd_{x} u + \abs{D_{x}}^{\bt} u = 0 \qquad \hbox{ for any } 0 \leq \bt < 1,
\end{equation}
where $\abs{D_{x}}^{\alp} = (-\lap)^{\frac{\alp}{2}}$.
These equations arise naturally as model problems in the theory of water waves~\cite{NaSh}, which makes the problem of singularity formation a natural one.

Gradient blow-up (also referred to as \emph{shock formation} or \emph{wave breaking} in the contexts of hyperbolic conservation laws or water waves, respectively \cite{Whi}) for~\eqref{eq:fkdv-0} in the range $ 0 \leq \alp < 1$ was conjectured from numerical experiments in~\cite{KlSa} and later in~\cite{HuTa}. While the existence of gradient-blow-up solutions was known for \eqref{eq:fkdv-0} in the range $0 \leq \alp < \frac{2}{3}$ and for \eqref{eq:whitham} (see for instance \cite{HuTa, Hu2017, SautWang, Yang2020}), our result seems to be the first construction of such blow-up solutions in the range $\frac{2}{3} \leq \alp < 1$. We furthermore give a quantitative description of the blow-up dynamics in a stable blow-up regime in the case $0 < \alp < \frac{2}{3}$, which seems to have not appeared in the literature (note that in the case $\alp =0$, the Burgers--Hilbert equation, a precise description of the same blow-up dynamics has already appeared in the recent work of Yang \cite{Yang2020}, which is discussed further below). In all cases, we also observe that there exist smooth compactly supported initial data of either sign (everywhere nonnegative or nonpositive) that give rise to the same blow-up behavior, which disproves (yet suggests a refinement of) a conjecture made by Klein--Saut \cite[Conjecture~3]{KlSa} for \eqref{eq:whitham}; see Remark~\ref{rem:id-sign-0} below.

Concerning \eqref{eq:fburgers}, gradient blow-up was shown in the papers \cite{KiNaSh, AlDrVo, DoDuLi} for all ranges $0 < \bt < 1$. Moreover, in a recent work of Chickering--Moreno-Vasquez--Pandya \cite{CMVP}, which we learned of while preparing our article, a quantitative description of a stable blow-up dynamics analogous to \cite{Yang2020} in the Burgers--Hilbert case was given in the case $0 < \bt < \frac{2}{3}$. Our work provides an alternative, independent description of the same stable blow-up regime, as well as the precise description of some examples of gradient-blow-up solutions in the case $\frac{2}{3} \leq \bt < 1$, which seems new.

Our proof is based on a systematic study of the stability of self-similar blow-up solutions for \eqref{eq:burgers} under perturbations of the equation. As is well-known, \eqref{eq:burgers} admits a two-parameter family of scaling symmetries (corresponding to separate rescaling of time and space), which results in a one-parameter family of self-similar change of variables\footnote{Here, $(s, y, U)$ are the self-similar variables for solutions defined for negatives times that blow up at $(t, x) = (0, 0)$.} 
\begin{equation*}
(t, x, u) \mapsto (s, y, U) = \left( -\log (-t), \tfrac{x}{(-t)^{b}}, \tfrac{(-t)}{(-t)^{b}} u \right),
\end{equation*}
parametrized by $b > 0$. Among these $b$'s, there exist countably many choices that lead to \emph{smooth} self-similar solutions (i.e., $s$-independent solution in the self-similar variables), namely $b = \frac{2k+1}{2k}$ for $k= 1, 2, \ldots$ \cite{EgFo}. In what follows, we will refer the self-similar solutions in the case $k = 1$ as \emph{ground states}, and those in the case $k \geq 2$ as \emph{excited states}. 

A key predecessor of this article is the recent work of Yang \cite{Yang2020} that, based on the modulation-theoretic approach of Buckmaster--Shkoller--Vicol \cite{BuShVi1}, constructed an open set of initial data giving rising to gradient-blow-up solutions to the Burgers--Hilbert equation  (i.e., \eqref{eq:fkdv} with $\alp = 0$) with ground state self-similar solutions to \eqref{eq:burgers} as blow-up profiles (see also the very recent work \cite{CMVP} for \eqref{eq:fburgers} with $0 < \bt < \frac{2}{3}$). In this article, we extend \cite{Yang2020} (and \cite{CMVP}) to more general perturbations of the Burgers equation, while simultaneously allowing for the use of excited states as blow-up profiles\footnote{At this point, we note the interesting work of Collot--Ghoul--Masmoudi \cite{Collot2018}, which considered a two-dimensional partially dissipative perturbation of \eqref{eq:burgers} and constructed blow-up solutions with both ground and excited states as blow-up profiles. We refer to Section~\ref{sub:history} for further discussion.}.

In fact, these two extensions go hand in hand. A somewhat amusing point, made precise in this article, is that {\it higher excited self-similar profiles, which are {\bf less stable under perturbations of the initial data} (i.e., they are stable under higher co-dimensional set of initial data perturbations), are {\bf more stable under perturbations of the equation}, in that higher order dispersive and/or dissipative terms are allowed\footnote{It is important to distinguish the perturbations of the equations discussed here, which are terms of the form $L u$ for some linear operator $L$, with an external forcing term $f$, which is independent of $u$. The effect of an external forcing term with compact support in spacetime should resemble that of a compactly supported initial data perturbation.}.} An explanation behind this phenomenon is as follows. The key factor that determines the stability of a self-similar profile under perturbations of the equation turns out to be its \emph{rate of concentration} (i.e., the exponent $b = \frac{2k+1}{2k}$), and the slower rates of the excited states lead to larger classes of admissible perturbations of the equation. To see this point heuristically, one may simply compare the ``strength'' of each term in the equation on the characteristic time and length scales of the $k$-th Burgers self-similar profile, which are $\sim (-t)$ and $\sim (-t)^{\frac{2k+1}{2k}}$, respectively. For instance, for \eqref{eq:fkdv-0}, compare
\begin{equation*}
	\rd_{t} u \sim (\tau - t)^{-1} u \quad \hbox{ vs. } \quad \abs{D_{x}}^{\alp-1} \rd_{x} u \sim (-t)^{-\frac{2k+1}{2k} \alp} u.
\end{equation*}
(In self-similar variables for {\eqref{eq:burgers}}, the ``strength'' of $u \rd_{x} u$ is the same as that of $\rd_{t} u$.) The ``strength'' of the perturbation $\abs{D_{x}}^{\alp-1} \rd_{x} u$ is weaker than that of $\rd_{t} u$ when $- \frac{2k+1}{2k} \alp < -1$, or equivalently, $\alp < \frac{2k}{2k+1}$; note that this range improves as $k$ increases. Our main theorem demonstrates that under this condition, the Burgers self-similar profile with $b = \frac{2k+1}{2k}$ is stable under passage to \eqref{eq:fkdv-0}, leading to gradient-blow-up solutions to \eqref{eq:fkdv-0} that asymptote to the same Burgers self-similar profile near the singularity. On the other hand, it will become apparent that the instability of the self-similar profile under initial data perturbations does \emph{not} affect its stability under perturbations of the equation.

We remark that the preceding points are in parallel with the recent remarkable works of Merle--Rapha\"el--Rodnianski--Szeftel on singularity formation for the compressible Euler equations, the compressible Navier--Stokes equations and defocusing nonlinear Schr\"odinger equations (NLS). In \cite{MRRS-E}, smooth self-similar profiles for the polytropic compressible Euler equations with characteristic length scales $(-t)^{\frac{1}{r}}$ are constructed for discrete values of $r$. Then these profiles are used to demonstrate singularity formation for second-order dissipative (i.e., compressible Navier--Stokes \cite{MRRS-NS}) and dispersive (i.e., energy-supercritical defocusing NLS after Madelung transform \cite{MRRS-NLS}) perturbations of the Euler equations, where the admissible values of $r$ with respect to each perturbation may be  determined with similar heuristics as above.

Another innovation in this article is the introduction of a robust yet sharp weighted $L^{2}$-based method to establish the optimal spatial growth of the solution in the appropriate self-similar variables, in lieu of the method of characteristics employed in \cite{BuShVi1} and subsequent works. Such information is necessary to establish the sharp H\"older regularity, and in general even the boundedness (due to the lack of the maximum principle), of the solution up to the blow-up time. We refer the reader to Section~\ref{subsec:strategy} below for a short description of this method.

\subsection{First statement of the main result and discussion}
We now precisely state the class of equations studied in this article. For $u : \bbR_{t} \times \bbR_{x} \to \bbR$, consider
\begin{equation}\label{eq:fkdv}
	\rd_{t} u +u \rd_{x} u + \Gmm \rd_{x} u + \Ups u = 0.
\end{equation}
Here, $\Gmm$ and $\Ups$ are Fourier multipliers with symbols $\Gmm(\xi)$ and $\Ups(\xi)$ satisfying the following properties:
\begin{enumerate}
\item $\Gmm(\xi), \Ups(\xi) \in C^{\infty}(\bbR \setminus \set{0})$ are \emph{real-valued} and \emph{even}\footnote{This property is equivalent to the requirement that the Fourier multiplier $\Gmm \rd_{x} + \Ups$ maps real-valued functions to real-valued functions.};
\item $\Gmm(\xi) \xi$ and $\Ups(\xi)$ are symbols of order $\alp$ and $\bt$ with $0 \leq \alp, \bt < 1$, in the sense that for every multi-index $I$, there exists constants $C_{\Gmm, N}, C_{\Ups, N} > 0$ such that for every $\abs{\xi} \geq 1$,
\begin{equation} \label{eq:symb-bnd}
	\abs{\rd_{\xi}^{I} (\Gmm(\xi) \xi)} \leq C_{\Gmm, N} \abs{\xi}^{\alp-\abs{I}}, \qquad
	\abs{\rd_{\xi}^{I} \Ups(\xi)} \leq C_{\Ups, N} \abs{\xi}^{\bt-\abs{I}}.
\end{equation}
On the other hand, we assume that $\Gmm(\xi) \xi$ and $\Ups(\xi)$ are bounded on $\set{\xi \in \bbR : \abs{\xi} \leq 1}$.

\item $\Ups(\xi) \geq 0$ (i.e., $\Ups$ is elliptic).
\end{enumerate}
Clearly, \eqref{eq:fkdv-0}, \eqref{eq:whitham} are examples of \eqref{eq:fkdv} with $\Upsilon = 0$, and \eqref{eq:fburgers} is an example of \eqref{eq:fkdv} with $\Gmm = 0$. The order of $\Gmm$ for \eqref{eq:whitham} is $\alp = \frac{1}{2}$.

By the standard energy method, it can be readily seen that the initial value problem for \eqref{eq:fkdv} is locally well-posed in $H^{s}$ for any $s > \frac{3}{2}$. Our main theorem concerns the formation of singularity for \eqref{eq:fkdv} starting from smooth and well-localized initial data. In simple terms, the statement of our main theorem is as follows.
\begin{theorem}\label{thm:main-simple}
Let $k$ be a positive integer such that $\alpha, \beta < \frac{2k}{2k+1}$. Then there exist smooth initial data $u_0$ for \eqref{eq:fkdv} such that the resulting solution of \eqref{eq:fkdv} blows up in finite time in $\calC^{\sgm}$ for every $\sgm > \frac{1}{2k+1}$, while its $\mathcal{C}^{\frac{1}{2k+1}}$ norm stays bounded until the blow-up time. In the case $k=1$, and for $\alpha <  \frac 23$ and $\beta <  \frac 23$, the blow-up behavior is stable in $H^{5}$. In the case $k > 1$, these initial data form a ``codimension $2k-2$ subset'' of $H^{2k+3}$.
\end{theorem}

For more precise statements regarding the description of the initial data and blow-up dynamics, we refer to Theorem~\ref{thm:main}, Lemma~\ref{lem:id-map} and the ensuing discussion. Note moreover that it would be possible, by a more refined analysis, to show that the ``codimension $2k-2$ subset'' of $H^{2k+3}$ in the above statement constitutes in reality a suitably regular submanifold of initial data. However, this is not carried out in the present work.

\begin{remark}[Stable blow-up regime]
Note that Theorem~\ref{thm:main-simple} applies to \eqref{eq:fkdv-0} for $0 \leq \alp < \frac{2}{3}$, \eqref{eq:whitham} and \eqref{eq:fburgers} for $0 \leq \bt < \frac{2}{3}$ with $k = 1$, and as a result we obtain a blow-up behavior that is stable under initial data perturbations for these equations. On the other hand, in the range $\frac{2}{3} \leq \alp < 1$, the term $\abs{D_{x}}^{\alp-1} \rd_{x}$ cannot be merely treated as a small perturbation for $k = 1$, and we must perturb off of an excited Burgers self-similar profile. Description of a stable (under initial data perturbations) blow-up for \eqref{eq:fkdv-0} for $\frac{2}{3} \leq \alp < 1$ remains an open problem. 
\end{remark}

\begin{remark}[The sign of the initial data] \label{rem:id-sign-0}
In \cite[Conjecture~3]{KlSa}, based on numerical investigation, the following interesting conjecture concerning the blow-up dynamics for \eqref{eq:whitham} and the sign of the initial data was made: 
\begin{quote}
\begin{itemize}
\item solutions to the Whitham equation [...] for negative initial data $u_{0}$ of sufficiently large mass will develop a cusp at $t^{\ast} > t_{c}$ of the form $\abs{x-x^{\ast}}^{1/3}$ [...]
\item solutions to the Whitham equation [...] for positive initial data $u_{0}$ of sufficiently large mass will develop a cusp at $t^{\ast} > t_{c}$ of the form $\abs{x-x^{\ast}}^{1/2}$ [...]
\end{itemize}
\end{quote}
Our construction provides \emph{an open set of initial data of each sign} whose corresponding solutions all have the same blow-up behavior (i.e., $\calC^{\frac{1}{3}}$ remains bounded while $\calC^{\sgm}$ for any $\sgm > \frac{1}{3}$ blows up), thereby providing a counterexample to this conjecture as stated; see Remark~\ref{rem:id-sign} below. Nevertheless, it is possible that the blow-up observed in \cite{KlSa} for positive initial data is another stable blow-up regime, whose blow-up profile must have a positive sign. Verification of this revised picture remains an interesting open problem.
\end{remark}

\subsection{Prior works} \label{sub:history}
The models we consider in the present paper have been considered several times in the literature. We try to give a (non-exhaustive) list of previous results here, dividing them into four main areas.

\begin{itemize}
\item {\bf Water waves}. Some of the above equations, such as \eqref{eq:fkdv-0} and \eqref{eq:whitham} arise as approximated models in the theory of \emph{water waves}. In his 1967 paper Whitham~\cite{whitham} introduced the equation bearing his name, arising as a nonlinear approximation for surface water waves, where the dispersive term satisfies the appropriate dispersion relation. For further discussion of the connections of \eqref{eq:fkdv-0} with the theory of water waves, we refer the reader to the work of Klein--Linares--Pilod--Saut~\cite{klein2018}.

Many authors since the work of Whitham focused on the issue of singularity formation for such models. Wave breaking for \eqref{eq:whitham} was first shown only formally by Seliger~\cite{Sel}, followed by Naumkin--Shishmarev~\cite{NaSh}. The Russian authors were able to extend Seliger's argument to~\eqref{eq:fkdv-0} in the case $0 \leq \alpha < \frac{2}{3}$.  However, it appears that their arguments were not completely rigorous. In follow-up work, A.~Constantin--Escher~\cite{CoEs} made these arguments fully rigorous in the case of a model problem very similar to \eqref{eq:whitham}, requiring however boundedness of the kernel, which does not cover the case of \eqref{eq:whitham} itself.

In Castro--Cordoba--Gancedo~\cite{castro2010}, the authors proved blow-up for~\eqref{eq:fkdv-0} in the full range $0 < \alpha < 1$: their result show that the solution blows up in $\mathcal{C}^{1,\sgm}$, however it does not imply gradient blow-up in this case.

In Klein--Saut~\cite{KlSa}, the authors performed numerical experiments on \eqref{eq:fkdv-0} in the full range $0 < \alpha < 1$, which lead them to conjecture that wave breaking happens in the full range.

In Hur--Tao~\cite{HuTa}, the authors were then able to show wave breaking for~\eqref{eq:fkdv-0} in the case $0 < \alpha < \frac 12$ and for \eqref{eq:whitham}. In later work, Hur was able to extend the blow-up construction to the range $0 < \alpha < \frac 23$, see Hur~\cite{Hu2017}.

More recently, Yang~\cite{Yang2020} extended the shock formation construction to the case $\alpha = 0$, by means of a modulation-theoretic analysis in self-similar variables similar to \cite{BuShVi1} (discussed below), which gives a precise description of singularity formation. Finally, Saut--Wang~\cite{SautWang} have also proved gradient blow-up for \eqref{eq:fkdv-0} in the case $0 \leq \alpha < \frac 3 5$ as well as for \eqref{eq:whitham}.

Concerning model problems, let us mention the work of Klein--Saut--Wang~\cite{KleinSautWang}, where the authors consider the \emph{modified} fKdV equation (which features a cubic nonlinearity) in the range $\alpha \in (0,2)$. In the weakly dispersive range ($\alpha \in (0,1)$), they show the existence of wave breaking solutions. Note also that, by the work of Saut--Wang~\cite{SautWangMod}, modified fKdV admits global solutions for small data when $\alpha$ is in the full range $(0,2)$, $\alpha \neq 1$.

Let us finally briefly mention the case of $\alpha \geq 1$, where the situation seems to be delicate. Conjecturally, when $\alpha \in (1,3/2)$, the picture of ``shock formation'' is expected not to hold (see Klein--Saut~\cite{KlSa}) for the fKdV equation. In a recent work, Rimah was able to establish a precise version of this statement for a paralinearized version of the fKdV equation, thereby excluding wave breaking in the case $\alpha \in (1,2)$ for a paralinearized model problem~\cite{Rimah}.  For modified fKdV with $\alpha \in (1,2)$, Klein--Saut--Wang (again in~\cite{KleinSautWang}), conjecture that, in the focusing case for $\alpha \in (1,2)$, the $L^\infty$ norm of the solution blows up (and no wave breaking occurs).

\item {\bf Weak dissipation}. Weakly dissipative models have also attracted significant attention from the fluid dynamics community. For \eqref{eq:fburgers} in the case $0 \leq \beta < 1$, Kiselev--Nazarov--Shterenberg~\cite{KiNaSh} and independently Alibaud--Droniou--Vovelle~\cite{AlDrVo} as well as Dong--Du--Li~\cite{DoDuLi} were able to show gradient blow-up. Very recently, Chickering--Moreno-Vasquez--Pandya \cite{CMVP} used an approach similar to \cite{BuShVi1, Yang2020} to give a precise description of stable blow-up dynamics in the case $0 \leq \bt < \frac{2}{3}$.

We note that the blow-up solutions to \eqref{eq:fburgers} constructed in this article sharply complement known regularity criteria for \eqref{eq:fburgers}. More precisely, regularity results on linear advection-fractional dissipation equation \cite{CoWu, Sil1, Sil2} (see also \cite{Ibdah} for time-integrated criteria) imply that if $u$ is a solution to \eqref{eq:fburgers} such that $u \in L^{\infty}_{t}([0, \tau_{+}); \calC^{(1-\bt)+\eps})$ for some $\eps > 0$, then $\rd_{x} u$ is H\"older continuous up to time $\tau_{+}$, and therefore the solution extends past $[0, \tau_{+})$. On the other hand, for each $k \geq 1$, Theorem~\ref{thm:main-simple} demonstrates the existence of a blow-up solution $u$ to \eqref{eq:fburgers} with $u \in L^{\infty}_{t} ([0, \tau_{+}); \calC^{\frac{1}{2k+1}})$ for any $\bt < \frac{2k}{2k+1}$ (or more instructively, $\frac{1}{2k+1} < 1-\bt$).

\item {\bf Self-similar constructions in fluids}. 
Our blow-up construction is based on the method of modulation theory in self-similar variables using smooth self-similar solutions to the Burgers equation as blow-up profiles. This method was recently applied to compressible fluid dynamics with great success in a series of work \cite{BuShVi1, BuShVi11, BuShVi2} by Buckmaster--Shkoller--Vicol. In \cite{BuShVi1, BuShVi11}, the authors use a self-similar method to show shock formation for polytropic compressible Euler in two and three space dimensions, giving a precise asymptotic description of shock formation at the point of first singularity, even in the presence of vorticity. They moreover extended their treatment to the non-isentropic case in~\cite{BuShVi2}, showing for the first time generation of vorticity at the shock. 
 
We also mention the work of Buckmaster--Iyer~\cite{BuIy}, in which the authors show formation of unstable shocks for two-dimensional polytropic compressible Euler by using (first) excited states as blow-up profiles, albeit via a different argument (Newton iteration) than what is used in this article (topological argument) to control the unstable directions.

Another instance of this method can be found in the interesting work of Collot--Ghoul--Masmoudi~\cite{Collot2018}, in which the authors construct gradient blow-up for a two-dimensional Burgers equation with transverse viscosity, which is a simplified model for Prandtl's boundary layer equation. In particular, similarly to the present article, \cite{Collot2018} employs weighted $L^{2}$-bounds (albeit nonsharp)  and makes use of all excited states as blow-up profiles via a topological argument. 

Concerning self-similar solutions in fluids, we finally mention the groundbreaking recent work of Elgindi on the blow-up of the 3D incompressible Euler equations in the $\mathcal{C}^{1,\alpha }$ regularity class~\cite{Elgindi}.

\item {\bf Geometric blow-up constructions.} Finally, we also mention the geometric blow-up constructions pioneered by Christodoulou in \cite{christodoulou1}, where shock formation for the compressible irrotational Euler equations is shown. The work of Christodoulou relies on powerful energy estimates, which allow not only to construct the point of first singularity, but also the maximal development of the solution. These ideas enabled Christodolou to later address the restricted shock development problem \cite{christodoulou2}. Moreover, Luk--Speck used geometric ideas to show stability of planar shocks under perturbations with nonzero vorticity~\cite{LuSp}. In the context of the present work, it would be very interesting to extend this type of reasoning to include weakly dispersive and dissipative effects.

\end{itemize}

\subsection{Strategy of the proof} \label{subsec:strategy}

In this section, we outline the strategy of the proof. For the purposes of this section, let us restrict to the case of~\eqref{eq:fkdv-0}. Our argument is based on the underlying analysis of stable (and unstable) blow-up for the Burgers equation. It is well known (see, for instance, \cite{EgFo}) that, for any given $k \in \N$, $k \geq 1$, the Burgers equation admits self-similar solutions exhibiting blow-up in $\calC^{0, \frac{1}{2k+1}+}$, which are each associated to a self-similar blow-up profile and self-similar coordinates.

We start from equation~\eqref{eq:fkdv-0}, which is written in the variables $(t, x, u)$, and we rewrite it in the appropriate self-similar variables arising from Burgers, which we call $(s, y, U)$. For the precise definition of these variables, see Section~\ref{sec:derivation}. We expect the unstable behavior to be encoded by the derivatives of $U$ up to and including order $2k$ at $y = 0$. In view of this observation, we are going to track of the values of $\p_y^j U(s,0)$ throughout the evolution, $j = 0, \ldots, 2k$.

Three of the unstable modes can  be controlled naturally by modulation parameters adapted to the symmetries of the equation: time translation, space translation and Galilean transformation. The modulation conditions will therefore be imposed on $U(s,0)$, $\p_y U(s,0)$ and $\p_y^{2k}U(s,0)$, and the modulation parameters are going to be called $(\tau, \xi, \kappa)$. In the case $k=1$, there are only three unstable directions, which allows us to show stable blow-up.

In the case $k > 1$, the remaining $2k - 2$ unstable directions will have to be controlled by  selecting the initial data appropriately.  For the precise definition of the modulation parameters and to see how they arise from the symmetries of the equation, see Section~\ref{sec:derivation}. 

With this setup, in the self-similar variables, \eqref{eq:fkdv-0} then becomes:
\begin{equation}
\begin{aligned}
	&\rd_{s} U + b y \rd_{y} U + \left(- e^{b s} \xi_{s} + (1+e^{s} \tau_{s}) e^{(b-1) s} \kpp \right) \rd_{y} U - \left(b - 1 \right) U + e^{(b-1)s} \kpp_{s} 
	+ (1+e^{s} \tau_{s}) U \rd_{y} U \\
	& = (1+e^{s} \tau_{s})e^{-s(1- b\alpha)}\p_y |D_y|^{\alpha-1} U. 
\end{aligned}
\end{equation}
Here, we have defined $b = \frac{2k + 1}{2k}$.

For ease of exposition, we are now going to set all the modulation parameters to zero. We obtain the following equation:
\begin{equation}\label{eq:ssfkdvintro}
    \p_s U - (b-1) U + (U+by)\p_y U = e^{-s(1- b\alpha)}\p_y |D_y|^{\alpha-1} U.
\end{equation}

The key observation is that, as long as $1 - b \alpha > 0$, we are able to treat the term on the RHS in a perturbative way due to the exponentially decaying prefactor. Since $b \to 1$ as $k \to \infty$, we are able to treat values of $\alpha$ arbitrarily close to $1$ by choosing $k$ appropriately.

We will set up a bootstrap argument and our goal will be to show that equation~\eqref{eq:ssfkdvintro} admits a global (in self-similar time $s$) solution. The starting point is  that $\p_y U$ can be treated almost independently by a Lagrangian analysis, which yields a uniform bound for $\p_y U$ in $L^\infty$. Through the intermediate step of showing a uniform $L^2$ bound for $\p_y^2 U$, we finally propagate appropriate weighted (in $y$) bounds for top-order derivatives. The weights are adapted to the Lagrangian flow of the equation, and their purpose is to show that the solution displays the correct asymptotic behavior at the time of blow-up. This is carried out in a weighted $L^2$ framework, which has a twofold advantage. First, we show blow-up without the need to consider the difference with the exact self-similar blow-up profile, which is an amusing aspect by itself. Second, this part of the argument is entirely $L^2$ based, which avoids derivative loss at the top order.

The final part of the argument is then devoted to addressing the ``unstable'' part, i.e. the ODE analysis for the modulation parameters and for the unstable derivatives of $U$ at $y = 0$. We introduce a ``trapping condition'' for the unstable coefficients (i.e., a decay condition on derivatives of $U$ at $y = 0$) and show, by way of a shooting argument, that initial data can be selected such that the trapping condition holds for all times.

We are now going to describe the strategy in more detail, again focusing on the case of fKdV.

\begin{enumerate}
\item \emph{Control of $\| \p_y U \|_{L^\infty}$}. We differentiate equation~\eqref{eq:ssfkdvintro} by $\p_y$ and we obtain:
\begin{equation}\label{eq:uprimeintro}
    \p_s U' +U' +(U')^2 +(U+by)\p_y U' =e^{-s(1- b\alpha)}\p_y |D_y|^{\alpha -1} U'.
\end{equation}
Let us for a moment neglect the nonlocal RHS. We rewrite $U'$ in Lagrangian coordinates (we let $\tilde U'$  be $U'$ written in Lagrangian coordinates) and we obtain the following equation for $\tilde U'$:
$$
\p_s \tilde U' = -\tilde U' (\tilde U' +1).
$$
We immediately see that the inequality $-1 < \tilde U < 1$ is preserved by the above Lagrangian ODE, and moreover this bound carries over to the original equation~\eqref{eq:uprimeintro}. This control is going to be the starting point of our analysis (see Lemma~\ref{lem:U1}). 

Building upon this inequality, we then show that, depending on the region considered, $U'(s,y)$ either satisfies a coarse polynomial bound in terms of $|y|$, or it decays exponentially in self-similar time (see Lemma~\ref{lem:U1}, part 4). We will use this, later in the course of the argument, to show dissipativity of the equation in a region where $y$ is large.

\item \emph{Control of $\| \p^{2k+2}_y U \|_{L^2(\R)}$ and $\| \p^{2k+3}_y U \|_{L^2(\R)}$}. These terms are ``top order'' in terms of derivatives. In this part, we shall first accomplish the intermediate task of controlling $\| \p^2_y U \|_{L^2(\R)}$. We first show, by a Lagrangian argument, that $U''$ satisfies a uniform $L^\infty$ bound in the ``close'' region $|y|\leq \frac{1}{4}$. We emphasize that, in this case, it is extremely important that the bound, as well as its region of validity, be independent of the bootstrap parameters. The reason is that we then perform a weighed $L^2$ estimate in the region $R:= \{ y \in \R: \frac{1}{4}\leq |y| \leq y_2\}$, and we wish to control the expression 
$$
\nrmb{\exp\left(- \frac{\lambda}{2} y\right) U''}_{L^2(R)},
$$
where $\lambda$ is a positive real parameter, and $y_2$ is a positive number which we regard as large. We thus require the parameter $\lambda$ to depend on the lower bound for $|y|$ when $y \in R$, and it is therefore crucial that this lower bound be independent of the bootstrap parameters. Finally, we need to show a bound in the ``far away'' region, where $|y| \geq y_2$. We use the smallness of $U'$  to show that the equation for $U''$ has a dissipative character for $|y| \geq y_2$. Combining the three  regions, we obtain a uniform $L^2$ bound for $U''$. This is the content of Lemma~\ref{lem:uii}.

Turning now to the proof of the bounds for $\| \p^{2k+2}_y U \|_{L^2(\R)}$ and $\| \p^{2k+3}_y U \|_{L^2(\R)}$, we recall the familiar observation that, taking $2k+2$ derivatives of equation~\eqref{eq:ssfkdvintro}, the linear term on the LHS becomes dissipative everywhere on $\R$. Combining this fact with interpolation and the control of $\p^2_y U$ in $L^2$, which was obtained as an intermediate step, allows us to deduce a uniform bound for $\| \p^{2k+2}_y U \|_{L^2(\R)}$. Using this control, it is then straightforward to derive a bound for $\| \p^{2k+3}_y U \|_{L^2(\R)}$ (we require control up to this order due to a technical point: we will need to bound $\p_y^{2k+1}U$ at $y = 0$ by Sobolev embedding, in order to control the evolution of $\p_y^{2k+1}U$ at $y = 0$ in the ``unstable'' part of the argument). The high order bounds are obtained in Lemma~\ref{lem:high}.

\item \emph{Control of weighted $L^2$ norms}. Recall that the exact self-similar profile $\bar U$ for the Burgers equation satisfies, for large $|y|$, 
\begin{equation}\label{eq:profiledecay}
|y|^{\frac 1 {2k+1} - j}\lesssim |\p_y^j \bar U| \lesssim |y|^{\frac 1 {2k+1} - j},
\end{equation}
where $j \geq 0$, $j \in \N$. 

In this part of the argument, we wish to propagate an appropriate $L^2$ version of the above polynomial decay bound, using a weighted $L^2$ space, for top-order number of derivatives (i.e., when $j = 2k+3$). This information is needed to show that the blow-up solution is in the correct H\"older regularity class up to the blow-up time.

The weights are constructed such that, in a region of bounded $x$ (i.~e.~a region which corresponds to the image under the Lagrangian flow of a bounded $y$-interval centered at $0$), one obtains the corresponding decay in $y$. Outside this region, the weight is ``tapered'': it is independent of $y$, and grows exponentially in self-similar time at the correct rate.

More precisely, given $n \in \N$ and $L > 0$, we define the semi-norm
\begin{equation}
\nrm{V}_{\dot{\calH}^{n}_{< L}} = \sup_{j \in \bbZ, \, 2^{j} < L} \left( \int_{2^{j-1} < \abs{y} < 2^{j}} (\abs{y}^{n-\frac{1}{2k+1}} \rd_{y}^{n} V)^{2} \frac{\ud y}{y}\right)^{\frac{1}{2}} + L^{n-\frac{1}{2k+1}-\frac{1}{2}} \left( \int_{\abs{y} >\frac{L}{2}} (\rd_{y}^{n} V)^{2} \ud y\right)^{\frac{1}{2}}.
\end{equation} 

Note that it consists of two terms: each expression in the first summand scales according to~\eqref{eq:profiledecay}, and the second summand is obtained by choosing the weight to be the matching constant outside the region $|y| \leq \frac L 2$. In practice, since the Lagrangian flow away from $y = 0$ is well approximated by $y = C e^{bs}$, we are going to set $L = e^{bs}$.

Our goal will then be to show that $\nrm{U}_{\dot{\calH}^{2k+3}_{< e^{bs}}}$ is uniformly bounded in $s$. As a first step, we first show a uniform bound on $\nrm{U}_{\dot{\calH}^{1}_{< e^{bs}}}$.  To obtain it, we multiply the equation for $U'$ by a weight approximately adapted to the Lagrangian flow in a region of large $y$. The growth rate of the weight is also chosen appropriately.

Using this information, the bound for $\nrm{U}_{\dot{\calH}^{2k+3}_{< e^{bs}}}$ is obtained in a similar fashion. In this case, however, one needs to be careful about a potential loss of derivative, as the nonlocal term does not commute with the weight. To deal with this issue, we show a commutator estimate (see Lemma~\ref{lem:mult-w}). The weighted bounds are proved in Lemma~\ref{lem:weighted2}.

\begin{remark}
Note that, if we set $L = \infty$, the above semi-norm is \emph{scale invariant}.
\end{remark}

\item \emph{Topological argument}. Finally, in Section~\ref{sec:unstable}, we employ a topological procedure relying on the instability of the ODE system satisfied by the Taylor coefficients of $U$ at $y = 0$ to close the argument. This procedure will moreover select appropriately the initial data in the unstable case. This type of construction is well known in the dispersive community: see, for instance, the paper by C\^ote, Martel, and Merle \cite{CoMaMe}.

Recall the trapping condition, i.e.~a decay condition for the ``unstable'' Taylor coefficients of $U$ at $y = 0$. We want to show that, upon appropriately choosing initial data, it can be arranged that the solution remains trapped globally in time.

First, in Lemma~\ref{lem:modulation} we show that, under the bootstrap assumptions and assuming the trapping condition, the evolution of the modulation parameters is controlled.

Finally, in Lemma~\ref{lem:outgoing}, we show that the ODE system satisfied by the first $2k+1$ Taylor coefficients of $U$ at $y=0$ displays an unstable character, and we use this fact, combined with a Brouwer-type argument, to show that we can select initial data such that the corresponding solution is trapped for all time. This concludes the argument.
\end{enumerate}

\begin{remark}
Note that parts 1.~and 2.~of the above outline rely on showing $L^\infty$ estimates, which are proved here by means of Lagrangian analysis. This Lagrangian approach seems to be the most efficient way (in terms of degree of technicality) to analyze directly the unknown $U$ (which is what we do in this paper), rather than the difference between $U$ and the exact self-similar profile. However, we believe that, if instead one were to analyze the difference between $U$ and the corresponding self-similar profile, one would be able to carry out the argument without the need for Lagrangian analysis. This approach would make the argument completely $L^2$ based.
\end{remark}

\subsection{Organization of the paper}

In Section~\ref{sec:prel}, we introduce the relevant equations, the self-similar coordinates, the modulation parameters, and the unstable ODE system for Taylor coefficients at $y=0$. In Section~\ref{sec:mainbst}, we give a precise statement of the main theorem (Theorem~\ref{thm:main}), and reduce its proof to establishing two key lemmas, Lemma~\ref{lem:main} (main bootstrap lemma) and Lemma~\ref{lem:top} (shooting lemma for unstable coefficients when $k > 1$). After collecting some useful lemmas for the Fourier multipliers arising in our problem in Section~\ref{sec:fourier}, the following two sections are devoted to the proof of the two key lemmas. In Section~\ref{sec:Uest}, close the bootstrap assumptions on the solution $U$ in appropriate self-similar variables. In Section~\ref{sec:unstable}, we estimate the ODEs for the modulation parameters and stable coefficients, thereby completing the proof of Lemma~\ref{lem:main}. Moreover, in case $k > 1$, we analyze the ODEs for the unstable coefficients and establish Lemma~\ref{lem:top}.

\subsection*{Acknowledgements}
F. Pasqualotto would like to acknowledge Tristan Buckmaster and Javier G\'omez-Serrano for insightful discussions on blow-up constructions. This material is based partially upon work supported by the National Science Foundation under Grant No.~DMS-1928930 while the authors participated in a program hosted by the Mathematical Sciences Research Institute in Berkeley, California, during the Spring 2021 semester. S.-J.~Oh was partially supported by the Samsung Science and Technology Foundation under Project Number SSTF-BA1702-02, a Sloan Research Fellowship and a National Science Foundation CAREER Grant under NSF-DMS-1945615.

\section{Preliminaries}\label{sec:prel}
\subsection{Notation and conventions}
As is usual, we use $C > 0$ to denote a positive constant that may change from line to line. Dependencies of $C$ are expressed by subscripts. Moreover, we use the standard notation $A \aleq B$ for $\abs{A} \leq C B$, and $A \aeq B$ for $A \aleq B$ and $B \aleq A$, and dependencies of the implicit constant $C$ are expressed by subscripts.

Given a symbol $\Gmm(\xi)$, we denote by $\Gmm(D_{x})$ its quantization in $x$, i.e., $\Gmm(D_{x}) V = \calF_{x}^{-1}[\Gmm(\xi)\calF_{x}[V](\xi)]$, where $\calF_{x}$ denotes the Fourier transform in the variable $x$. For each $k \in \bbZ$, we define the \emph{Littlewood--Paley projection} $P_{\leq k}$ to be the Fourier multiplier operator with symbol $P_{\leq k}(\xi)$ defined by $P_{\leq k}(\xi) = P_{\leq 0}(2^{-k} \xi)$, where $P_{\leq 0}$ is a nonnegative smooth function supported in $[-2, 2]$ and equals $1$ on $[-1, 1]$. We also introduce the symbols $P_{k}(\xi) = P_{\leq k}(\xi) - P_{\leq k-1}(\xi)$ and $P_{> k}(\xi) = 1 - P_{\leq k}(\xi)$, as well as the corresponding Fourier multipliers (which are also called Littlewood--Paley projections).

\subsection{Derivation of the equations in self-similar variables}\label{sec:derivation}

Given parameters $\tau, \xi, \kpp \in \bbR$ and $\lmb > 0$ (called \emph{modulation parameters}), consider the change of variables $(t, x, u) \mapsto (s, y, U)$,
\begin{equation} \label{eq:ss-var}
	t = \tau - e^{-s}, \qquad x = \lmb y + \xi, \qquad u(t, x) = \frac{\lmb}{ \tau - t} U + \kpp.
\end{equation}
Note that varying the modulation parameters $\tau$, $\xi$, $\kpp$ and $\lmb$ correspond to applying time translation, space translation, Galilean boost ($u \mapsto u(x-\kpp t) - \kpp$) and spatial scaling to the solution, which are exact symmetries for the (invsicid) Burgers equation. Hence, $(s, y, U)$ are nothing but the rescaled variables for Burgers equation centered at $(t, x, u) = (\tau, \xi, \kpp)$ at spatial scale $\lmb$. 

We let the modulation parameters depend dynamically on $t$, i.e., $\tau = \tau(t)$, $\xi = \xi(t)$, $\kpp = \kpp(t)$ and $\lmb = \lmb(t)$, and consider the same change of variables \eqref{eq:ss-var}. Note that
\begin{align*}
\rd_{t} s &= e^{s} (1 + e^{s} \tau_{s})^{-1}, & 
\rd_{t} y &= - e^{s} (1 + e^{s} \tau_{s})^{-1} \left(\frac{\lmb_{s}}{\lmb} y + \frac{\xi_{s}}{\lmb}\right), &
\rd_{x} s &= 0, &
\rd_{x} y &= \frac{1}{\lmb},
\end{align*}
so that
\begin{align*}
	\rd_{t} u
	&= \lmb e^{2 s} (1+e^{s} \tau_{s})^{-1} \left(\rd_{s} U - \frac{\lmb_{s}}{\lmb} y \rd_{y} U - \frac{\xi_{s}}{\lmb} \rd_{y} U + \left(\frac{\lmb_{s}}{\lmb} + 1 \right) U + \frac{e^{-s}}{\lmb} \kpp_{s} \right), \\
	u \rd_{x} u
	&= \lmb e^{2s} \left(\left(U + \frac{e^{-s}}{\lmb} \kpp \right) \rd_{y} U\right), \\
	\Gmm (D_{x}) \rd_{x} u &= \lmb e^{s} \Gmm(\lmb^{-1} D_{y}) \lmb^{-1} \rd_{y} \left( U + \lmb^{-1} e^{-s} \kpp \right), \\
	\Ups (D_{x}) u &= \lmb e^{s} \Ups(\lmb^{-1} D_{y}) \left(U + \lmb^{-1} e^{-s} \kpp\right).
\end{align*}
Thus, \eqref{eq:fkdv} becomes
\begin{equation*}
\begin{aligned}
	&\rd_{s} U - \frac{\lmb_{s}}{\lmb} y \rd_{y} U + \left(- \frac{\xi_{s}}{\lmb} + (1+e^{s} \tau_{s}) \frac{e^{-s}}{\lmb} \kpp\right) \rd_{y} U + \left(\frac{\lmb_{s}}{\lmb} + 1 \right) U + \frac{e^{-s}}{\lmb} \kpp_{s} \\
	&+ (1+e^{s} \tau_{s}) \left(U \rd_{y} U + e^{-s} \Gmm(\lmb^{-1} D_{y}) \lmb^{-1} \rd_{y} \left(U + \lmb^{-1} e^{-s} \kpp\right) + e^{-s} \Ups(\lmb^{-1} D_{y}) \left(U + \lmb^{-1} e^{-s} \kpp\right)  \right)= 0.
\end{aligned}
\end{equation*}
In what follows, we shall assume the following \emph{self-similarity ansatz} for $\lmb$: Given $k \in \N$, set $b : = \frac{2k + 1}{2k}$ and
\begin{equation} \label{eq:scales}
	\lmb = (\tau - t)^{b} = e^{- b s}.
\end{equation}
Then $\frac{\lmb_{s}}{\lmb} = - b$, and we arrive at
\begin{equation} \label{eq:ssfkdv-pre}
\begin{aligned}
	&\rd_{s} U + b y \rd_{y} U + \left(- e^{b s} \xi_{s} + (1+e^{s} \tau_{s}) e^{(b-1) s} \kpp \right) \rd_{y} U - \left(b - 1 \right) U + e^{(b-1)s} \kpp_{s} \\
	&+ (1+e^{s} \tau_{s}) \left(U \rd_{y} U + e^{-s} \Gmm(e^{b s} D_{y}) e^{b s} \rd_{y} \left(U + e^{(b-1)s} \kpp\right) + e^{-s} \Ups(e^{b s} D_{y}) \left(U + e^{(b-1)s} \kpp\right)  \right)= 0.
\end{aligned}
\end{equation}
If $\Gmm$ and $\Ups$ were zero, and $\tau$, $\xi$, $\kpp$ were constant, then \eqref{eq:ssfkdv-pre} is precisely the self-similar Burgers equation with scales \eqref{eq:scales}. As is well-known, the values $b = \frac{2k+1}{2k}$ with $k = 1, 2, \ldots$ are distinguished by the property that they admit \emph{smooth} steady profiles of the self-similar Burgers equation \cite[Section~11.2]{EgFo}; see also Subsection~\ref{sub:profiledef} below. 

Our intention is to view the linear terms $e^{-s}\Gmm(e^{bs} D_{y}) e^{bs} \rd_{y} U$ and $e^{-s}\Ups(e^{bs} D_{y}) U$ as perturbations. To motivate the way we will decompose these terms, consider the model cases $\Gmm(\xi) = c_{\Gmm} \abs{\xi}^{\alp-1}$ and $\Ups(\xi) = c_{\Ups} \abs{\xi}^{\bt}$ ($c_{\Gmm}, c_{\Ups} \in \bbR$). Then $\Gmm(e^{bs} D_{y}) e^{b s} \rd_{y} = c_{\Gmm} e^{b \alp s} \abs{D_{y}}^{\alp-1} \rd_{y}$ and $\Ups(e^{bs} D_{y}) = c_{\Ups} e^{b \bt s} \abs{D_{y}}^{\bt}$, so that
\begin{equation*}
	e^{-s}\Gmm(e^{bs} D_{y}) e^{bs} \rd_{y} U + e^{-s}\Ups(e^{bs} D_{y}) U
	= c_{\Gmm} e^{-(1- b \alp) s} \abs{D_{y}}^{\alp-1} \rd_{y} U
	+ c_{\Ups} e^{-(1- b \bt) s} \abs{D_{y}}^{\bt}U.
\end{equation*}
In the regime we perform our construction, $\abs{D_{y}}^{\alp-1} \rd_{y} U$ and $\abs{D_{y}}^{\bt}U$ will morally remain bounded in time\footnote{More precisely, we will have boundedness of the gradient of these terms, and controlled growth for the terms themselves.}. Therefore, we may regard these terms as perturbative when $b \alp < 1$ and $b \bt < 1$, in which case the factors $e^{-(1- b \alp) s} $ and $e^{-(1- b \bt) s} $ decay exponentially.

In view of the above discussion, in what follows, we are going to denote
\begin{equation}\label{eq:mudef}
\mu = \min\{1 - b \alpha, 1-b\beta, 1\}.
\end{equation}
Note that, under our assumptions, $\mu > 0$. To simplify our notation, we will now rewrite the operator on the RHS as follows:
\begin{equation}
e^{-s} \Gmm(e^{b s} D_{y}) e^{b s} \rd_{y} U + e^{-s} \Ups(e^{b s} D_{y}) (U + e^{(b-1)s}\kpp )= - e^{-\mu s} \calH \left(U + e^{(b-1)s} \kpp\right) - e^{-s} \calL\left(U + e^{(b-1)s} \kpp\right),
\end{equation}
where
\begin{align}
\calH(V) &= - P_{> 0}(e^{b s} D_{y}) \left( e^{-\max\set{\alp, \bt, 0} b s} \Gmm(e^{bs} D_{y}) e^{bs} \rd_{y} V  + e^{- \max\set{\alp, \bt, 0} b s} \Ups(e^{bs} D_{y}) V \right), \label{eq:H-def} \\
\calL(V) &= - P_{\leq 0}(e^{b s} D_{y}) \left( \Gmm(e^{bs} D_{y}) e^{bs} \rd_{y} V + \Ups(e^{bs} D_{y}) V \right). \label{eq:L-def}
\end{align}

Note that $\calH\left(U + e^{(b-1)s} \kpp\right) = \calH(U)$ thanks to $\chi_{\geq 1}(e^{bs} D_{y})$. Putting everything together, we finally have
\begin{equation}\label{eq:ssfkdv}
\boxed{
\begin{aligned}
	&\rd_{s} U + b y \rd_{y} U + \left(- e^{b s} \xi_{s} + (1+e^{s} \tau_{s}) e^{(b-1) s} \kpp \right) \rd_{y} U - \left(b - 1 \right) U + e^{(b-1)s} \kpp_{s} 
	+ (1+e^{s} \tau_{s}) U \rd_{y} U \\
	& = (1+e^{s} \tau_{s}) \left( e^{-\mu s} \calH (U) + e^{-s} \calL\left(U + e^{(b-1)s} \kpp\right) \right). 
\end{aligned}
}
\end{equation}

\subsection{Definition of the profile}\label{sub:profiledef}

We now solve the steady profile equation for the Burgers problem (i.e., $\Gmm$, $\Ups$ are zero and $\tau$, $\xi$ and $\kpp$ are fixed):
\begin{equation}\label{eq:steady}
(1-b) \mathring{U}+(b  y + \mathring{U})\p_y \mathring{U}= 0.
\end{equation}
We define $\mathring U$ to be a solution to the above equation (an exact self-similar profile) such that $\mathring U(0) = 0$, $\rd_{y} \mathring U(0) = -1$, $\rd_{y}^{2k+1} \mathring U(0) = (2k)!$ and $\p_y^j \mathring U (0) = 0$ for $j = 2, \ldots, 2k$. We can ensure the last condition by simply noticing the the self-similar profile equation is equivalent to:
$$
y = - \mathring U - h_1\mathring U^{2k + 1},
$$
where $h_1 > 0$ is a free parameter. From this implicit definition, we see that the first three nonvanishing Taylor coefficients at $y=0$ are $\rd_{y} \mathring U(0)$, $\rd_{y}^{2k+1} \mathring U(0)$ and $\rd_{y}^{4k+1} \mathring U(0)$. We fix $h_1$ so that $\mathring U$ satisfies $\rd_{y}^{2k+1} \mathring U(0) = (2k)!$; in what follows, we will suppress the dependence of constants on $h_{1}$.

By construction, $\mathring U$ has the following Taylor expansion about $y = 0$:
\begin{equation}\label{eq:tchoice}
\mathring U(y) = - y + \frac{1}{2k+1} y^{2k+1} + O(y^{4k+1})
\end{equation}
and the following expansion about $y = \pm \infty$:
\begin{equation}\label{eq:chi-exp-infty}
\mathring U(y) = \mp h_{1}^{-\frac{1}{2k+1}} \abs{y}^{\frac{1}{2k+1}} \left(1 + O(\abs{y}^{-\frac{2k}{2k+1}}) \right).
\end{equation}
We now define our choice of the profile. Consider a function $\bar \chi: \R \to \R$ which is positive, equal to $1$ on the interval $[-1,1]$, equal to $0$ outside of the interval $[-8,8]$, and such that $\bar \chi' \geq - \frac 14$. We then define the cut-off function $\chi$ to be transported by the linearized flow generated by $\mathring U$:
\begin{equation}\label{eq:chidef}
\p_s \chi + (by + \mathring U) \p_y \chi = 0, \qquad \chi(0, y) = \bar \chi(y).
\end{equation}

Some basic properties of the cut-off function $\chi$ are as follows:
\begin{lemma}[Support property of $\chi(s,y)$]\label{lem:chi}
We have $\supp \chi \subseteq [-C e^{b s}, C e^{b s}]$.
\end{lemma}
\begin{proof}[Proof of Lemma~\ref{lem:chi}]
Define the Lagrangian trajectories $Y_{\pm}(s)$ in the following way:
\begin{equation}\label{eq:ypm}
\p_s Y_{\pm}(s) = b Y_{\pm}(s) + \mathring U(Y_{\pm}(s)), \qquad Y_\pm (0) = \pm 10,
\end{equation}
so that $\chi(y) = 0$ for all $y: |y| \geq Y(0)$. The conclusion of the lemma will follow if we can show that $|Y_{\pm}| \leq C e^{bs}$. By the form of $\mathring U$, we have that $|\mathring U(Y_{\pm}(s))| \leq 1 + |Y_{\pm}(s)|$, and the lemma follows by integrating~\eqref{eq:ypm} in $s$.
\end{proof}
By Lemma~\ref{lem:chi} and \eqref{eq:chi-exp-infty}, we have
\begin{equation} \label{eq:profile-Linfty}
	\sup_{y \in \supp \chi(s, \cdot)} \abs{\mathring{U}(y)} \aleq e^{\frac{1}{2k+1} b s} = e^{(b-1) s}.
\end{equation}
We finally define the profile
\begin{equation}\label{eq:profdef}
\bar U(s,y) = \chi(s,y) \mathring U (y).
\end{equation}
This is no longer a time-independent profile. Moreover, the modified profile satisfies the equation
\begin{equation}\label{eq:cutoffp}
\p_s \bar U - (b-1) \bar U + (b y + \bar U) \p_y \bar U = - \mathring{U} (1 - \chi)  \rd_{y} \bar{U}.
\end{equation}

\subsection{Equation for iterated derivatives of \texorpdfstring{$U$}{U}}

We let $U^{(j)} = \p_y^{(j)} U$, with $j \geq 1$. We derive the equation satisfied by $U^{(j)}$. For $j = 1$,
\begin{equation}\label{eq:commutu-1}
\boxed{
\begin{aligned}
	&\rd_{s} U' +U'+\left( (1 + e^{s} \tau_{s}) U + b y - e^{b s} \xi_{s} + (1+e^{s} \tau_{s}) e^{(b-1) s} \kpp \right) \rd_{y} U'
	+ (1 + e^{s} \tau_{s}) (U')^{2} \\
	& = (1+e^{s} \tau_{s}) \left( e^{-\mu s} \calH (U') + e^{-s} \rd_{y} \calL(U + e^{(b-1)s} \kpp) \right),
\end{aligned}
}
\end{equation}

and for $j \geq 2$,
\begin{equation}\label{eq:commutu}
\boxed{
\begin{aligned}
	&\rd_{s} U^{(j)} + \left( (1 + e^{s} \tau_{s}) U + b y - e^{b s} \xi_{s} + (1+e^{s} \tau_{s}) e^{(b-1) s} \kpp \right) U^{(j+1)} \\
	&+ \left( 1+ (j-1) b 
	+ (j+1)(1 + e^{s} \tau_{s}) U' \right) U^{(j)}
	+ (1 + e^{s} \tau_{s}) M^{(j)} \\
	& = (1+e^{s} \tau_{s}) \left( e^{-\mu s} \calH (U^{(j)}) + e^{-s} \rd_{y}^{j} \calL(U + e^{(b-1)s} \kpp) \right). 
\end{aligned}
}
\end{equation}
Here, $M^{(j)} = \p_y^{(j)} (U U') - U U^{(j+1)} - ( j+1) U' U^{(j)}$ for $j \geq 2$.

\subsection{Perturbation equation and commutation}
We now define the perturbation $W = U - \bar U$, and we obtain the following equation for $W$:
\begin{equation} \label{eq:perturb}
\boxed{
\begin{aligned}
	&\rd_{s} W  + (b y + \br{U} + W) \rd_{y} W - \left(b - 1 - \rd_{y} \bar{U} \right) W \\
	&+ \left(- e^{b s} \xi_{s} + (1+e^{s} \tau_{s}) e^{(b-1) s} \kpp \right) \rd_{y} (\bar{U} + W) + e^{(b-1)s} \kpp_{s} + \frac{1}{2} e^{s} \tau_{s} \rd_{y} (\bar{U} + W)^{2} \\
	& = E_{\chi} + (1+e^{s} \tau_{s}) \left( e^{-\mu s} \calH (U) + e^{-s} \calL\left(U + e^{(b-1)s} \kpp\right) \right). 
\end{aligned}}
\end{equation} 
Here, $E_{\chi} = \mathring{U} (1 - \chi) \rd_{y} \bar{U}$.

\begin{remark}
Note that the error term $E_{\chi}$ arising from the cutoff is identically zero near $y =0$. 
\end{remark}

 Suppose now that $j \geq 1$. We now commute the above equation~\eqref{eq:perturb} with $\p_y^j$, and obtain:
\begin{equation}\label{eq:commut}
\boxed{
\begin{aligned}
	&\rd_{s} W^{(j)} + \left(b y + \bar{U} + W \right) W^{(j+1)} 
	+ \left( 1+ (j-1) b + (j+1) \bar{U}' \right) W^{(j)}
	+ N(j) + L(j) \\
	&+ \left(- e^{b s} \xi_{s} + (1+e^{s} \tau_{s}) e^{(b-1) s} \kpp \right) (\bar{U}^{(j+1)} + W^{(j+1)}) + \frac{1}{2} e^{s} \tau_{s}  \rd_{y}^{(j+1)}(\bar{U} + W)^{2}  \\
	& = E_{\chi}^{(j)} + (1+e^{s} \tau_{s}) \left( e^{-\mu s} \calH (U^{(j)}) + e^{-s} \rd_{y}^{j} \calL(U + e^{(b-1)s} \kpp) \right). 
\end{aligned}
}
\end{equation}
Here,
\begin{align*}
N{(j)} &= \p_y^{(j)} (W W') - W W^{(j+1)}, \\
L{(j)} & = (\bar U W)^{(j+1)} - \bar U W^{(j+1)} - (j+1) \bar U' W^{(j)}.
\end{align*}

\subsection{Derivation of the modulation equations and unstable ODE system at \texorpdfstring{$y = 0$}{y = 0}}

We now derive the equations satisfied by the derivatives of $W$ at the origin. For each $j \geq 0$, we let  
\begin{align*}
w_{j} := W^{(j)}(s, 0), \qquad
F^{(j)}(s,0) := e^{-\mu s} \calH \left( U^{(j)} \right) + e^{-s} \rd_{y}^{j} \calL \left(U +e^{(b-1)s} \kpp\right),
\end{align*}
For $j \geq 0$, we have
\begin{equation} \label{eq:modul-gen}
\begin{aligned}
	&\rd_{s} w_{j} 
	+ \left( (j-1) (b-1) - 1 \right) w_{j}
	+ w_{0} w_{j+1}
	+ \left. N(j) \right|_{y = 0} + \left. L(j) \right|_{y = 0} \\
	&+ \left(- e^{b s} \xi_{s} + (1+e^{s} \tau_{s}) e^{(b-1) s} \kpp \right) (\bar{U}^{(j+1)}(0) + w_{j+1}) 
	+ \dlt_{0j} e^{(b-1)s} \kpp_{s} 
	+ \frac{1}{2} e^{s} \tau_{s}  \left. \rd_{y}^{(j+1)}(\bar{U} + W)^{2} \right|_{y=0}  \\
	& =(1+e^{s} \tau_{s}) F^{(j)}(s, 0),
\end{aligned}
\end{equation}
where $\dlt_{0j}$ is the Kronecker delta symbol, which equals $1$ when $j = 0$ and vanishes otherwise. 
We also used the following properties of the profile: $\bar U(s,0) = \p^{j}_y \bar U = 0$ for $2\leq j \leq 2k$, and $\p_y \bar U(s,0) = 1$.

We first consider the cases $j = 0, 1$ or $2k$:
\begin{align*}
	&\rd_{s} w_{0} 
	 -b  w_{0} 
	+ w_{0} w_{1} \\
	&+ \left(- e^{b s} \xi_{s} + (1+e^{s} \tau_{s}) e^{(b-1) s} \kpp \right) (-1 + w_{1}) 
	+ e^{(b-1)s} \kpp_{s} 
	+ e^{s} \tau_{s} w_{0} (-1 + w_{1}) \\
	& =(1+e^{s} \tau_{s}) F^{(0)}(s, 0), \\
	&\rd_{s} w_{1} 
	- w_1
	+ w_{0} w_{2}
	+ w_{1}^{2} \\
	&+ \left(- e^{b s} \xi_{s} + (1+e^{s} \tau_{s}) e^{(b-1) s} \kpp \right) w_{2}
	+ e^{s} \tau_{s} \left( w_{2} w_{0} + (-1 + w_{1})^{2} \right) \\
	& =(1+e^{s} \tau_{s}) F^{(1)}(s, 0), \\
	&\rd_{s} w_{2k} 
	+ \left( (2k-1) (b-1) - 1 \right) w_{2k}
	+ w_{0} w_{2k+1}
	+ \left. N(2k) \right|_{y=0} + (2k)! w_{1} \\
	&+ \left(- e^{b s} \xi_{s} + (1+e^{s} \tau_{s}) e^{(b-1) s} \kpp \right) ((2k)! + w_{2k+1}) 
	+ e^{s} \tau_{s} \left((2k+1) (-1 + w_{1}) w_{2k}
+ ((2k)! + w_{2k+1}) w_{0}\right)  \\
	& =(1+e^{s} \tau_{s}) F^{(2k)}(s, 0).
\end{align*}

Observe that the coefficients in front of the $s$-derivatives of the modulation parameters in these three equations are non-degenerate. For this reason, we shall use these equations to determine the dynamic evolution equations for $\kpp$, $\tau$ and $\xi$ by \emph{imposing} the conditions
\begin{equation} \label{eq:orth-cond}
	\boxed{w_{0} = w_{1} = w_{2k} = 0 \quad \hbox{ for all $s$,}}
\end{equation}
which leads to the following equations:
\begin{equation} \label{eq:modul-kpp}
\boxed{
e^{(b-1) s} \kpp_{s} + e^{bs} \xi_{s} - (1+e^{s} \tau_{s}) e^{(b-1) s} \kpp 
= (1+e^{s} \tau_{s}) F^{(0)}(s, 0),
}
\end{equation}
\begin{equation} \label{eq:modul-tau}
\boxed{
e^{s} \tau_{s} 
= (1+e^{s} \tau_{s}) F^{(1)}(s, 0) - w_2(-e^{bs}\xi_s + (1+e^{s}\tau_s)e^{(b-1)s}\kpp ), 
}
\end{equation}
\begin{equation} \label{eq:modul-xi}
\boxed{
\left((2k)!+w_{2k+1}\right) \left( e^{bs} \xi_{s} - (1+e^{s} \tau_{s}) e^{(b-1) s} \kpp \right)
= \left. N(2k) \right|_{y=0} - (1+e^{s} \tau_{s}) F^{(2k)}(s, 0).
}
\end{equation}
\begin{remark}
Note also that, in case $k =1$, the last term in equation~\eqref{eq:modul-tau} vanishes.
\end{remark}

Conversely, if $\kpp_{s}$, $\tau_{s}$ and $\xi_{s}$ are fixed so that \eqref{eq:modul-kpp}--\eqref{eq:modul-xi} are satisfied\footnote{For this purpose, we need to ensure that the coefficient $(2k)! + w_{2k+1}$ is uniformly bounded away from zero; this assertion will be one of the bootstrap assumptions below.} and $w_{0}$, $w_{1}$ and $w_{2k}$ are initially zero, then by \eqref{eq:modul-gen} in the cases $j= 0$, $1$, and $2k$, \eqref{eq:orth-cond} holds.

When $k > 1$, the conditions in \eqref{eq:orth-cond} do \emph{not} fix all values of $w_{j}$ for $j=0, \ldots, 2k$. In such a case, we use the above equation to determine the evolution of $w_{j}$. More precisely, for the remaining indices $j = 2, \ldots, 2k-1$, the ODE for $w_{j}$ is
\begin{equation} \label{eq:modul1}
\boxed{
\begin{aligned}
	&\rd_{s} w_{j} 
	+ \left( (j-1) (b-1) - 1 \right) w_{j}
	+ \left. N(j) \right|_{y = 0} \\
	&+ \left(- e^{b s} \xi_{s} + (1+e^{s} \tau_{s}) e^{(b-1) s} \kpp \right) w_{j+1}
	+ e^{s} \tau_{s} \left(- w_{j} + \left. N(j) \right|_{y=0} \right)  \\
	& =(1+e^{s} \tau_{s}) F^{(j)}(s, 0).
\end{aligned}}
\end{equation}
Here, we used the properties of $\br{U}^{(j)}(0)$ and $w_{0} = w_{1} = w_{2k} = 0$.

We will now rewrite the above system as a system of ODEs.
Introduce the vector $\vec{w}(s) = (w_2(s), \ldots, w_{2k-1}(s))$. Then, $\vec{w}(s)$ satisfies the following system of ODEs:
\begin{equation}\label{eq:odesys}
 \boxed{   \p_s \vec{w}(s) - D \vec{w}(s) + (1+e^{s} \tau_{s}) \mathcal{N}(\vec{w}(s)) = M \vec{w}(s) + \vec{f}(s).}
\end{equation}
Here, $D$ and $M$ are $(2k-1)\times(2k-1)$ matrices given by
\begin{align*}
	D &= \mathrm{diag}\,\left( \lmb_{2}, \ldots, \lmb_{2k-1} \right), \qquad \lmb_{j} = 1 - (j-1) (b-1) = 1 - \frac{j-1}{2k},\\
	M &=e^{s} \tau_{s} I + \left( e^{bs} \xi_{s} - (1+e^{s} \tau_{s}) e^{(b-1) s} \kpp\right) N,
\end{align*}
where $I$ is the identity matrix and $N$ is the nilpotent matrix such that $N_{j (j+1)} = 1$ and $N_{j j'} = 0$ otherwise.

Since $b = \frac{2k+1}{2k}$, each eigenvalue $\lmb_{j}$ of $D$ is strictly positive, so the main linear part $(\rd_{s} - D) \vec{w}(s)$ defines an \emph{unstable} system of ODEs.
In addition, $\mathcal{N}(\vec{w}(s))$ is a vector with quadratic entries as functions of the entries of $\vec{w}$, and $\vec{f}$ is the vector $((1+e^{s} \tau_{s}) F^{(2)}(s,0), \ldots, (1+e^{s} \tau_{s}) F^{(2k-1)}(s, 0))$. 

\section{Precise formulation of the main theorem and reduction to the main bootstrap lemma}\label{sec:mainbst}
\subsection{Initial data in the original variables and the main theorem}
The purpose of this subsection is to give a precise formulation of the main theorem of this paper (Theorem~\ref{thm:main}). We begin by specifying the set of initial data. 

We begin by introducing the following co-dimension $2k+1$ subspace of $H^{2k+3}$:
\begin{equation*}
H^{2k+3}_{(2k)} = \set{W_{0} \in H^{2k+3} : W_{0}(0) = W_{0}'(0) = \cdots = W_{0}^{(2k)}(0) = 0}.
\end{equation*}
We parametrize the initial data in $H^{2k+3}$ that will lead to the desired gradient blow-up solutions with the help of the map $\bfPhi: (0, \infty) \times \bbR \times \bbR \times \bbR^{2k-2} \times H^{2k+3}_{(2k)} \to H^{2k+3}$, which is defined by the formula
\begin{equation} \label{eq:id-map}
\begin{aligned}
	& (\tau_{0}, \xi_{0}, \kpp_{0}, w_{2, 0}, \ldots, w_{2k-1, 0}, W_{0}) \\
	&\mapsto u_{0}(x) = \left. \tau_{0}^{b-1} \left(\chi(-\log \tau_{0}, y) \left( \mathring{U}(y) + \tau_{0}^{1-b} \kpp_{0}\right) + \bar \chi(y) \sum_{j=2}^{2k-1} \frac {w_{j, 0}} {j!} y^{j}  + W_{0}(y) \right) \right|_{y = \tau_{0}^{-b}(x-\xi_{0}) },
\end{aligned}
\end{equation}
where $b = \frac{2k+1}{2k}$, $\mathring{U}$ is the $k$-th smooth self-similar profile for the Burgers equation and $\chi(s, \cdot)$ and $\bar{\chi}(\cdot)$ are as in Subsection~\ref{sub:profiledef}. When $k = 1$, the term $ \bar \chi(y) \sum_{j=2}^{2k-1} \frac {w_{j, 0}} {j!} y^{j}$ is omitted. 

Note that \eqref{eq:id-map} maps the point $(\tau_{0}, \xi_{0}, \kpp_{0}, w_{2, 0}, \ldots, w_{2k-1, 0}, W_{0}) = (\tau_{0}, \xi_{0}, \kpp_{0}, 0, \ldots, 0)$ to the translated and rescaled self-similar Burgers profile whose gradient at $x = \xi_{0}$ is negative and of size $\tau_{0}^{-1}$, i.e., 
\begin{equation*}
(\tau_{0}, \xi_{0}, \kpp_{0}, \ldots, 0) \mapsto u_{0}(x) = \chi(-\log \tau_{0}) \left( \tau_{0}^{b-1} \mathring{U}(\tau_{0}^{-b}(x - \xi_{0})) + \kpp_{0} \right).
\end{equation*}
When $k > 1$, $w_{0, j}$ equals the $j$-th Taylor coefficient of $U(y)$ at $y =0$ in the self-similar variables for $j=2, \ldots, 2k-1$. 

Given $\tau_{0}, \eps_{0} > 0$, we consider the following open subset of $H^{2k+3}_{(2k)}$:
\begin{align}\label{eq:datacalh}
\calO_{\tau_{0}, \eps_{0}} &= \set*{W_{0} \in H^{2k+3}_{(2k)} : \tau_{0}^{\frac{3}{2} b-1} \left( \nrm{W_{0}}_{L^{2}} + \tau_{0}^{-b(2k+3)} \nrm{\rd_{y}^{2k+3} W_{0}}_{L^{2}}\right)  < \eps_{0}}.
\end{align}
When $k > 1$, for $\vec{v}_{0} \in \bbR^{2k-2}$ and $r > 0$, we also introduce the notation
\begin{equation*}
B_{\vec{v}_{0}}(r) = \set{\vec{v} \in \bbR^{2k-2} : \abs{\vec{v} - \vec{v}_{0}} < r}.
\end{equation*}

We are now ready to formulate the main theorem in precise terms.
\begin{theorem} [Precise formulation of the main result] \label{thm:main}
Let $k$ be a positive integer such that $\alp, \bt < \frac{2k}{2k+1}$ and set $b = \frac{2k+1}{2k}$. Then there exist 
$\gmm > 0$ and positive decreasing functions $\tau_{\ast}(\cdot)$, $\eps_{\ast}(\cdot)$ such that the following holds.
Let $\xi_{0} \in \bbR$, $\kpp_{0} \in \bbR$, $\tau_{0} < \tau_{\ast}(\abs{\kpp_{0}})$, $\eps_{0} < \eps_{\ast}(\abs{\kpp_{0}})$ and $W_{0} \in \calO_{\tau_{0}, \eps_{0}}$. When $k = 1$, the initial data $u_{0}(x)$ given by \eqref{eq:id-map} gives rise to a (well-posed) solution to \eqref{eq:fkdv} with initial conditions $u(0, x) = u_{0}(x)$  that blows up in finite time. When $k \geq 2$, there exists $\vec{w}_{0} \in B_{0}(\tau_{0}^{\gmm}) \subseteq \bbR^{2k-2}$ such that the initial data $u_{0}(x)$ given by \eqref{eq:id-map} gives rise to a (well-posed) solution to \eqref{eq:fkdv} with initial conditions $u(0, x) = u_{0}(x)$ that blows up in finite time. In both cases, the following statements hold:
\begin{enumerate}
\item The blow-up time $\tau_{+}$ obeys the bound $\abs{\tau_{+} - \tau_{0}} < C \tau_{0}^{1+\gmm}$.
\item There exist $\xi_{+}$, $\kpp_{+}$ such that
\begin{equation*}
	\abs{\kpp_{+} - \kpp_{0}} \leq C \tau_{0}^{b-1+\gmm}, \quad \abs{\xi_{+} - (\xi_{0} + \tau_{+} \kpp_{0})} \leq C \tau_{0}^{b+\gmm},
\end{equation*}
and such that
\begin{align*}
	\sup_{0 \leq t < \tau_{+}} \nrm{u(t, \cdot)}_{L^{\infty}} + [u(t, \cdot)]_{\calC^{\frac{1}{2k+1}}} 
	\leq C,
\end{align*}
while for every $\sgm \in (\frac{1}{2k+1}, 1)$,
\begin{align*}
C_{\sgm}^{-1} \abs{t - \tau_{+}}^{-\frac{2k+1}{2k}(\sgm-\frac{1}{2k+1})} \leq [u(t, \cdot)]_{\calC^{\sgm}} &\leq C_{\sgm} \abs{t - \tau_{+}}^{-\frac{2k+1}{2k}(\sgm-\frac{1}{2k+1})}
\end{align*}
as $t \to \tau_{+}$.
\end{enumerate}
\end{theorem}

\begin{remark} \label{rem:sharper-description}
For the blow-up solutions in Theorem~\ref{thm:main}, we expect that $\mathring{U}$ to be the blow-up profile, in the sense that $U(s, y)$ in appropriate self-similar variables converges to $\mathring{U}$ as $s \to \infty$ on compact sets of $y$. Such a statement would follow from estimates for $W = U - \chi \mathring{U}$ on top of those proved in this paper, but we have not carried out the details. We refer to \cite{Yang2020} for the proof of this statement in the case of Burgers--Hilbert (i.e., \eqref{eq:fkdv-0} with $\alp = 0$).
\end{remark}

\begin{remark}[Sign of the initial data] \label{rem:id-sign}
There exist smooth compactly supported initial data with both signs (i.e., everywhere nonnegative or nonpositive) that satisfy the hypothesis of Theorem~\ref{thm:main}. Indeed, in \eqref{eq:id-map}, note that $\abs{\mathring{U}} \leq C_{0} \tau_{0}^{1-b}$ on the support of $\chi(-\log \tau_{0}, \cdot)$ (see Lemma~\ref{lem:chi}) for some constant $C_{0} > 0$ independent of $\tau_{0}$. Therefore, if we choose, say, $\abs{\kpp_{0}} > 2 C_{0}$, then the initial profile $\chi(-\log \tau_{0}) ( \tau_{0}^{b-1} \mathring{U}(\tau_{0}^{-b}(x - \xi_{0})) + \kpp_{0} )$ has a definite sign independent of $\tau_{0} > 0$. Moreover, observe that $W_{0} \in \calO_{\tau_{0}, \eps_{0}}$ satisfies the pointwise bound $\abs{W_{0}} \aleq \tau_{0}^{1-b} \eps_{0}$ by the Sobolev embedding. As a consequence, when $k = 1$, the image of \eqref{eq:id-map} with the above choice of $\kpp_{0}$ and $\eps_{0} > 0$ sufficiently small leads to the existence of an open subset of signed initial data in $H^{5}$ that leads to the blow-up behavior described in Theorem~\ref{thm:main}, as alluded to in Remark~\ref{rem:id-sign-0} above. When $k \geq 2$, by taking $\eps_{0}$ and $\tau_{0} > 0$ sufficiently small, we may ensure that the initial data constructed by Theorem~\ref{thm:main} has a definite sign.
\end{remark}

All statements in Theorem~\ref{thm:main-simple} can be read off from Theorem~\ref{thm:main}, with the exception of the stability and the co-dimensionality statements. To formulate these statements, we show that the map \eqref{eq:id-map} is a local homeomorphism.

\begin{lemma} \label{lem:id-map}
For each $\Tht = (\tau_{0}, \xi_{0}, \kpp_{0}, w_{2, 0}, \ldots, w_{2k-1, 0}, W_{0})$ satisfying the hypothesis of Theorem~\ref{thm:main},  the map $\bfPhi$ defined by \eqref{eq:id-map} is a homeomorphism from an open neighborhood of $\Tht$ onto an open neighborhood of  $\bfPhi(\Tht)$ in $H^{2k+3}$.
\end{lemma}
Note that $\bfPhi$ does not possess any further regularity in, for instance, $\tau_{0}$, as it acts as a scaling parameter.
\begin{proof}
Continuity of the map $\bfPhi$ is evident. For every $\mathring{\Tht} \in \tilde{\calO}_{\mathring{\xi}, \mathring{\kpp}, \eps_{0}}$, we may directly construct the continuous inverse in a small neighborhood of $\mathring{u} = \bfPhi(\mathring{\Tht})$ as follows. Let $u$ be sufficiently close to $\mathring{u}$ in the $H^{2k+3}$ topology. Since $\rd_{x}^{2k+1} \mathring{u}(\mathring{\xi}) = (2k)! \tau_{0}^{-2k-2}$ and $\rd_{x}^{2k} \mathring{u}(\mathring{\xi}) = 0$, we may ensure that $\rd_{x}^{2k+1} u(\mathring{\xi})$ is nonzero and $\rd_{x}^{2k} u(\mathring{\xi})$ is small. Hence, we can find a unique point $\xi_{0}$ near $\mathring{\xi}$ such that $\rd_{x}^{2k} u(\xi_{0}) = 0$. Next, we choose $\tau_{0} = - (\rd_{x} u(\xi_{0}))^{-1}$, $\kpp_{0} = u(\xi_{0})$ and $w_{j, 0} = \tau_{0}^{bj} \rd_{x}^{j} u(\xi_{0})$. Finally, define $W_{0}$ from $u$ and the parameters using \eqref{eq:id-map}.
\end{proof}

Lemma~\ref{lem:id-map} shows that the set of initial data for which Theorem~\ref{thm:main} applies in the case $k = 1$ is an open subset of $H^{5}$, which is the precise sense in which the blow-up dynamics described in Theorem~\ref{thm:main} is stable. In the case $k \geq 2$, it establishes the precise sense in which the initial data given by prescribing $\xi_{0} \in \bbR$, $\kpp_{0} \in \bbR$, $\tau_{0} < \tau_{\ast}(\abs{\kpp_{0}})$, $\eps_{0} < \eps_{\ast}(\abs{\kpp_{0}})$ and $W_{0} \in \calO_{\tau_{0}, \eps_{0}}$ but not specifying $\vec{w}_{0} \in B_{0}(\tau_{0}^{\gmm}) \subseteq \bbR^{2k-2}$ is ``co-dimension $2k-2$'' in $H^{2k+2}$, as alluded to in Theorem~\ref{thm:main-simple}.

\begin{remark}
An interesting question, which is not pursued in this article, is the regularity of the co-dimension $2k-2$ set of initial data in $H^{2k+3}$ given by Theorem~\ref{thm:main} and Lemma~\ref{lem:id-map} (e.g., does it form a $C^{1}$ submanifold of $H^{2k+3}$ modelled by $H^{2k+3}$?). Such a result seems to require a careful analysis of the difference of blow-up solutions.
\end{remark}

\subsection{Initial data in self-similar variables} \label{subsec:ss-id}
In this short subsection, we rephrase our ansatz for the initial data in the self-similar variables \eqref{eq:ss-var}, in which most of our analysis will take place. 

We prescribe the initial data at $s = \sgm_{0}$, where conditions on $\sgm_{0}$ will be specified later. In the self-similar variables $(s, y, U)$ given by \eqref{eq:ss-var} with $\tau(\sgm_{0}) = \tau_{0}$, $\xi(\sgm_{0}) = \xi_{0}$ and $\kpp(\sgm_{0}) = \kpp_{0}$, the initial data for $U$ is of the form
\begin{equation}\label{eq:idu}\tag{D1}
U(\sgm_{0}, y) = \chi(\sgm_{0},y) \left(\mathring U(y) + e^{(b-1) \sgm_{0}} \kpp_{0} \right) + \bar \chi(y) \sum_{j=2}^{2k-1} \frac {w_{j, 0}} {j!} y^{j} + W_0(y) -  e^{(b-1) \sgm_{0}} \kpp_{0},
\end{equation}
where the assumptions on $W_{0}$ are as follows: 
\begin{gather}
\rd_{y}^{j} W_{0}(0) = 0 \qquad \hbox{ for all } j = 0, \ldots, 2k, \tag{D2} \label{eq:idw2}\\
\nrm{W_{0}}_{L^{2}} + e^{b(2k+3)\sgm_{0}} \nrm{\rd_{y}^{2k+3} W_{0}}_{L^{2}}  < \eps_{0} e^{(\frac{3}{2} b-1)\sgm_{0}}. \tag{D3} \label{eq:idw3}
\end{gather}
When $k > 1$, the following \emph{smallness conditions} are assumed for the unstable coefficients:
\begin{equation}\label{eq:idw4}\tag{D4}
|w_{j, 0}| \leq e^{-\gmm \sgm_{0}} \qquad \text{for } j = 2, \ldots, 2k-1.
\end{equation}

\subsection{Main bootstrap and shooting lemmas}
In this section, we state two central ingredients of our proof, namely, the main bootstrap lemma in self-similar coordinates (Lemma~\ref{lem:main}) and a shooting lemma for handling the unstable modes when $k > 2$ (Lemma~\ref{lem:top}). 

Recall that $\mu = \min\set{1-b\alp, 1-b \bt, 1}$. Let $\mu_{0}$ be given by\footnote{The reason why separate out the case is $\max\set{\alp, \bt} = \frac{1}{2k+1}$ entirely technical; see Lemma~\ref{lem:forcing} below. We note that $\frac{2k-\frac{3}{2}}{2k}$ can be replaced by any positive number strictly less than $\frac{2k-1}{2k}$.}
\begin{equation} \label{eq:mu0-def}
	\mu_{0} = \begin{cases} 
	\min\set{\mu, \frac{2k-1}{2k}} & \hbox{ when } \max\set{\alp, \bt} \neq \frac{1}{2k+1},\\
	\frac{2k-\frac{3}{2}}{2k} & \hbox{ when } \max \set{\alp, \bt} = \frac{1}{2k+1}.
	\end{cases}
\end{equation}
Fix also a  number $\gmm$ satisfying
\begin{equation} \label{eq:gmm-def}
	0 < \gmm < \mu_{0}.
\end{equation}

To formulate our bootstrap assumptions, we introduce a semi-norm $\dot{\calH}^{n}_{< L}$ ($n$ is a nonnegative integer and $L > 0$) defined by the formula
\begin{equation*}
\nrm{V}_{\dot{\calH}^{n}_{< L}} = \sup_{j \in \bbZ, \, 2^{j} < L} \left( \int_{2^{j-1} < \abs{y} < 2^{j}} (\abs{y}^{n-\frac{1}{2k+1}} \rd_{y}^{n} V)^{2} \frac{\ud y}{y}\right)^{\frac{1}{2}} + L^{n-\frac{1}{2k+1}-\frac{1}{2}} \left( \int_{\abs{y} >\frac{L}{2}} (\rd_{y}^{n} V)^{2} \ud y\right)^{\frac{1}{2}}.
\end{equation*}
A notable feature of this semi-norm is that, in the limit $L =\infty$, it is \emph{invariant} under the self-similar transformation $x = \lmb y$, $u(x) = \lmb^{1-\frac{1}{b}} U(y)$ with $b = \frac{2k+1}{2k}$ for any $\lmb > 0$.

\begin{lemma}[Main bootstrap lemma]\label{lem:main}
There exist increasing functions $\eps_{\ast}^{-1}(\cdot)$, $A(\cdot), y_{0}^{-1}(\cdot)$ and $\sgm_{\ast}(\cdot)$ on $[0, \infty)$, all of which are bounded from below by $1$, such that the following holds. 
Let $\kpp_{0} \in \bbR$, $\sgm_{0} \geq \sgm_{\ast}(\abs{\kpp_{0}})$ and assume that the initial data conditions~\eqref{eq:idu}--\eqref{eq:idw4} are satisfied at $s = \sgm_{0}$ with $\eps_{0} \leq \eps_{\ast}(\abs{\kpp_{0}})$. Suppose that, for some $\sgm_{1} > \sgm_{0}$, $A = A(\abs{\kpp_{0}})$ and $y_{0} = y_{0}(\abs{\kpp_{0}})$, the following estimates are satisfied for $s \in [\sgm_{0}, \sgm_{1}]$:
\begin{align}
&\| \p_y U(s, \cdot)\|_{L^\infty(\R)} \leq 1+2y_0, \label{eq:bst1} \tag{B1}\\
&\| \p_y U(s, \cdot)\|_{L^\infty(\{|y| \geq y_0\})} \leq 1-\frac{y_0^{2k}}{4}, \label{eq:bst3} \tag{B2}\\
&\| \p_y^{2k+3} U(s, \cdot)\|_{L^2(\R)} \leq 2 A, \label{eq:bst2} \tag{B3} \\
&\nrm{U}_{\dot{\calH}^{1}_{<e^{b s}}} \leq 2 A, \label{eq:bst-w-1} \tag{B4} \\
&\nrm{U}_{\dot{\calH}^{2k+3}_{<e^{b s}}} \leq 2 A, \label{eq:bst-w-high} \tag{B5} \\
&|e^{s} \tau_{s}| + \abs{e^{(b-1)s} \kpp_{s}} + \abs{e^{bs} \xi_{s} - (1+e^{s} \tau_{s}) e^{(b-1)s} \kpp} \leq A e^{- \gmm s}, \label{eq:bootstrapmod} \tag{B6}\\
&|W^{(2k+1)}(s,0)| \leq 1. \label{eq:bootstraplow2}\tag{B7}
\end{align}
Assume also that $U(s, 0) = U'(s, 0)+1 = U^{(2k)}(s, 0) = 0$ for all $s \in [\sgm_0, \sgm_1]$.
In case $k > 1$, assume furthermore that $\vec{w}$ satisfies the \emph{trapping condition}
\begin{align} \label{eq:trapped} \tag{T}
&|\vec{w}(s)| \leq e^{-\gmm s}  \hbox{ for } s \in [\sgm_{0}, \sgm_{1}].
\end{align}
Then, stronger estimates actually hold on the interval $s \in [\sgm_{0}, \sgm_{1}]$, as follows:
\begin{align}
&\| \p_y U(s, \cdot)\|_{L^\infty(\R)} \leq 1+y_0, \label{eq:ibst1} \tag{IB1}\\
&\| \p_y U(s, \cdot)\|_{L^\infty(\{|y| \geq y_0\})} \leq 1-\frac{y_0^{2k}}{2}, \label{eq:ibst3} \tag{IB2}\\
&\| \p_y^{2k+3} U(s, \cdot)\|_{L^2(\R)} \leq A, \label{eq:ibst2} \tag{IB3}\\
&\nrm{U}_{\dot{\calH}^{1}_{< e^{bs}}} \leq  A, \label{eq:ibst-w-1} \tag{IB4} \\
&\nrm{U}_{\dot{\calH}^{2k+3}_{< e^{bs}}} \leq A, \label{eq:ibst-w-high} \tag{IB5} \\
&|e^{s} \tau_{s}| + \abs{e^{(b-1)s} \kpp_{s}} + \abs{e^{bs} \xi_{s} - (1+e^{s} \tau_{s}) e^{(b-1)s} \kpp} \leq e^{- \gmm s}, \label{eq:ibootstrapmod} \tag{IB6}\\
&|W^{(2k+1)}(s,0)| \leq \frac{1}{2}. \label{eq:ibootstraplow2}\tag{IB7}
\end{align}
\end{lemma}

\begin{remark}[On dependencies]\label{rem:deps}
We would like to clarify the order in which the above functions $\eps_*$, $A$, $y_0$, and $\sigma_\ast$ are chosen. We start from $\eps_{\ast}$, which is essentially the size of the initial data. Then we choose $A$, which is the bootstrap parameter (we will eventually choose it to be very large), and, in order to be able to Taylor expand at $y = 0$, we choose $y_0$ to be very small based on $A$. This then forces us to choose $\sigma_{\ast}$ very large depending on $y_0$ and $A$.
\end{remark}

When $k = 1$, then Lemma~\ref{lem:main} is already sufficient to set up a bootstrap argument to show the global existence of $U(s, y)$ for all $s \geq \sgm_{0}$, which is the key step in the proof of Theorem~\ref{thm:main} (see the proof of Theorem~\ref{thm:main} below). 

When $k > 1$, the trapping condition \eqref{eq:trapped} for $\vec{w}$ is \emph{not} improved in general, so we need an extra argument to find a global-in-$s$ solution. For this purpose, we introduce the notion of a \emph{trapped solution} as follows:
\begin{definition} \label{def:trapped}
Let $k > 1$. For $\kpp_{0} \in \bbR$, $\xi_{0} \in \bbR$ and $W_{0}$ satisfying the initial data conditions \eqref{eq:idw2}--\eqref{eq:idw3}, let $A$, $y_{0}$ and $\sgm_{0}$ be determined from Lemma~\ref{lem:main}. We say that a solution $U(s, y)$ with the initial data \eqref{eq:idu} induced by $\sgm_{0}$, $\kpp_{0}$, $\xi_{0}$, $W_{0}$ and $\abs{\vec{w}_{0}} \leq e^{-\gmm \sgm_{0}}$ is \emph{trapped} on an interval $[\sgm_{0}, \sgm_{1}]$ if it satisfies \eqref{eq:bst1}--\eqref{eq:bootstraplow2} and \eqref{eq:trapped} on $[\sgm_{0}, \sgm_{1}]$.
\end{definition}
By Lemma~\ref{lem:main}, it follows that the only way a trapped solution $U(s, y)$ on $[\sgm_{0}, \sgm_{1}]$ can fail to be trapped for $s > \sgm_{1}$ is if \eqref{eq:trapped} is saturated at $s = \sgm_{1}$, i.e., $\abs{\vec{w}_{1}(\sgm_{1})} = e^{-\gmm \sgm_{1}}$. Combining this property with a topological fact (namely, the nonexistence of a continuous retraction of a closed ball to its boundary), we shall prove the existence of a globally trapped solution:

\begin{lemma}[Shooting lemma] \label{lem:top}
Let $W_0$, $\kpp_{0}$ and $\xi_{0}$ be fixed so that the conditions \eqref{eq:idw2}--\eqref{eq:idw3} hold, and let $A$, $y_{0}$, and $\sgm_{0}$ be as in Lemma~\ref{lem:main}. Then there is a vector $\abs{\vec{w}_{0}} < e^{-\gmm \sgm_{0}}$ such that the corresponding solution $U(s, y)$ with initial data at $\sgm_{0}$ induced by $\vec{w}_0$ and $W_0$ remains trapped for all $s \geq \sgm_{0}$. 
\end{lemma}

We are going to prove Lemmas~\ref{lem:main} and \ref{lem:top} in Sections~\ref{sec:Uest} and~\ref{sec:unstable} by breaking the proof into several parts. In the remainder of this section, we show how to establish Theorem~\ref{thm:main} assuming Lemmas~\ref{lem:main} and \ref{lem:top}. 

In addition to Lemmas~\ref{lem:main} and \ref{lem:top}, we need three more ingredients, which will be useful in the rest of the paper. The first ingredient is the following simple pointwise bound from the weighted $L^{2}$-Sobolev norm $\dot{\calH}_{<L}^{n}$:
\begin{lemma} \label{lem:btstrap-Linfty}
For any $1 \leq \ell \leq 2k+2$, we have
\begin{equation*}
\abs{\rd_{y}^{\ell} V(y)} \aleq_{\ell, k} \max \set*{\abs{y}^{-\ell+\frac{1}{2k+1}}, L^{-\ell+\frac{1}{2k+1}}} (\nrm{V}_{\dot{\calH}_{<L}^{1}} + \nrm{V}_{\dot{\calH}_{<L}^{2k+3}}).
\end{equation*}

\end{lemma}
\begin{proof}
This lemma follows easily from the Sobolev embedding on the unit interval and scaling; we omit the details.
\end{proof}

The second ingredient is the observation that equation~\eqref{eq:fkdv} admits an $L^2$ bound for $u(t,x)$, which readily translates into an $L^2$ bound for $U(s,y)$ itself. We record this fact in the following lemma.

\begin{lemma}\label{lem:UL2}
Assume that the initial data conditions~\eqref{eq:idu}--\eqref{eq:idw4} are satisfied at $s = \sigma_0$, and $u(t,x)$, $U(s,y)$ are as above. Then, there is $C>0$ such that the following bound holds for $s \in [\sgm_{0}, \sgm_{1}]$:
\begin{equation} \label{eq:U-L2}
	\nrm{U+e^{(b-1)s} \kpp}_{L^{2}_{y}} \leq C e^{(\frac{3}{2} b-1) s} {(1+\kpp_{0})}.
\end{equation}
\end{lemma}

\begin{proof}
We first express the initial data for $u$ in terms of the initial data for $U$. Due to~\eqref{eq:idu}, we have
$$
u(\tau_0,x) = e^{(1-b)\sgm_0}\left( \chi(\sgm_{0},y) \left(\mathring U(y) + e^{(b-1) \sgm_{0}} \kpp(\sgm_{0})\right) + \bar{\chi}(y) \sum_{j=2}^{2k-1} \frac{w_{j, 0}}{j!} y^{j} +  W_{0}(y) \right),
$$
where we remind the reader that $x = e^{-b \sgm_{0}} y + \xi_{0}$. To bound the first (and dominant) term in the above expression, recall from \eqref{eq:profile-Linfty} that $\abs{\mathring{U}(y)} \aleq e^{\frac{1}{2k+1} b \sgm_{0}} = e^{(b-1) \sgm_{0}}$ on the support of $\chi(\sgm_{0}, \cdot)$. Therefore,
\begin{align*}
\int e^{2 (1-b)\sgm_0}\chi(\sgm_{0},y)^{2} \left(\mathring U(y) + e^{(b-1) \sgm_{0}} \kpp(\sgm_{0})\right)^{2} \, \ud x
\aleq (1+\kpp_{0})^{2} \int \chi(\sgm_{0},e^{b \sgm_{0}} (x - \xi_{0}))^{2} \, \ud x \aleq (1+\kpp_{0})^{2},
\end{align*}
where we used Lemma~\ref{lem:chi} again in the last inequality.
The contribution of the last term is bounded precisely by \eqref{eq:idw4}, while the contribution of the second term would decay as $\sgm_{0} \to \infty$ according to our assumptions on the initial data. We eventually obtain:
\begin{equation*}
\nrm{u_{0}}_{L^{2}_x} \leq C (1+\kpp_{0}).
\end{equation*}
We now use the fact that equation~\eqref{eq:fkdv}  satisfies an a-priori $L^{2}$ bound, since $\Gmm(D_{x}) \rd_{x}$ is anti-symmetric (dispersive) and $\Ups(D_{x})$ is nonnegative (dissipative). We then calculate, using the fact that $u =e^{(1-b)s} (U + e^{(b-1) s} \kpp)$,
\begin{equation*}
	\int \abs{u}^{2} \, \ud x
	= \int e^{2(1-b)s} (U + e^{(b-1) s} \kpp)^{2} \, \ud (e^{-b s} y)
	= e^{2s - 3bs} \int  (U + e^{(b-1) s} \kpp)^{2} \, \ud y.
\end{equation*}
This readily implies
\begin{equation*}
	\nrm{U+e^{(b-1)s} \kpp}_{L^{2}_{y}} 
	= e^{s - \frac{3}{2} b s}\nrm{u(\tau(s) -e^{-s}, x)}_{L^{2}_{x}} \leq C e^{(\frac{3}{2} b - 1) s} (1+\kpp_{0}). \qedhere
\end{equation*}
\end{proof}

Finally, the third ingredient concerns some specific bounds for the initial data which follow from the requirements in Section~\ref{subsec:ss-id}. We record these bounds in the next subsection.

\subsection{Consequences of the initial data bounds}
We record here some consequences of the initial data bounds from Section~\ref{subsec:ss-id} which will be used in the proof of Theorem~\ref{thm:main}. By \eqref{eq:idw3} and interpolation, we have
\begin{equation} \label{eq:W0-H-ini}
	\nrm{\rd_{y} W_{0}}_{L^{2}} \leq C \eps_{0} e^{-(1-\frac{1}{2} b) \sgm_{0}}, \quad
	\nrm{\rd_{y}^{2k+3} W_{0}}_{L^{2}} \leq C \eps_{0} e^{-(1+(2k+\frac{3}{2}) b) \sgm_{0}},
\end{equation}
and by the Gagliardo--Nirenberg inequality,
\begin{equation} \label{eq:w2k+1-ini}
	\abs{\rd_{y}^{2k+1} W_{0}(0)} \leq C \eps_{0} e^{-(1+2k b) \sgm_{0}}.
\end{equation}
On the other hand, \eqref{eq:idw3} also implies
\begin{equation} \label{eq:W0-w-ini}
	\nrm{W_{0}}_{\dot{\calH}_{<e^{b\sgm_{0}}}^{1}}
	+ \nrm{W_{0}}_{\dot{\calH}_{<e^{b\sgm_{0}}}^{2k+3}} \leq C \eps_{0}.
\end{equation}
Noting that $\nrm{\bar{U}}_{\dot{\calH}_{<e^{b \sgm_{0}}}^{n}} \leq C_{n}$ for any $n = 0, 1, \ldots$, we have
\begin{equation} \label{eq:U-w-ini}
	\nrm{U(\sgm_{0}, \cdot)}_{\dot{\calH}_{<e^{b\sgm_{0}}}^{1}}
	+ \nrm{U(\sgm_{0}, \cdot)}_{\dot{\calH}_{<e^{b\sgm_{0}}}^{2k+3}} \leq C.
\end{equation}
By the definition of $\bar{U}$, \eqref{eq:W0-w-ini} and Lemma~\ref{lem:btstrap-Linfty}, we also obtain the pointwise bound
\begin{equation} \label{eq:U'-ini}
	\abs{\rd_{y} U(\sgm_{0}, y)} \leq C \max\set{(1+\abs{y})^{-\frac{2k}{2k+1}}, e^{-\sgm_{0}} \eps_{0}}.
\end{equation}

\subsection{Proof of the main theorem}
We are now ready to give a proof of Theorem~\ref{thm:main}.
\begin{proof}[Proof of Theorem~\ref{thm:main} assuming Lemmas~\ref{lem:main} and \ref{lem:top}]
Let $\tau_{\ast}(\cdot) = e^{-\sgm_{\ast}(\cdot)}$ and define $\sgm_{0}$ by $\tau_{0} = e^{-\sgm_{0}}$. In case $k = 1$, by a standard bootstrap argument using Lemma~\ref{lem:main}, there exist $C^{1}$ functions $\tau(\cdot)$, $\kpp(\cdot)$, and $\xi(\cdot)$ on $[\sgm_{0}, \infty)$ such that in the self-similar variables $(s, y, U)$ given by \eqref{eq:ss-var} with $\tau(\cdot)$, $\kpp(\cdot)$ and $\xi(\cdot)$, $U(s, y)$ is a globally trapped solution on $[\sgm_{0}, \infty)$ and $\tau$, $\kpp$, and $\xi$ solve \eqref{eq:modul-kpp}--\eqref{eq:modul-xi} with $\tau(\sgm_{0}) = \tau_{0}$ (so that $s = \sgm_{0}$ corresponds to $t = 0$), $\kpp(\sgm_{0}) = \kpp_{0}$ and $\xi(\sgm_{0}) = \xi_{0}$. In case $k \geq 2$, by Lemmas~\ref{lem:main} and \ref{lem:top}, there exists $\vec{w}_{0} \in B_{0}(e^{- \gmm \sgm_{0}})$ such that the above conclusion holds.

By integrating the ODEs for $\tau_{s}$, $\kpp_{s}$ and $\xi_{s}$ in \eqref{eq:ibootstrapmod}, it follows that $(\tau(s), \kpp(s), \xi(s)) \to (\tau_{+}, \kpp_{+}, \xi_{+})$ as $s \to \infty$, where
\begin{equation} \label{eq:modul-final}
	\abs{\tau_{+} - \tau} \aleq e^{-(1+\gmm)s}, \quad
	\abs{\kpp_{+} - \kpp} \aleq e^{-(b-1+\gmm)s}, \quad
	\abs{\xi_{+} - \xi - (\tau_{+} - \tau + e^{-s}) \kpp_{+}} \aleq e^{-(b+\gmm)s}.
\end{equation}
In particular, by \eqref{eq:ibootstrapmod} and $\abs{\tau_{+} - \tau} \aleq e^{-(1+\gmm)s}$, it follows that the change of variables $s \to t$ is a well-defined strictly increasing map from $[\sgm_{0}, \infty)$ onto $[0, \tau_{+})$. Since $\rd_{y} U(s, 0) = -1$ for all $s$, it follows that $\rd_{x} u(t, \xi(s(t))) = - (\tau(s(t)) - t)^{-1} \to \infty$ as $t \to \tau_{+}$, which implies that $u$ indeed blows up as $t \nearrow \tau_{+}$. The desired bounds on $\tau_{+}$, $\kpp_{+}$ and $\xi_{+}$ also follow from \eqref{eq:modul-final}.

To complete the proof, it remains to establish the regularity and blow-up properties of $u$, which we derive from properties of $U$ and the change of variables \eqref{eq:ss-var}. To begin with, note that, by \eqref{eq:ibst-w-1}--\eqref{eq:ibst-w-high} and Lemma~\ref{lem:btstrap-Linfty}, we have
\begin{equation} \label{eq:U'-ptwise-sharp}
	\abs{U'(s, y)} \leq C A \max\set*{\abs{y}^{-\frac{2k}{2k+1}}, e^{-s}} \quad \hbox{ for } \abs{y} \geq 1.
\end{equation}
On the other hand, $\abs{U'(s, y)} \leq 2$ for $\abs{y} \leq 1$ by \eqref{eq:ibst1}--\eqref{eq:ibst3}. Using $U(s, 0) = 0$ and by integration, we arrive at
\begin{equation} \label{eq:U-ptwise-sharp}
	\abs{U(s, y)} \leq 
	\begin{cases}
	C \abs{y} & \hbox{ for } \abs{y} \leq 1, \\
	C A \max\set*{\abs{y}^{\frac{1}{2k+1}}, \abs{y} e^{-s}} & \hbox{ for } \abs{y} \geq 1.
	\end{cases}
\end{equation}
For $\abs{y} > e^{b s}$, we may eliminate the linear growth $\abs{y} e^{-s}$ by using the Sobolev inequality based on the $L^{2}$ bound \eqref{eq:U-L2} and 
\begin{equation*}
\nrm{\rd_{y} (U + e^{(b-1) s} \kpp)}_{L^{2}( \abs{y} > e^{b s} )} = \nrm{\rd_{y} U}_{L^{2} ( \abs{y} > e^{b s} )} \leq e^{( \frac{1}{2} b-1) s} \nrm{U}_{\dot{\calH}_{<e^{bs}}^{1}} \leq e^{( \frac{1}{2} b - 1) s} A.
\end{equation*}
As a consequence, we obtain
\begin{equation} \label{eq:U-ptwise-outside}
	\abs{U(s, y) + e^{(b-1)} \kpp} \leq C e^{( b-1) s} (1 + \kpp_{0} + A) \quad \hbox{ for } \abs{y} \geq e^{b s},
\end{equation}
which is an improvement over \eqref{eq:U-ptwise-sharp}. In particular, it follows that 
\begin{equation} \label{eq:U-ptwise-uniform}
\abs{U(s, y)} \leq C e^{(b-1) s} (1 + \kpp_{0} + A) \quad \hbox{ for all } y \in \bbR,
\end{equation}
which implies via \eqref{eq:ss-var} that $\nrm{u}_{L^{\infty}}$ is uniformly bounded up to the blow-up time $\tau_{+}$.

To prove the upper bounds on the H\"older semi-norms, first observe the simple gradient bound $\abs{U'} \leq C A \abs{y}^{-\frac{2k}{2k+1}}$ from \eqref{eq:ibst1}--\eqref{eq:ibst3} and \eqref{eq:U'-ptwise-sharp}. 

For each $\Dlt y > 0$, note that
\begin{align*}	
	[U]_{\calC^{\frac{1}{2k+1}}}
	= \sup_{y \in \bbR, \Dlt y > 0} \frac{\abs{U(s, y + \Dlt y) - U(s, y)}}{(\Dlt y)^{\frac{1}{2k+1}}}
	\leq C A \sup_{y \in \bbR, \Dlt y > 0} (\Dlt y)^{-\frac{1}{2k+1}} \int_{y}^{y+\Dlt y} \abs{y'}^{-\frac{2k}{2k+1}} \, \ud y' 
	\leq C_{\sgm} A.
\end{align*}
By \eqref{eq:ss-var}, the boundedness of $[u]_{\calC^{\frac{1}{2k+1}}}$ up to the blow-up time $\tau_{+}$ follows. Then interpolating with the trivial upper bound $\nrm{\rd_{x} u}_{L^{\infty}} = (\tau_{+} - t)^{-1} \nrm{\rd_{y} U}_{L^{\infty}} \aleq (\tau_{+} - t)^{-1}$, the upper bounds when $\frac{1}{2k+1} < \sgm < 1$ follow.

Finally, to establish the lower bounds on the H\"older semi-norms, note first that $\inf_{s \geq \sgm_{0}, \, \abs{y} \leq c_{0}} \abs{U'(s, y)} > 0$ for some $c_{0} > 0$ by Taylor expansion. By the mean value theorem,
\begin{equation*}
	\abs{y}^{-\sgm} \abs{U(s, y) - U(s, 0)} \geq C \abs{y}^{1-\sgm}   \quad \hbox{ for } \abs{y} \leq c_{0},
\end{equation*}
and then by \eqref{eq:ss-var}, the desired lower bound follows.
\end{proof}

\section{Lemmas on Fourier multiplier} \label{sec:fourier}
In this section, we establish key analytic lemmas concerning the operators $\calH$ and $\calL$, whose definitions are recalled here for convenience:
\begin{align*}
\calH(V) &= - P_{> 0}(e^{b s} D_{y}) \left( e^{-\max\set{\alp, \bt, 0} b s} \Gmm(e^{bs} D_{y}) e^{bs} \rd_{y} V  + e^{- \max\set{\alp, \bt, 0} b s} \Ups(e^{bs} D_{y}) V \right), \\
\calL(V) &= - P_{\leq 0}(e^{b s} D_{y}) \left( \Gmm(e^{bs} D_{y}) e^{bs} \rd_{y} V + \Ups(e^{bs} D_{y}) V \right).
\end{align*}

Observe that the assumptions on $\Gmm$ and $\Ups$ remain true under any increase of $\alp$ or $\bt$. In the proofs in this section, we will often {\bf assume, without loss of generality, that $\alp = \bt$ and $\alp \geq 0$, so that $\max\set{\alp, \bt, 0} = \alp$.} 

We begin with simple $L^{2}$ and $L^{\infty}$ estimates for $\calH$ and $\calL$.
\begin{lemma} \label{lem:mult-L2Linfty}
For any $\ell \geq 0$, we have
\begin{align}
\nrm{\rd_{y}^{\ell} \calL(V)}_{L^{2}} & \aleq_{\alp, \bt} e^{-\ell b s} \nrm{V}_{L^{2}}, \label{eq:calL-L2} \\
\nrm{\rd_{y}^{\ell} \calL(V)}_{L^{\infty}} &\aleq_{\alp, \bt} e^{-(\frac{1}{2}+\ell) bs}\nrm{V}_{L^{2}}. \label{eq:calL-Linfty}
\end{align}
For $\max\set{\alp, \bt} < 1$, we have
\begin{align} 
	\nrm{\calH(V)}_{L^{2}} & \aleq_{\alp, \bt} \nrm{V}_{L^{2}}^{1-\max\set{\alp, \bt, 0}} \nrm{\rd_{y} V}_{L^{2}}^{\max\set{\alp, \bt, 0}}, \label{eq:calH-L2} \\
	\nrm{\calH(V)}_{L^{\infty}} &\aleq_{\alp, \bt} \nrm{V}_{L^{\infty}}^{1-\frac{2}{3} \max\set{\alp, \bt, 0}} \nrm{\rd_{y}^{2} V}_{L^{2}}^{\frac{2}{3}\max\set{\alp, \bt, 0}}. \label{eq:calH-Linfty}
\end{align}
\end{lemma}
\begin{proof}
The $L^{2}$ bound \eqref{eq:calL-L2} for $\calL$ is simply a consequence of the fact that, thanks to the frequency projection $P_{\leq 0}(e^{b s} D_{y})$ and the assumptions on $\Gmm$, $\Ups$, $\calL$ is a Fourier multiplier with bounded symbol. The case $\ell \geq 1$ then follows, thanks again to the frequency projection $P_{\leq 0}(e^{b s} D_{y})$. Moreover, \eqref{eq:calL-Linfty} follows from Bernstein's inequality.

To prove \eqref{eq:calH-L2}, it suffices to prove that, for all $k \in \bbZ$,
\begin{align*}
	\nrm{P_{k}(D_{y}) \calH V}_{L^{2}}
	& \aleq \min\set*{2^{\alp k} \nrm{P_{k} (D_{y}) V}_{L^{2}}, 2^{-(1-\alp)k} \nrm{\rd_{y} P_{k}(D_{y}) V}_{L^{2}}}, \\
	\nrm{P_{k}(D_{y}) \calH V}_{L^{\infty}}
	& \aleq \min\set*{2^{\alp k} \nrm{P_{k} (D_{y}) V}_{L^{\infty}}, 2^{-(\frac{3}{2}-\alp)k} \nrm{\rd_{y}^{2} P_{k}(D_{y}) V}_{L^{2}}}.
\end{align*}
To see this (in particular, the $L^{\infty}$ bound), note that 
\begin{equation*}
P_{k}(D_{y}) e^{-b\alp s}\Gmm(e^{b s} D_{y}) e^{b s} \rd_{y} V = 2^{\alp k} K_{k} \ast V(y), \quad \hbox{ where } 2^{\alp k} K_{k} = \calF^{-1}_{\xi_{y}}[i P_{k}(\xi_{y}) e^{-b\alp s}\Gmm(e^{b s} \xi_{y}) e^{b s} \xi_{y}].
\end{equation*}

Indeed, by the assumptions on $\Gmm$, the kernel of $P_{k'}(D_{x}) \Gmm(D_{x}) \rd_{x}$ is of the form $2^{\alp k'} K_{k'}(x)$, where $\int \abs{K_{k}(x)} \, \ud x \aleq 1$ (independent of $k$). By rescaling $x = e^{b s} y$, we see that the kernel of $P_{k}(D_{y}) e^{-b \alp s} \Gmm(e^{b s} D_{y}) e^{b s} \rd_{y} $ is of the form $2^{\alp k} e^{-bs} K_{k-(\log 2)^{-1} b s}(e^{-bs} y)$,  where the $y$-integral of $e^{-bs} \abs{K_{k-(\log 2)^{-1} b s}(e^{-bs} y)}$ is uniformly bounded in $k$. The desired bounds for the contribution of $\Gmm$ in $\calH$ now follows from Young's inequality. A similar bound holds for $\Ups$. \qedhere
\end{proof}

Next, we prove a sharp upper bound on the kernel of the operator $\calH$.
\begin{lemma} \label{lem:mult-ker}
For each $s$, there exists a function $K_{s} \in C^{\infty}(\bbR \setminus \set{0})$ such that
\begin{align*}
\calH V(y) = \int_{-\infty}^{\infty} K_{s}(y - y') \rd_{y} V(y') \, \ud y', 
\end{align*}
where
\begin{equation*}
\abs{K_{s}(y)} + \abs{y} \abs{\rd_{y} K_{s}(y)} \aleq \abs{y}^{-\max\set{\alp, \bt, 0}}.
\end{equation*}
\end{lemma}

\begin{proof}
Without loss of generality, assume $\alp = \bt \geq 0$. By the Fourier inversion formula, we have
\begin{align*}
P_{\geq 0}(D_{x}) \Gmm(D_{x}) \rd_{x} f & = \int_{-\infty}^{\infty} K(x - x') \rd_{x} f(x') \, \ud x', \\
P_{\geq 0}(D_{x}) \Ups(D_{x}) f & = \int_{-\infty}^{\infty} K'(x - x') \rd_{x} f(x') \, \ud x',
\end{align*}
where
\begin{align*}
	\abs{K(z)} + \abs{z} \abs{\rd_{z} K(z)} & \aleq \abs{z}^{-\alp}, \\
	\abs{K'(z)} + \abs{z} \abs{\rd_{z} K'(z)} & \aleq \abs{z}^{-\alp}.
\end{align*}
The desired statement now follows by applying the rescaling $x = e^{b s} y$. \qedhere
\end{proof}

Finally, we formulate and prove a key commutator estimate for $\calH$ in the weighted $L^{2}$-Sobolev space $\dot{\calH}_{<L}^{n}$ introduced earlier. For this purpose, it is instructive to generalize the weight in the semi-norm and introduce
\begin{equation*}
	\nrm{V}_{\dot{\calH}_{< L}^{n, \nu}} 
	= \sup_{j \in \bbZ, \, 2^{j} < L} \left( \int_{2^{j-1} < \abs{y} < 2^{j}} (\abs{y}^{\nu} \rd_{y}^{n} V)^{2} \, \ud y \right)^{\frac{1}{2}}
	+ L^{\nu} \left( \int_{\abs{y} > \frac{L}{2}} (\rd_{y}^{n} V)^{2} \, \ud y \right)^{\frac{1}{2}}.
\end{equation*}

\begin{lemma} \label{lem:mult-w}
Let $-\frac{1}{2} < \nu < \frac{1}{2}$, $\ell \in \set{0, 1, \ldots}$ and $L > 1$. Let $\varpi$ be a smooth function satisfying one of the following assumptions:
\begin{enumerate}[leftmargin=*, label=Case~\arabic*.]
\item $\supp \varpi \subseteq \set{2^{j_{0}-1-c_{0}} < \abs{y} < 2^{j_{0}+c_{0}}}$ and $0 \leq \varpi \leq C_{0} 2^{(\nu + \ell) j_{0}}$ for some $c_{0}, C_{0} > 0$ and $j_{0} \in \bbZ$ such that $2^{j_{0}} < L$, or
\item $\supp \varpi \subseteq \set{\abs{y} > 2^{j_{0}-1-c_{0}}}$ and $0 \leq \varpi \leq C_{0} 2^{(\nu + \ell) j_{0}}$ for some $c_{0}, C_{0} > 0$ and $j_{0} = \lfloor \log_{2} L \rfloor$.
\end{enumerate}
Then for any $s \in \bbR$ and $V \in \calH_{< L}^{\ell+1, \nu}$, we have
\begin{equation} \label{eq:calH-w}
\nrm{\varpi \calH \rd_{y}^{\ell} V}_{L^{2}} \aleq_{\alp, \bt, \nu, c_{0}, C_{0}} 2^{-\max\set{\alp, \bt, 0} j_{0}} (\nrm{V}_{\dot{\calH}_{< L}^{0, \nu}} + \nrm{V}_{\dot{\calH}_{< L}^{\ell+1, \nu}}),
\end{equation}
where the implicit constant is independent of $s$ and $L$.

Suppose, in addition, that $\abs{\varpi'} \leq C_{0} 2^{(\nu + \ell - 1)j_{0}}$ and $\ell \geq 1$. Then for any $s \in \bbR$ and $V \in \calH_{< L}^{\ell, \nu}$, we have
\begin{equation} \label{eq:calH-comm-w}
\nrm{[\varpi, \calH] \rd_{y}^{\ell} V}_{L^{2}} \aleq_{\alp, \bt, \nu, c_{0}, C_{0}} 2^{-\max\set{\alp, \bt, 0} j_{0}} (\nrm{V}_{\dot{\calH}_{< L}^{0, \nu}} + \nrm{V}_{\dot{\calH}_{< L}^{\ell, \nu}}),
\end{equation}
where the implicit constant is independent of $s$ and $L$.
\end{lemma}

The range $-\frac{1}{2} < \nu < \frac{1}{2}$ is sharp. In the proof of Lemma~\ref{lem:weighted2} below, this lemma will be applied with $V = U'$ and $\nu = \frac{1}{2} - \frac{1}{2k+1}$; indeed, observe that $\nrm{U}_{\dot{\calH}_{<L}^{n}} = \nrm{U'}_{\dot{\calH}_{<L}^{n-1, \frac{1}{2} - \frac{1}{2k+1}}}$ for $n \geq 1$.

\begin{remark}
We note that while \eqref{eq:calH-w} and \eqref{eq:calH-comm-w} are sharp in terms of the spatial weights, it is not sharp in terms of regularity, as we are only working with integer regularity indices. Indeed, the orders of the operators $\varpi \calH \rd_{y}^{\ell}$ and $[\varpi, \calH] \rd_{y}^{\ell}$ are $\alp + \ell$ and $\alp + \ell - 1$, respectively, while we are using $\ell+1$ and $\ell$ derivatives on the RHS, respectively (recall that $0 \leq \alp < 1$). A crucial point, however, is that the RHS of the commutator estimate, \eqref{eq:calH-comm-w}, involves at most $\ell$ derivatives, which is important for avoiding any loss of derivatives in Lemma~\ref{lem:weighted2}. 

Another important point is that \eqref{eq:calH-w} and \eqref{eq:calH-comm-w} are independent of $s$ and $L$. In particular, the only $s$-independent information we have on the symbol of $\calH$ are the scale-invariant bounds 
$$
\abs{\rd_{\xi_{y}}^{N} \calH(\xi_{y})} \aleq_{N} \abs{\xi_{y}}^{\max\set{\alp, \bt, 0}-N},
$$
which are essentially all we use about $\calH$.
 \end{remark}

\begin{proof}
Without loss of generality, assume $\alp = \bt \geq 0$. In what follows, we suppress the dependence of implicit constants on $\alp$, $\nu$, $c_{0}$ and $C_{0}$. In what follows, we simply write $P_{k} = P_{k}(D_{y})$ ($k \in \bbZ$). 

To simplify the notation, we introduce the following schematic notation: We denote by $\tilde{P}_{k}$ (resp.~$\tilde{K}_{k_{0}}$) any function, which may vary from expression to expression, that obeys the same support properties and bounds as $P_{k}$ (at the level of the symbol) (resp.~$K_{k}$), i.e., $\supp \tilde{P}_{k} \subseteq \set{\xi \in \bbR : 2^{k-5} < \abs{\xi} < 2^{k+5}}$ and $\abs{(\xi \rd_{\xi})^{n} \tilde{P}_{k}(\xi)} \aleq_{n} 1$ (resp.~$\abs{\rd_{y}^{n} \tilde{K}_{k}(y)} \aleq_{N, n} \frac{2^{(1+n)k}}{\brk{2^{k} y}^{n+N}}$). 

With the above conventions, we have the schematic identities $P_{k} = \tilde{P}_{k}$ and
\begin{equation*}
	\calH = \sum_{k} \calH P_{k}  = \sum_{k} 2^{\alp k} \tilde{P}_{k},
\end{equation*}
where an important point in the last identity is that the implicit bounds for $\tilde{P}_{k}$ are \emph{independent} of $s$. Note also that any operator of the form $\tilde{P}_{k}$ has a kernel of the form $\tilde{P}_{k} V = \tilde{K}_{k} \ast V$.

Next, we introduce a nonnegative smooth partition of unity $\set{\eta_{j}}_{j \in \bbZ}$ on $\bbR$ subordinate to the open cover $\set{A_{j} = \set{y \in \bbR : 2^{j-3} < \abs{y} < 2^{j+2}}}_{j \in \bbZ}$. We shall write $\eta_{\geq j} = \sum_{j' \geq j} \eta_{j'}$. We also introduce the shorthands
\begin{equation*}
	\breve{\eta}_{j_{0}} = 2^{-(\nu + \ell) j_{0}} \varpi \hbox{ in Case~1,} \qquad
	\breve{\eta}_{\geq j_{0}} = 2^{-(\nu + \ell) j_{0}} \varpi \hbox{ in Case 2}.
\end{equation*}
As the notation suggests, $\breve{\eta}_{j_{0}}$ and $\breve{\eta}_{\geq j_{0}}$ have similar support and upper bound properties as $\eta_{j_{0}}$ and $\eta_{\geq j_{0}}$, respectively, thanks to the hypothesis on $\varpi$. However, note that we only have control of up to one derivative of $\breve{\eta}_{j_{0}}$ and $\breve{\eta}_{\geq j_{0}}$.

\smallskip
\noindent {\it Case~1, Step~1.}
We will use the following three bounds to treat the ``non-local'', ``low frequency'' and ``far-away input'' cases, respectively: for $\abs{j - j_{0}} > c_{0}+5$ and $k \geq - j_{0} - 5$,
\begin{align}
	2^{(\nu + \ell) j_{0}} 2^{(\alp + \ell) k}  \nrm{\breve{\eta}_{j_{0}} \tilde{P}_{k} (\eta_{j} V)}_{L^{2}} &\aleq 2^{-\alp j_{0}} 2^{-c \abs{j_{0} + k}} 2^{-c \abs{j - j_{0}}} \nrm{V}_{\dot{\calH}_{<L}^{0, \nu}}, \label{eq:w-Pk-high}
\end{align}
and for $k < - j_{0} - 5$,
\begin{align}
	2^{(\nu + \ell) j_{0}} 2^{(\alp + \ell) k}  \nrm{\breve{\eta}_{j_{0}} \tilde{P}_{k} (\eta_{j} V)}_{L^{2}} &\aleq 2^{-\alp j_{0}} 2^{-c \abs{j_{0} + k}} 2^{-c \abs{j + k}} \nrm{V}_{\dot{\calH}_{<L}^{0, \nu}}, \label{eq:w-Pk-low}
\end{align}
and
\begin{align}
	2^{(\nu + \ell) j_{0}} 2^{(\alp + \ell) k}  \nrm{\breve{\eta}_{j_{0}} \tilde{P}_{k} (\eta_{\geq \log_{2} L} V)}_{L^{2}} &\aleq 2^{-\alp j_{0}} 2^{-c \abs{j_{0} + k}} 2^{-c \abs{\log_{2} L - j_{0}}} \nrm{V}_{\dot{\calH}_{<L}^{0, \nu}}, \label{eq:w-Pk-far}
\end{align}
We defer their proofs for a moment and prove \eqref{eq:calH-w} and \eqref{eq:calH-comm-w} assuming \eqref{eq:w-Pk-high}--\eqref{eq:w-Pk-far}. 

\smallskip
\noindent {\it Case~1, Step~1.(a).}
To prove \eqref{eq:calH-w}, we begin by expanding
\begin{equation*}
	\varpi \calH \rd_{y}^{\ell} V
	= \sum_{j, k : 2^{j} < L} 2^{(\nu + \ell) j_{0}} 2^{(\alp + \ell) k} \breve{\eta}_{j_{0}} \tilde{P}_{k} (\eta_{j} V)
	+ \sum_{k} 2^{(\nu + \ell) j_{0}} 2^{(\alp + \ell) k} \breve{\eta}_{j_{0}} \tilde{P}_{k} (\eta_{\geq \log_{2} L} V) 
	=: \mathrm{I}_{near} + \mathrm{I}_{far}.
\end{equation*}
The term $\mathrm{I}_{far}$ can be treated using \eqref{eq:w-Pk-far}, so it only remains to estimate $\mathrm{I}_{near}$. Unless $\abs{j - j_{0}} \leq c_{0} + 5$ and $k \geq - j_{0} - 5$, we can apply \eqref{eq:w-Pk-high} and \eqref{eq:w-Pk-low} to obtain an acceptable bound for the summand. When $\abs{j - j_{0}} \leq c_{0} + 5$ and $k \geq - j_{0} - 5$, we use the schematic identity $\tilde{P}_{k} = 2^{-(\ell+1) k} \tilde{P}_{k} \rd_{x}^{\ell+1}$ to simply bound
\begin{align*}
2^{(\nu + \ell) j_{0}} 2^{(\alp + \ell) k} \nrm{\breve{\eta}_{j_{0}} \tilde{P}_{k} (\eta_{j} V)}_{L^{2}}
& \aleq 2^{- \alp j_{0}} 2^{- (1-\alp) (j_{0}+k)} (\nrm{V}_{\dot{\calH}_{<L}^{0, \nu}} + \nrm{V}_{\dot{\calH}_{<L}^{\ell+1, \nu}}),
\end{align*}
which can be summed up in $k \geq -j_{0} + 5$.	

\smallskip
\noindent {\it Case~1, Step~1.(b).} 
Now we prove the commutator estimate \eqref{eq:calH-comm-w}. We begin by making the following decomposition:
\begin{align}
[\breve{\eta}_{j_{0}}, \calH] \rd_{y}^{\ell} V 
&= \sum_{k \geq -j_{0}-5} 2^{(\nu + \ell) j_{0}} 2^{\alp k} [\breve{\eta}_{j_{0}}, \tilde{P}_{k}] \rd_{y}^{\ell}  V \notag
+ \sum_{k < -j_{0}-5} 2^{(\nu + \ell) j_{0}} 2^{\alp k} \breve{\eta}_{j_{0}} \tilde{P}_{k} \rd_{y}^{\ell} V \notag \\
&\peq + \sum_{k < -j_{0}-5} 2^{(\nu + \ell) j_{0}} 2^{\alp k} \tilde{P}_{k} ( \breve{\eta}_{j_{0}} \rd_{y}^{\ell} V) \notag \\
&=: \mathrm{I} + \mathrm{II} + \mathrm{III}. \label{eq:calH-comm-w-decomp}
\end{align}

We treat each term in \eqref{eq:calH-comm-w-decomp} as follows. For $\mathrm{I}$, we start by writing
\begin{equation} \label{eq:calH-comm-w-1}
\begin{aligned}
\mathrm{I}
&= \sum_{j , k : \abs{j - j_{0}} \leq c_{0} + 5,\, k \geq -j_{0}-5} 2^{(\nu + \ell) j_{0}} 2^{\alp k} [\breve{\eta}_{j_{0}}, \tilde{P}_{k}] \rd_{y}^{\ell}  (\eta_{j} V) \\
&\peq + \sum_{j , k : \abs{j - j_{0}} > c_{0} + 5,\, k \geq -j_{0}-5} 2^{(\nu + \ell) j_{0}} 2^{\alp k} [\breve{\eta}_{j_{0}}, \tilde{P}_{k}] \rd_{y}^{\ell}  (\eta_{j} V).
\end{aligned}
\end{equation}
To treat the first sum on the RHS of \eqref{eq:calH-comm-w-1}, we make use of the commutator structure. We write
\begin{align*}
[\breve{\eta}_{j_{0}}, \tilde{P}_{k_{0}}] \tilde{V}
&= \int \tilde{K}_{k_{0}}(y-y') (\breve{\eta}_{j_{0}}(y) - \breve{\eta}_{j_{0}}(y')) \tilde{V}(y') \, \ud y',
\end{align*}
where $\tilde{V} = \rd_{y}^{\ell}  (\eta_{j} V)$. Then using the bound for $\breve{\eta}_{j_{0}}'$ (which comes from that for $\varpi'$) and the $O(2^{k_{0}})$-localization of $\tilde{K}_{k_{0}}$, the kernel on the RHS can be written as $2^{-k_{0}-j_{0}} \breve{K}(y, y')$, where $\sup_{y} \nrm{\breve{K}(y, \cdot)}_{L^{1}}$ and $\sup_{y'} \nrm{\breve{K}(\cdot, y')}_{L^{1}}$ are bounded by an absolute constant. Hence, by Schur's test, 
\begin{align*}
\nrm{2^{(\nu + \ell) j_{0}} 2^{\alp k} [\breve{\eta}_{j_{0}}, \tilde{P}_{k}] \rd_{y}^{\ell}  (\eta_{j} V)}_{L^{2}} 
& \aleq 2^{(-\alp + \nu + \ell) j_{0}} 2^{-(1-\alp)(j_{0}+ k)} \nrm{\rd_{y}^{\ell}(\eta_{j} V)}_{L^{2}} \\
& \aleq 2^{-\alp j_{0}} 2^{-(1-\alp)(j_{0}+ k)} (\nrm{V}_{\dot{\calH}_{<L}^{0, \nu}} + \nrm{V}_{\dot{\calH}_{<L}^{\ell, \nu}}),
\end{align*}
which is summable in the range $\set{(j, k) : \abs{j - j_{0}} \leq c_{0} + 5,\, k \geq -j_{0}-5}$.

For the second sum on the RHS of \eqref{eq:calH-comm-w-1}, we simply note that $[\breve{\eta}_{j_{0}}, \tilde{P}_{k_{0}}] \rd_{y}^{\ell} (\eta_{j} V) = \breve{\eta}_{j_{0}} \tilde{P}_{k_{0}} (\eta_{j} V)$ by the support properties of $\breve{\eta}_{j_{0}}$ and $\eta_{j}$. Hence, we may apply \eqref{eq:w-Pk-high}, which is acceptable in the range $\set{(j, k) : \abs{j - j_{0}} > c_{0} + 5,\, k \geq -j_{0}-5}$.

Next, the term $\mathrm{II}$ in \eqref{eq:calH-comm-w-decomp} is directly bounded using \eqref{eq:w-Pk-low}.

Finally, we turn to the term $\mathrm{III}$ in \eqref{eq:calH-comm-w-decomp}. We write
\begin{align*}
\sum_{k < -j_{0}-5} 2^{(\nu + \ell) j_{0}} 2^{\alp k} \tilde{P}_{k} ( \breve{\eta}_{j_{0}} \rd_{y}^{\ell} V)
= \sum_{j, k : k < -j_{0}-5} 2^{(\nu + \ell) j_{0}} 2^{\alp k} \tilde{P}_{k} ( \breve{\eta}_{j_{0}} \rd_{y}^{\ell} (\eta_{j} V))
\end{align*}
By the support properties of $\breve{\eta}_{j_{0}}$ and $\eta_{j}$, the summand vanishes unless $\abs{j_{0} - j} \leq c_{0} + 5$. In that case, we have
\begin{align*}
	2^{(\nu + \ell) j_{0}} 2^{\alp k} \nrm{\tilde{P}_{k} (\breve{\eta}_{j_{0}} \rd_{y}^{\ell}(\eta_{j} V))}_{L^{2}}
	&\aleq 2^{(\nu + \ell +\frac{1}{2}) j_{0}} 2^{(\alp+\frac{1}{2}) k} \nrm{\rd_{y}^{\ell}(\eta_{j} V)}_{L^{2}} \\
	&\aleq 2^{\frac{1}{2} j_{0}} 2^{(\alp+\frac{1}{2}) k} (\nrm{V}_{\dot{\calH}_{<L}^{0, \nu}}+ \nrm{V}_{\dot{\calH}_{<L}^{\ell, \nu}}),
\end{align*}
where on the first line we used the $L^{1} \hookrightarrow L^{2}$ Bernstein inequality and $\nrm{\breve{\eta}_{j_{0}}}_{L^{2}} \aleq 2^{\frac{1}{2} j_{0}}$. The above bound is acceptable in the range $\set{(j, k) : \abs{j - j_{0}} \leq c_{0} + 5,\, k < -j_{0}-5}$.

\smallskip
\noindent {\it Case 1, Step~2.}
It remains to prove \eqref{eq:w-Pk-high}, \eqref{eq:w-Pk-low}, and \eqref{eq:w-Pk-far}. We start with the following bound for the kernel $\tilde{K}_{k}$ of $\tilde{P}_{k}$: for $\abs{j - j_{0}} > c_{0} + 5$ and any $N \geq 0$,
\begin{equation} \label{eq:Kk}
	\abs{\breve{\eta}_{j_{0}}(y) \tilde{K}_{k}(y - y') \eta_{j}(y)} \aleq_{N} 2^{k} 2^{-N(\max\set{j, j_{0}} + k)}. 
\end{equation}
Indeed, \eqref{eq:Kk} follows from the bound for $\tilde{K}_{k}$ and the simple fact that $\abs{y - y'} \aeq 2^{\max\set{j, j_{0}}}$ if $\abs{y} \aeq 2^{j_{0}}$, $\abs{y'} \aeq 2^{j}$ and $\abs{j - j_{0}} > c_{0} + 5$. 

As a result, we have $\nrm{\breve{\eta}_{j_{0}} \tilde{P}_{k} (\eta_{j} V)}_{L^{\infty}} \aleq_{N} 2^{k} 2^{-N(\max\set{j, j_{0}} + k)} \nrm{\eta_{j} V}_{L^{1}}$. By two applications of H\"older's inequality, as well as $\nrm{\eta_{j} V}_{L^{2}} \aleq 2^{-\nu j} \nrm{V}_{\dot{\calH}_{<L}^{0, \nu}}$, we have
\begin{align*}
	2^{(\nu + \ell) j_{0}} 2^{(\alp + \ell) k}  \nrm{\breve{\eta}_{j_{0}} \tilde{P}_{k}(\eta_{j} V)}_{L^{2}}
	&\aleq_{N} 2^{-\alp j_{0}} 2^{(\nu + \frac{1}{2} + \alp + \ell) (j_{0} + k)} 2^{(-\nu+\frac{1}{2}) (j+k)} 2^{-N(\max \set{j, j_{0}} + k)} \nrm{V}_{\dot{\calH}_{<L}^{0, \nu}}.
\end{align*}
By choosing $N$ to be appropriately large, \eqref{eq:w-Pk-high} and \eqref{eq:w-Pk-low} in the case $j > - k + C_{0}$ follow (note that in the last case, $j_{0} < - k - 5$, so $j > j_{0} + C_{0} + 5$). To treat the remaining cases in \eqref{eq:w-Pk-low}, namely $k < - j_{0} -5$ and $j \leq -k + C_{0}$, we simply use the H\"older and Bernstein inequalities to estimate
\begin{align*}
	2^{(\nu + \ell) j_{0}} 2^{(\alp + \ell) k}  \nrm{\breve{\eta}_{j_{0}} \tilde{P}_{k}(\eta_{j} V)}_{L^{2}}
	&\aleq 2^{(\nu + \frac{1}{2} + \ell) j_{0}} 2^{(\alp + \ell) k}  \nrm{P_{k} (\eta_{j} V)}_{L^{\infty}} 
	\aleq 2^{(\nu + \frac{1}{2} + \ell) j_{0}} 2^{(\alp + 1 + \ell) k}  \nrm{\eta_{j} V}_{L^{1}} \\
	&\aleq 2^{-\alp j_{0}} 2^{(\nu + \frac{1}{2} + \alp + \ell) (j_{0} + k)} 2^{(-\nu +\frac{1}{2}) (j+k)} \nrm{V}_{\dot{\calH}_{<L}^{0, \nu}}.
\end{align*}
Finally, to prove \eqref{eq:w-Pk-far}, we first split
\begin{equation*}
\eta_{\geq \log_{2} L} V = \sum_{j : \log_{2} L \leq j < \log_{2} L + c_{0}+10} \eta_{j} V + \eta_{\geq \log_{2} L + c_{0} + 10} V.
\end{equation*}
Observe that the contribution of the first term can be treated using \eqref{eq:w-Pk-high} and \eqref{eq:w-Pk-low}. For the remaining piece, thanks to the spatial separation between the supports of $\breve{\eta}_{j_{0}}$ and $\eta_{\geq \log_{2} L + c_{0} + 10}$, \eqref{eq:Kk} implies $\nrm{\breve{\eta}_{j_{0}} \tilde{P}_{k} (\eta_{\geq \log_{2} L + c_{0} + 10} V)}_{L^{\infty}} \aleq_{N} 2^{\frac{1}{2} k} 2^{-\frac{N}{2} (\log_{2} L + k)} \nrm{\eta_{\geq \log_{2} L + c_{0} + 10} V}_{L^{2}}$. Hence, by H\"older's inequality, 
\begin{align*}
	2^{(\nu + \ell) j_{0}} 2^{(\alp + \ell) k}  \nrm{\breve{\eta}_{j_{0}} \tilde{P}_{k}(\eta_{\geq \log_{2} L + c_{0} + 10} V)}_{L^{2}}
	&\aleq_{N} 2^{-\alp j_{0}} 2^{(\nu + \frac{1}{2} + \alp + \ell) (j_{0} + k)} 2^{-(\frac{N}{2} + \nu)(\log_{2} L + k)} \nrm{V}_{\dot{\calH}_{<L}^{0, \nu}}.
\end{align*}
By choosing $N$ appropriately, \eqref{eq:w-Pk-far} follows.

\smallskip
\noindent {\it Case 2, Step~1.}
In this case, $L \aeq_{c_{0}} 2^{j_{0}}$. We will use the following bound to treat the ``nearby input'' case: for $j < \log_{2} L$,
\begin{align}
	2^{(\nu + \ell) j_{0}} 2^{(\alp + \ell) k}  \nrm{\breve{\eta}_{\geq j_{0}} \tilde{P}_{k} (\eta_{j} V)}_{L^{2}} &\aleq 2^{-\alp j_{0}} 2^{-c \abs{j + k}} 2^{-c \abs{j_{0} - j}} \nrm{V}_{\dot{\calH}_{<L}^{0, \nu}}. \label{eq:w-Pk-near}
\end{align}
We defer their proofs for a moment and prove \eqref{eq:calH-w} and \eqref{eq:calH-comm-w} assuming \eqref{eq:w-Pk-near}.

\smallskip
\noindent {\it Case~2, Step~1.(a).}
As before, to prove \eqref{eq:calH-w}, we expand
\begin{equation*}
	\varpi \calH \rd_{y}^{\ell} V
	= \sum_{j, k : 2^{j} < L} 2^{(\nu + \ell) j_{0}} 2^{(\alp + \ell) k} \breve{\eta}_{\geq j_{0}} \tilde{P}_{k} (\eta_{j} V)
	+ \sum_{k} 2^{(\nu + \ell) j_{0}} 2^{(\alp + \ell) k} \breve{\eta}_{\geq j_{0}} \tilde{P}_{k} (\eta_{\geq \log_{2} L} V) 
	=: \mathrm{I}_{near} + \mathrm{I}_{far}.
\end{equation*}
This time, the term $\mathrm{I}_{near}$ can be treated using \eqref{eq:w-Pk-near}, so it only remains to estimate $\mathrm{I}_{far}$. By almost orthogonality (in case $\alp + \ell= 0$) or interpolation (in case $\alp + \ell > 0$), it is straightforward to prove
\begin{align*}
\nrm{\sum_{k} 2^{(\nu + \ell) j_{0}} 2^{(\alp + \ell) k} \breve{\eta}_{\geq j_{0}} \tilde{P}_{k} (\eta_{\geq \log_{2} L} V)}_{L^{2}}
\aleq 2^{(\nu + \ell) j_{0}} \nrm{\eta_{\geq \log_{2}} V}_{L^{2}}^{\frac{1 -\alp}{\ell+1}} \nrm{\rd_{y}^{\ell+1} (\eta_{\geq \log_{2}} V)}_{L^{2}}^{\frac{\ell+\alp}{\ell+1}},
\end{align*}
which is acceptable.

\smallskip
\noindent {\it Case~2, Step~1.(b).}
To prove \eqref{eq:calH-comm-w}, we expand
\begin{align*}
	[\varpi, \calH] \rd_{y}^{\ell} V
	&= \sum_{k \geq - j_{0} - 5} 2^{(\nu + \ell) j_{0}} 2^{\alp k} [\breve{\eta}_{\geq j_{0}}, \tilde{P}_{k}] \rd_{y}^{\ell} (\eta_{\geq \log_{2} L - c_{0} - 10} V) \\
	&\peq + \sum_{k < - j_{0} -5} 2^{(\nu + \ell) j_{0}} 2^{\alp k} [\breve{\eta}_{\geq j_{0}}, \tilde{P}_{k}] \rd_{y}^{\ell} (\eta_{\geq \log_{2} L - c_{0} - 10} V)  \\
	&\peq + \sum_{j, k : j < \log_{2} L - c_{0} - 10} 2^{(\nu + \ell) j_{0}} 2^{\alp k} [\breve{\eta}_{\geq j_{0}}, \tilde{P}_{k}] \rd_{y}^{\ell} (\eta_{j} V) 
	=: \mathrm{I}' + \mathrm{II}' + \mathrm{III}'.
\end{align*}

For $\mathrm{I}'$, we argue as in term $\mathrm{I}$ in Case~1, Step~1.(b) using the commutator structure and bound
\begin{align*}
	\nrm{\mathrm{I}'}_{L^{2}}
	\aleq 2^{-\alp j_{0}} \sum_{k \geq -j_{0} - 5} 2^{(\nu+\ell) j_{0}} 2^{-(1-\alp) (j_{0}+k)} \nrm{\rd_{y}^{\ell} (\eta_{\geq \log_{2} L - c_{0} - 10} V)}_{L^{2}} 
	\aleq 2^{-\alp j_{0}} (\nrm{V}_{\dot{\calH}_{<L}^{0, \nu}} + \nrm{V}_{\dot{\calH}_{<L}^{\ell, \nu}}).
\end{align*}

For $\mathrm{II}'$, we expand the commutator expression and write
\begin{align*}
\mathrm{II}'
&= \sum_{k < - j_{0} -5} 2^{(\nu + \ell) j_{0}} 2^{\alp k} \breve{\eta}_{\geq j_{0}} \tilde{P}_{k} \rd_{y}^{\ell} (\eta_{\geq \log_{2} L - c_{0} - 10} V) 
- \sum_{k < - j_{0} -5} 2^{(\nu + \ell) j_{0}} 2^{\alp k} \tilde{P}_{k}(\breve{\eta}_{\geq j_{0}} \rd_{y}^{\ell} (\eta_{\geq \log_{2} L - c_{0} - 10} V) ).
\end{align*}
We simply bound $2^{\alp k} \aleq 2^{- \alp j_{0}}$, then use almost orthogonality to estimate
\begin{equation*}
	\nrm{\mathrm{II}'}_{L^{2}}
	\aleq 2^{- \alp j_{0}} 2^{(\nu + \ell) j_{0}} \nrm{\rd_{y}^{\ell} (\eta_{\geq \log_{2} L - c_{0} - 10} V)}_{L^{2}}
	\aleq 2^{- \alp j_{0}} (\nrm{V}_{\dot{\calH}_{<L}^{0, \nu}} + \nrm{V}_{\dot{\calH}_{<L}^{\ell, \nu}}).
\end{equation*}

For $\mathrm{III}'$, we simply observe that, by the support properties of $\breve{\eta}_{\geq j_{0}}$ and $\eta_{j}$,
\begin{equation*}
	\mathrm{III}' = \sum_{j, k : j < \log_{2} L - c_{0} - 10} 2^{(\nu + \ell) j_{0}} 2^{\alp k} \breve{\eta}_{\geq j_{0}} \tilde{P}_{k} \rd_{y}^{\ell} (\eta_{j} V),
\end{equation*}
which can be treated using \eqref{eq:w-Pk-near}.

\smallskip
\noindent {\it Case 2, Step~2.}
It remains to prove \eqref{eq:w-Pk-near}. Recall that $L \aeq_{c_{0}} 2^{j_{0}}$ and $j < \log_{2} L$. We first split
\begin{equation*}
\breve{\eta}_{\geq j_{0}} = \sum_{j : \log_{2} L \leq j < \log_{2} L + c_{0}+10} \eta_{j} \breve{\eta}_{\geq j_{0}} + \eta_{\geq \log_{2} L + c_{0} + 10} \breve{\eta}_{\geq j_{0}}.
\end{equation*}
Each summand in the first sum obeys the same properties as $\breve{\eta}_{j}$, so its contribution may be treated by \eqref{eq:w-Pk-high} and \eqref{eq:w-Pk-low}. Let us abbreviate the second term by $\breve{\eta}_{\geq \log_{2} L + c_{0} + 10}$. Thanks to the spatial separation property, \eqref{eq:Kk} implies $\nrm{\breve{\eta}_{\geq \log_{2} L + c_{0} + 10} \tilde{P}_{k} (\eta_{j} V)}_{L^{2}} \aleq_{N} 2^{\frac{1}{2} k} 2^{-\frac{N}{2} (j_{0} + k)} \nrm{\eta_{j} V}_{L^{\infty}}$. Hence, by H\"older's inequality,
\begin{align*}
	2^{(\nu + \ell) j_{0}} 2^{(\alp + \ell) k}  \nrm{\breve{\eta}_{\geq \log_{2} L + c_{0} + 10} \tilde{P}_{k}(\eta_{j} V)}_{L^{2}}
	&\aleq_{N} 2^{-\alp j_{0}} 2^{(\nu + \alp + \ell - \frac{N}{2}) (j_{0} + k)} 2^{(-\nu + \frac{1}{2})(j + k)} \nrm{V}_{\dot{\calH}_{<L}^{0, \nu}}.
\end{align*}
By choosing $N$ appropriately, \eqref{eq:w-Pk-near} follows. \qedhere
\end{proof}

\section{Estimates on \texorpdfstring{$U$}{U}} \label{sec:Uest}
In this section, we improve the bootstrap assumptions in Lemma~\ref{lem:main} that involve $U$.

\subsection{Non-top-order forcing term estimates}
To prepare for the ensuing analysis, we establish some bounds for the forcing term involving $\calH$ and $\calL$.
\begin{lemma} \label{lem:forcing}
Assume the hypotheses of Lemma~\ref{lem:main}. Then the following $L^{2}$ bounds hold:
\begin{align} 
e^{-\mu s} \nrm*{\rd_{y}^{(j)} \calH(U)(s, \cdot)}_{L^{2}} &\leq C A e^{-\mu s} & \hbox{for } 1 \leq j \leq 2k+2, \label{eq:forcing-H-j-L2}\\
e^{- s}  \nrm*{\rd_{y}^{(j)} \calL\left(U+e^{(b-1)s} \kpp\right)}_{L^{2}} 
&\leq C (1+\kpp_{0}) e^{-(2+(j-\frac{3}{2})b) s}  & \hbox{for } j \geq 0. \label{eq:forcing-L-j-L2}
\end{align}
Moreover, the following pointwise bounds hold:
\begin{align}
e^{- \mu s} \nrm*{\calH(U)(s, \cdot)}_{L^{\infty}(-4, 4)} &\leq C (A+\kpp_{0}) e^{-\mu_{0} s}, \label{eq:forcing-H-0} \\
e^{-\mu s} \nrm*{\rd_{y}^{(j)} \calH(U)(s, \cdot)}_{L^{\infty}} &\leq C A e^{-\mu s} & \hbox{for } 1 \leq j \leq 2k+1, \label{eq:forcing-H-j}\\
e^{- s}  \nrm*{\rd_{y}^{(j)} \calL\left(U+e^{(b-1)s} \kpp\right)}_{L^{\infty}} 
&\leq C (1+\kpp_{0}) e^{-(2+(j-1)b) s}  & \hbox{for } j \geq 0. \label{eq:forcing-L-j}
\end{align}
\end{lemma}
These bounds will be useful for the proof of essentially all non-top-order estimates (with the sole exception of the weighted $L^{2}$-Sobolev bound \eqref{eq:ibst-w-1}). On the other hand, to estimate $\rd_{y}^{2k+3} U$, we shall rely instead on the dispersive/dissipative property of $\Gmm$/$\Ups$ and appropriate commutator estimates; see Lemmas~\ref{lem:high} and \ref{lem:weighted2} below.

\begin{proof}
Bounds \eqref{eq:forcing-L-j-L2} and \eqref{eq:forcing-L-j} for $e^{-s} \calL$ immediately follow by combining \eqref{eq:calL-L2} and \eqref{eq:calL-Linfty}, respectively, with Lemma~\ref{lem:UL2}. On the other hand, to prove \eqref{eq:forcing-H-j-L2} and \eqref{eq:forcing-H-j}, note that, by \eqref{eq:bst1}, \eqref{eq:bst3} (for $\abs{y} \leq 1$) and \eqref{eq:bst-w-1} (for $\abs{y} \geq 1$), we have $\nrm{U'}_{L^{\infty}} \leq C$ and 
\begin{equation} \label{eq:U'-L2}
\nrm{U'}_{L^{2}} \leq C A.
\end{equation}
Recall also that $\nrm{\rd_{y}^{2k+3} U}_{L^{2}} \leq 2 A$ by \eqref{eq:bst2}. Therefore, \eqref{eq:forcing-H-j-L2} and \eqref{eq:forcing-H-j} follow from \eqref{eq:calH-L2} and \eqref{eq:calH-Linfty}, respectively. 

To prove the remaining bound \eqref{eq:forcing-H-0}, we apply Lemma~\ref{lem:mult-ker}. By introducing a smooth partition of unity (in the variable $y'$) subordinate to $\set{\abs{ y' } < 16} \cup \set{8 <\abs{ y' } < 2 e^{bs}} \cup \set{e^{bs}<\abs{ y' }}$, we may estimate
\begin{align*}
e^{- \mu s} |\calH(U)&(s, y)|
\leq e^{-\mu s} \abs*{\int_{-\infty}^{\infty} K(y-y') \rd_{y} U(s, y') \, \ud y'} \\
&\aleq e^{-\mu s} \int_{\abs{y'} \leq 16} \abs{y-y'}^{-\max\set{\alp, \bt, 0}} \abs{\rd_{y} U(s, y')} \, \ud y'
+ e^{-\mu s} \int_{8 \leq \abs{y'} \leq 2 e^{b s}} \abs{y'}^{-\max\set{\alp, \bt, 0}-1} \abs{U(s, y')} \, \ud y'  \\
&\pleq + e^{-\mu s} \int_{\abs{y'} \geq e^{b s}} \abs{y'}^{-\max\set{\alp, \bt, 0}-1} \abs{U + e^{(b-1)s} \kpp}(s, y') \, \ud y' \\
&\aleq e^{-\mu s} \nrm{\rd_{y} U}_{L^{\infty}}
+ e^{-\mu s} \nrm{\abs{y}^{-\frac{1}{2k+1}} U}_{L^{\infty}(8, 2e^{bs})} \int_{8 \leq \abs{y'} \leq 2 e^{b s}} \abs{y'}^{-\max\set{\alp, \bt, 0}-1+\frac{1}{2k+1}}  \, \ud y'   \\
&\pleq + e^{-\mu s} e^{-(\frac{1}{2}+\max\set{\alp, \bt, 0}) b s} \nrm{U+e^{(b-1)s} \kpp}_{L^{2}} .
\end{align*}
In the second inequality, we used integration by parts and the property that $\abs{y - y'} \aeq \abs{y'}$ for $\abs{y'} \geq 8$ (since $\abs{y} \leq 4$).  When $\max\set{\alp, \bt, 0} \neq \frac{1}{2k+1}$, by \eqref{eq:bst1}, \eqref{eq:bst3}, \eqref{eq:bst-w-1}, \eqref{eq:U-ptwise-sharp} and \eqref{eq:U-L2},
\begin{align*}
e^{- \mu s} \abs*{\calH(U)(s, y)}
&\aleq e^{-\mu s} \left(1 + \max\set{1, e^{(\frac{1}{2k+1}- \max\set{\alp, \bt, 0}) bs}} + e^{-(\frac{1}{2}+\max\set{\alp, \bt, 0}) b s} e^{(\frac{3}{2} b -1) s} \right) \aleq e^{-\min \set{\mu, \frac{2k-1}{2k}} s}.
\end{align*}
Indeed, note that $\mu = 1 - \max\set{\alp, \bt, 0} b$. Therefore,
\begin{align*}
	-\mu + (\tfrac{1}{2k+1}  - \max\set{\alp, \bt, 0}) b
	&= - 1 + \max\set{\alp, \bt, 0} b + \tfrac{1}{2k} - \max\set{\alp, \bt, 0} b
	= - \tfrac{2k-1}{2k}, \\
	-\mu - (\tfrac{1}{2} + \max\set{\alp, \bt, 0}) b + (\tfrac{3}{2} b - 1)
	&= - 1 + \max\set{\alp, \bt, 0} b - \tfrac{1}{2} b + \max\set{\alp, \bt, 0} b + \tfrac{3}{2} b - 1 \\
	&= - 2 + b = \tfrac{-4k + 2k + 1}{2k} = - \tfrac{2k-1}{2k}.
\end{align*}
Since $\mu_{0} = \min\set{\mu, \frac{2k-1}{2k}}$ in this case, \eqref{eq:forcing-H-0} follows. On the other hand, in case $\max\set{\alp, \bt, 0} = \frac{1}{2k+1}$, the same computation applies except that the second term is bounded instead by $C A bs$; however, such a modification is acceptable since $0 < \mu_{0} < \frac{2k-1}{2k}$.
\end{proof}

\subsection{Estimates on \texorpdfstring{$\p_y U$}{U'}}
Next, we proceed to obtain pointwise estimates for the low derivative $\p_yU$ using the method of characteristics. It is important that the bounds proved below (in particular, items 3--4) are independent of $A$. On the other hand, we need not obtain sharp pointwise bounds for $\rd_{y} U$ at this point, as they would follow from the weighted $L^{2}$-Sobolev bounds \eqref{eq:ibst-w-1} and \eqref{eq:ibst-w-high} (cf.~derivation of \eqref{eq:U'-ptwise-sharp} in the proof of Theorem~\ref{thm:main}).

\begin{lemma}\label{lem:U1}
There exist $\eps_{0}$, $A$, $y_0$, $\gmm$, $\sgm_{0}$ such that if the initial data conditions~\eqref{eq:idu}--\eqref{eq:idw4} are satisfied and the bootstrap assumptions \eqref{eq:bst1}--\eqref{eq:bootstraplow2} hold for $s \in [\sgm_{0}, \sgm_{1}]$, then the following conclusions hold.
\begin{enumerate}
\item For all $s \in [\sgm_{0}, \sgm_{1}]$, and all $|y|> y_0$, we have $(by + (1+e^s\tau_s)U(s,y)) y > 0$ (the flow is repulsive).
\item \eqref{eq:ibst1} and \eqref{eq:ibst3} hold.
\item For $|y| \geq 4$ and $s \in [\sgm_{0}, \sgm_{1}]$, we have
\begin{equation}\label{eq:unifdissip}
{-}U'(s,y) \leq \frac 56.
\end{equation}
\item There exist $C, r > 0$ independent of $A, y_0, \gmm$ such that, for all $y \in \R$, and for all $s \in [\sgm_{0}, \sgm_{1}]$, we have
\begin{equation}\label{eq:decayup}
|\p_y U(s,y)| \leq C \max\set*{\frac{1}{(1+|y|)^r}, A e^{-r s}}.
\end{equation}
\end{enumerate}
\end{lemma}

\begin{proof}[Proof of Lemma~\ref{lem:U1}]
We prove each item in order.

\smallskip
\noindent {\it Proof of~1.} We need to show that there exists a choice of $y_0$ and $\sgm_{0}$ (depending on $A$, see Remark~\ref{rem:deps}) such that the following holds in the interval $s \in [\sgm_{0}, \sgm_{1}]$:
$$
U(s,y) > - by \qquad \text{ if } y > y_0, \qquad U(s,y) < - \frac{b}{1+e^s \tau_s}y \qquad \text{ if } y < -y_0,
$$
Let us focus on the first claim, the second being analogous. By the fundamental theorem of calculus and the bootstrap assumptions, we have that
$$
U(s,y) - U(s,0) = \int_0^y U'(s, \underline y) d \underline y \geq - y_0(1+2y_0) - (y- y_0) \left( 1-\frac{y^{2k}_0}{4} \right).
$$
The claim follows by choosing $y_0$ appropriately, and then by choosing $\sgm_{0}$ appropriately to control the factor containing the modulation parameter $\tau$ (recalling that the bootstrap assumptions~\eqref{eq:bootstrapmod} hold true).

\smallskip
\noindent {\it Proof of~2.} We now show that we can restrict to $y_0$ small and $\sgm_{0} \gg1$ (depending on $A$, see Remark~\ref{rem:deps}) such that the following inequalities hold true:
\begin{align}
&\Big| U(s, \pm y_0) {\pm} y_0 \mp \frac{1}{2k+1} y_0^{2k+1} \Big| \leq \frac{y_0^{2k+1}}{4{(2k+1)}} \qquad \text{for all } s \in [\sgm_{0}, \sgm_{1}] \label{eq:idU1}\\
&\Big| U({\sgm_{0}},  y) {+} y {-} \frac{1}{2k+1} y^{2k+1} \Big| \leq \frac{y^{2k+1}}{4{(2k+1)}} \qquad \text{for all } y \in (-1/4,- y_0) \cup (y_0, 1/4),\label{eq:idU2}\\
&\Big| U'(s, \pm y_0) {+} 1 {-} y_0^{2k} \Big| \leq \frac{y_0^{2k}}{4} \qquad \text{for all } s \in [\sgm_{0}, \sgm_{1}] , \label{eq:id1}\\
&\Big| U'({\sgm_{0}}, y) {+} 1{-} y^{2k} \Big| \leq \frac{y^{2k}}{4} \qquad \text{for all } y \in (-1/4, y_0) \cup (y_0, 1/4),\label{eq:id2}\\
&{U'(\sgm_{0}, y) \geq -1 + \frac{y_0^{2k}}{2} \qquad \text{for all } |y| \geq y_0}.\label{eq:id3}
\end{align}
Note that the above inequalities imply (on initial data and at $|y| = y+0$) bounds which are strict improvements of  \eqref{eq:ibst1} and \eqref{eq:ibst3}. For instance, from \eqref{eq:id2} it follows that $0 \geq U'(\sgm_0, y) \geq -1 + \frac{3}{4} y^{2k} \geq -1 + \frac{3}{4} y_0^{2k}$, for all $y \in (-1/4,- y_0) \cup (y_0, 1/4)$, and similarly for the bounds~\eqref{eq:id1}, as well as~\eqref{eq:id3}.

Bounds \eqref{eq:idU2}, \eqref{eq:id2} and~\eqref{eq:id3} follow easily from our choice of initial data at $\sgm_{0}$ and the expression for the profile (with the choice~\eqref{eq:tchoice} concerning the $(2k+1)$\textsuperscript{st} derivative at $y = 0$), upon choosing $y_0$ to be small and consequently $\sgm_0$ to be large. The proofs of~\eqref{eq:idU1} and ~\eqref{eq:id1} are similar, so we will just focus on showing~\eqref{eq:id1}. Using Taylor expansion with integral remainder and Sobolev embedding, we have
\begin{equation}\label{eq:expansiony0}
\Big| U'(s, y_0) +1 - y^{2k}_0 - \sum_{j=1}^{2k+1} W^{(j)}(s, 0) \frac{y_0^{j-1}}{(j-1)!} \Big| \leq C e^{\frac A 2} y_0^{2k+1}
\end{equation}
Note that the two principal terms on the LHS of the previous estimate arise from our choice of the profile in display~\eqref{eq:tchoice}. We then notice that the coefficients $ W^{(j)}(s, 0)$, $j = 0, \ldots, 2k$ decay by \eqref{eq:bootstrapmod} (some of them are identically zero by~\eqref{eq:orth-cond}), and $|W^{(2k+1)}(s,0)| \leq 4$ by the bootstrap assumptions~\eqref{eq:bootstraplow2}. We then first choose $y_0$ to be small, and then choose $\sgm_{0}$ accordingly to be large to conclude that~\eqref{eq:id1} holds true (see Remark~\ref{rem:deps}).

Recall the equation for $U'$ from \eqref{eq:commutu-1}; we arrange the equation as follows:
\begin{equation} \label{eq:uprime}
\begin{aligned}
&\p_s U' +  U' {+} (U')^2 + (b  y +(1+{e^s\tau_s}) U)\p_y U' \\
&= \left( e^{bs} \xi_{s} - (1+e^{s} \tau_{s}) e^{(b-1)s} \kpp\right) U''
- { e^{s} \tau_{s} (U')^{2}}  \\
&\peq
+(1+e^{s} \tau_{s}) \left( e^{-\mu s} \calH (U') + {e^{-s}}\rd_{y} \calL(U+e^{(b-1)s} \kpp) \right) =: E^{(1)}.
\end{aligned}
\end{equation}
By \eqref{eq:bst1}--\eqref{eq:bst3} for $U'$, \eqref{eq:bst2} and \eqref{eq:U'-L2} for $U''$, \eqref{eq:bootstrapmod} for the modulation terms and Lemma~\ref{lem:forcing} for the forcing terms, we have
\begin{equation} \label{eq:E1bound}
	\nrm{E^{(1)}}_{L^{\infty}} \leq C A e^{-\gmm s},
\end{equation}
where we take $A$ sufficiently large compared to $\kpp_{0}$ if necessary.

We will first focus on showing~\eqref{eq:ibst3}. We now define Lagrangian coordinates for the flow of equation~\eqref{eq:uprime}. The flow can either start from a point on the half line $s = \sgm_{0}, y \geq y_0$ or from a point on the half line $s \geq \sgm_{0}, y= y_0$. Distinguishing between these two cases, we consider Lagrangian maps $X_1(s, \tilde y)$, $X_2(s, \tilde s)$ which are defined by solving the following initial value problems ($\tilde s$ and $\tilde y$ are the Lagrangian parameters):
\begin{align}
&\p_s X_1(s,\tilde y) = b X_1(s,\tilde y) {+}{(1+e^s\tau_s) }U(s, X_1(s,\tilde y)) , \qquad X_1(\sgm_{0},\tilde y) = \tilde y, \qquad \text{for all } \tilde y \geq y_0, \label{eq:lag1}\\
&\p_s X_2(s,\tilde s) = b X_2(s,\tilde s) {+}{(1+e^s\tau_s)}U(s, X_2(s,\tilde s)) , \qquad X_2(\tilde s , {\tilde s}) = y_0, \qquad \text{for all } \tilde s \geq \sgm_{0}. \label{eq:lag2}
\end{align}

We now rewrite equation~\eqref{eq:uprime} in Lagrangian coordinates, to obtain, letting $\tilde U'$ be the composition $\tilde U' (s, \tilde y) = U(s, X_1(s, \tilde y))$: 
\begin{equation}\label{eq:uprimew}
\p_s \tilde U' +  \tilde U' + (\tilde U')^2 =  {E^{(1)}(s, X_1(s, \tilde y)).}
\end{equation}

Let us for a moment assume that there exists $\check s$ such that $ \tilde U'(\check s, \tilde y) \geq -\frac{1}{2}$. Integrating \eqref{eq:uprimew}, and using the bound~\eqref{eq:E1bound}, it is then easy to show that, upon choosing $\sigma_0$ to be large (depending only on $y_0$ and $A$), $ |\tilde U'( s, \tilde y)| \leq \frac{3}{4}$ for all $s \geq \check s$. This implies~\eqref{eq:ibst3} on the half line $\set{y > 0}$ in this case.

Hence it suffices to show~\eqref{eq:ibst3} under the additional assumption that  $\tilde U'(\check s, \tilde y) \leq -\frac{1}{2}$. By the repulsivity of the flow, we always have that $|X_1(s, \tilde y)| \geq |y_0|$ for $s \geq \sgm_0$. Hence, for $s \geq \sgm_0$, we can apply the bootstrap assumption~\eqref{eq:bst3} (i.e., $\tilde U' \geq -1 + \frac{y_0^{2k}}{4}$), which yields the following bound:
\begin{equation}\label{eq:lagU1}
-  \tilde U' - (\tilde U')^2 \geq \left( 1 - \frac{y_0^{2k}}{4}\right)\frac{y_0^{2k}}{4}\geq \frac{y_0^{2k}}{8},
\end{equation}
by virtue of the fact that $y_0$ is chosen to be small. We now choose $\sigma_0$ to be sufficiently large for the following inequality to be valid\footnote{Note that this choice can be made independently of $\sgm_1$.}, in view of~\eqref{eq:E1bound}:
\begin{equation}\label{eq:lagU2}
\int_{\sgm_0}^{\sgm_{1}} |E^{(1)}(s, X_1(s, \tilde y))| \ud \, s \leq \frac{y_0^{2k}}{8}.
\end{equation}

We then integrate~\eqref{eq:uprimew}, taking into account the bounds~\eqref{eq:lagU1} and~\eqref{eq:lagU2}, as well as the bound~\eqref{eq:E1bound}. We deduce:
\begin{equation}\label{eq:linftyuprime1}
\begin{aligned}
 |U'(s,X_1(s, \tilde y))| \leq 1 -\frac 12 y_0^{2k}, \qquad \text{for } \sgm_{0} \leq s \leq \sgm_{1}, \ |\tilde y|\geq y_0.
\end{aligned}
\end{equation}
Essentially repeating the same argument for the $X_2$ trajectories, we have
\begin{equation}\label{eq:linftyuprime2}
\begin{aligned}
 |U'(s,X_2(s, \tilde s))| \leq 1 -\frac 12 y_0^{2k}, \qquad \text{for } \sgm_{0} \leq \tilde s \leq \sgm_{1}, \  \tilde s \leq s \leq \sgm_{1}.
\end{aligned}
\end{equation}
Combining the bounds~\eqref{eq:linftyuprime1} and~\eqref{eq:linftyuprime2} yields~\eqref{eq:ibst3} on the half line $\set{y > 0}$. Arguing in the same manner on $\set{y < 0}$, we obtain \eqref{eq:ibst3}.

Finally, \eqref{eq:ibst1} follows from~\eqref{eq:ibst3} and Taylor expansion about $y = 0$, following a similar reasoning as in inequality~\eqref{eq:expansiony0}.

\smallskip
\noindent {\it Proof of~3.} To prove \eqref{eq:unifdissip}, we first show a quantitative lower bound on the time the Lagrangian trajectories $X_{1}$, $X_{2}$ stay in $y \in [-4, 4]$. Recall the definition of Lagrangian coordinates $X_1$ and $X_2$ in~\eqref{eq:lag1} and~\eqref{eq:lag2}. Define now $\tilde U_i$, for $i \in \{1,2\}$ as follows: 
$$
\tilde U_1 (s, \tilde y) = U(s, X_1(s, \tilde y)), \qquad \tilde U_2 (s, \tilde s) = U(s, X_2(s, \tilde s)).
$$
We then have the following coupled system for $(\tilde U_i, X_i)$:
\begin{equation}
\begin{aligned}
&\p_s \tilde U_i = (b-1)\tilde U_i  + {E^{(0)}(s, X_i)},\\
&\p_s X_i = b X_i {+} {(1+e^s\tau_s)}\tilde U_i.
\end{aligned}
\end{equation}
Here, we recall that
\begin{equation}
\begin{aligned}
&E^{(0)} = \left(- e^{s} \tau_{s} U + e^{bs} \xi_{s} - (1+e^{s} \tau_{s}) e^{(b-1)s} \kpp\right) U' -  e^{(b-1)s} \kpp_{s} \\
&\peq
+(1+e^{s} \tau_{s}) \left( e^{-\mu s} \calH (U) + {e^{-s}} \calL(U+e^{(b-1)s} \kpp) \right)
\end{aligned}
\end{equation}
We let $A_i = X_i + \tilde U_i$, which diagonalizes the system up to a perturbative term on the RHS:
\begin{equation}
\begin{aligned}
\p_s \tilde U_i - (b-1)\tilde U_i &= {E^{(0)}(s, X_i)},\\
\p_s A_i - b A_i &= e^s \tau_s \tilde U_i +  {(1+e^s\tau_s)} {E^{(0)}(s, X_i)} .
\end{aligned}
\end{equation}
Let us now specify to the case $i = 1$. The RHS of the previous display is perturbative as long as $|X_1(s, \tilde y)| \leq 4$. Indeed, using Lemma~\ref{lem:forcing}, as well as the fact that $|U(s,y)| \leq 2|y|$ for all $y \in [-4,4]$, and the bootstrap assumptions~\eqref{eq:bootstrapmod}, we have
 \begin{align*}
&e^{-\mu s}  \nrm{\calH(U)}_{L^{\infty}(-4, 4)} \leq C A e^{-\mu_{0} s}, \\
&e^{-s} \nrm{\calL(U + e^{(b-1)s}\kpp)}_{L^{\infty}}  \leq C (1+\kpp_{0}) e^{(-2+b)s},\\
&e^s |\tau_s \tilde U_i(s, \tilde y)| \leq 8 e^{-\gamma s}.
\end{align*}
Let us first restrict to the case $y_0 \leq |\tilde y| \leq 1$. We integrate the previous display between $\sgm_{0}$ and $s$, to obtain, using the fact that the RHS is perturbative, 
\begin{equation}\label{eq:UAsmall}
\begin{aligned}
\Big| U(s, X_1(s, \tilde y)) e^{-(b-1)(s-\sgm_{0})} - U(\sgm_{0}, \tilde y) \Big| \leq |U(\sgm_{0}, \tilde y)|,\\
\Big| A_1(s, X_1(s, \tilde y))e^{-b(s-\sgm_{0})} - A_1(\sgm_{0}, \tilde y) \Big| \leq \frac 1 2|A_1(\sgm_{0}, \tilde y)|.
\end{aligned}
\end{equation}
Note that  the above inequalities hold by choosing $\sgm_0$ large as a function of $y_0$. Indeed, due to our choice of initial data for $U$ at $\sgm_{0}$, by possibly choosing $\sgm_0$ large, we have that, for all $\tilde y \in [-1, -y_0] \cup [y_0, 1]$, $|U(\sgm_{0}, \tilde y)| \geq \frac{1}{2} |U(\sgm_{0},  y_0)| \geq \frac{y_0}{4}$. It then suffices to choose $\sgm_{0}$ to be large so that, for all $\tilde y \in [-1, -y_0] \cup [y_0, 1]$,
$$
\int_{\sgm_0}^{\sgm_{1}}e^{-(b-1)(s-\sgm_{0})}  |E^{(0)}(s, X_1( s, \tilde y))| \ud \, s \leq  \frac{y_0}{4}.
$$
This shows that the first inequality in~\eqref{eq:UAsmall} holds, up to choosing $\sgm_0$ based on $y_0$. A similar reasoning holds for the second inequality in~\eqref{eq:UAsmall}.

Display~\eqref{eq:UAsmall} implies, due to our choice of initial data at $\sgm_{0}$,
\begin{equation}\label{eq:growthlag}
\begin{aligned}
| U(s, X_1(s, \tilde y))| \leq 2 |\tilde y| e^{(b-1)(s-\sgm_{0})},\\
| A_1(s, X_1(s, \tilde y))|\leq \frac{3}{2} (|\tilde y|^{2k+1} + {e^s \tau_s}) e^{b(s-\sgm_{0})}.
\end{aligned}
\end{equation}
We then deduce that, for all times $s \in [\sgm_{0}, -2k \log(|\tilde y|)+ \sgm_{0}]$, using the bootstrap assumptions~\eqref{eq:bootstrapmod} to show that the term $e^s \tau_s$ is perturbative,
\begin{equation}\label{eq:traj1}
\begin{aligned}
&| U(s, X_1(s, \tilde y))| \leq 2,\\
&| A_1(s, X_1(s, \tilde y))|\leq 2 - \frac {1}{4} \implies |X_1(s,\tilde y)| \leq 4.
\end{aligned}
\end{equation}
The reasoning for $X_2$ is completely analogous, and we deduce that, for all times $s \in [\tilde s, -2k \log(y_0)+\tilde s]$, 
\begin{equation}\label{eq:traj2}
\begin{aligned}
&| U(s, X_2(s, \tilde s))| \leq 2,\\
& |X_2(s,\tilde s)| \leq 4.
\end{aligned}
\end{equation}
We are now in shape to do an $L^\infty$ estimate for $U'$ in the near region. Let us recall the relevant equation:
\begin{equation}
\p_s U' +  U' +(U')^2 + (b  y +{(1+e^s\tau_s)}U)\p_y U'= {E^{(1)}}.
\end{equation}
In Lagrangian coordinates with respect to $X_1$, the above equation reads, letting $ \tilde U' = U'(s, X_1(s,\sgm_{0}))$:
\begin{equation}\label{eq:uplag}
\p_s \tilde U' +  \tilde U' + (\tilde U')^2 =  {E^{(1)}}(s, X_1(s, \tilde y)).
\end{equation}
Let us now suppose by contradiction that, for all $s \in [\sgm_{0}, \sgm_{0} - 2k \log (|\tilde y|)]$, 
\begin{equation}\label{eq:contrad}
U'(s, X_1(s,\tilde y)) \leq -\frac 45.
\end{equation}
Combined with the bootstrap assumption~\eqref{eq:bst3}, display~\eqref{eq:contrad} yields the bound $|\tilde U' +(\tilde U')^2| \leq \frac{y_0^{2k}}{8}$, which in turn implies, since the RHS of~\eqref{eq:uplag} is perturbative, upon choosing $\sgm_{0}$ to be larger, and recalling the definition of $\mu'$ from~\eqref{eq:E1bound},
\begin{equation}
 \frac{\p_s\tilde U'}{\tilde U' +(\tilde U')^2 } + 1  \leq \frac {{\mu'}}{4} e^{-s\frac{{\mu'}}2}.
\end{equation}
This implies
\begin{equation}
 \p_s \log\Big( \frac{ -\tilde U'}{1 + \tilde U'} \Big)+ 1  \leq   \frac {{\mu'}}{4} e^{-s\frac{{\mu'}}2}.
\end{equation}
By integration, since $\int_{\sgm_{0}}^{\infty} \frac{\mu'}{4} e^{-s \frac{\mu'}{2}} \, \ud s < \log 2$, and denoting $Q =  2 \frac{-\tilde U'(\sgm_{0}, \tilde y)}{1+ \tilde U'(\sgm_{0}, \tilde y)} > 0$, we have 
\begin{equation}
U'(s, X_1(s, \sgm_{0})) > \frac{-Qe^{-(s-\sgm_{0})}}{1+Q e^{-(s-\sgm_{0})}} .
\end{equation}
We now calculate this expression at $s = s_*  = \sgm_{0} - 2k \log(|\tilde y|)$.
We notice that, for $\eps_{0}$ sufficiently small, $\tilde U'(\sgm_{0}, \tilde y) \geq -1$, and $1+\tilde U'(\sgm_{0}, \tilde y) \geq \frac 1 2 \tilde y^{2k}$, hence $ 0\leq Q \leq 4\tilde y^{-2k}$. Since $e^{-(s_*-\sgm_{0})} = \tilde y^{2k}$, it follows that
\begin{equation}
U'(s_*, X_1(s_*, \sgm_{0}))> -\frac{4}{4+1}= - \frac 4 5.
\end{equation}
This contradicts~\eqref{eq:contrad}, and yields, for all $\tilde y$ such that $y_0 \leq |\tilde y| \leq 1$, the existence of $s^{(1)} \in [\sgm_{0},  s_*]$ such that 
\begin{equation}\label{eq:1co}
U'( s^{(1)}, X_1(s^{(1)}, \tilde y)) \geq -\frac 45.
\end{equation}
Moreover, in the case $|\tilde y| \geq 1$, the existence of $s^{(1)}$ in the conditions above follows immediately by our choice of initial data. Indeed, by possibly choosing $\sgm_0$ to be larger and $\eps_{0}$ smaller,
\begin{equation}\label{eq:11co}
U'( \sgm_0, X_1(\sgm_0, \tilde y)) \geq -\frac 45
\end{equation}
for all $\tilde y$ such that $|\tilde y| \geq 1$.

A completely analogous reasoning shows that for all $\tilde s \leq \sgm_{1} + 2k \log(|y_0|)$, there exists $s^{(2)} \in [\tilde s, \tilde s - 2k \log(|y_0|)]$ such that we have
\begin{equation}\label{eq:2co}
U'( s^{(2)}, X_2(s^{(2)}, \tilde s)) \geq- \frac 45.
\end{equation}
We now combine the bounds~\eqref{eq:1co}, \eqref{eq:11co} and~\eqref{eq:2co} with the bounds on the Lagrangian trajectory~\eqref{eq:traj1} and~\eqref{eq:traj2} to obtain the existence of $s^{(1)}, s^{(2)}$ (depending resp. on $\tilde y$ and $y_0$), such that
\begin{equation}
\begin{aligned}
&U'(s^{(1)}, X_1(s^{(1)}, \tilde y)) \geq -\frac 45, \qquad |X_1(s^{(1)}, \tilde y)| \leq 4, \qquad \text{for all } \tilde y: y_0 \leq |\tilde y| \leq 4, \label{eq:4bd}\\
&U'( s^{(2)}, X_2(s^{(2)}, \tilde s)) \geq -\frac 45, \qquad |X_2(s^{(2)}, \tilde s)| \leq 4.
\end{aligned}
\end{equation}

We now repeat the reasoning in part 2.~integrating equation~\eqref{eq:uprimew}, with the difference that the starting time of integration is now $s^{(1)}$ (resp. $s^{(2)}$), and we use the bounds~\eqref{eq:4bd}. Recall that the RHS of~\eqref{eq:uprimew} is perturbative everywhere by~\eqref{eq:E1bound}. Repeating the same argument on $\set{y <  0}$, we deduce \eqref{eq:unifdissip} valid for $|y|\geq 4$, $\sgm_{0} \leq s \leq \sgm_{1}$.

\smallskip
\noindent {\it Proof of~4.} We only consider the case of the half-space $\set{y > 0}$ in detail, as the other case is dealt with similarly. We define Lagrangian trajectories $X_{1}$ and $X_{2}$ as in \eqref{eq:lag1} and \eqref{eq:lag2}, respectively, but now with $y_{0}$ replaced by $y = 4$. By $U(s, 0) = 0$, \eqref{eq:ibst1}, \eqref{eq:ibst3} and \eqref{eq:unifdissip}, we may deduce the simple bound $\abs{U(s, y)} \leq y$ by integration. By \eqref{eq:bootstrapmod} and taking $\sgm_{0}$ sufficiently large, it follows that
\begin{equation} \label{eq:exp0}
	\rd_{s} X_{i} = b X_{i} + (1+e^{s} \tau_{s}) \tilde{U}_{i} \quad \imp \quad \frac{1}{2} \rd_{s} X_{i}^{2} \leq 2 b X_{i}^{2}
	\quad \imp \quad \abs{X_{i}} \leq e^{2b(s - \tilde{s})} \abs{\tilde{y}},
\end{equation}
where $(\tilde{s}, \tilde{y}) = (\sgm_{0}, \tilde{y})$ or $(\tilde{s}, 4)$ when $i = 1$ or $2$, respectively.

Next, we again recall the equation for $U'$ in Lagrangian coordinates, which yields, letting $\tilde U' = U'(s, X_{i}(s, \tilde y))$,
\begin{equation}
\p_s \tilde U' +  \tilde U' + (\tilde U')^2 =  E^{(1)}(s, X_{i}(s, \tilde y)).
\end{equation}
Multiplying by $\tilde U'$ and using~\eqref{eq:unifdissip} plus Cauchy--Schwarz, we have (recall that $0 < \gmm < 1$)
\begin{equation}\label{eq:udecay}
\frac{1}{2} \p_s (\tilde U')^2\leq  - \frac \gmm {10}  (\tilde U')^2 + (E^{(1)}(s, X_{i}(s, \tilde y)))^2.
\end{equation}
Recalling \eqref{eq:E1bound} for $E^{(1)}$, it follows that
\begin{equation*}
	\rd_{s} \left( e^{\frac{\gmm}{10} s} (\tilde{U}')^{2} \right) \leq C A e^{(\frac{\gmm}{10}-\gmm) s}.
\end{equation*}
Integrating this equation, we obtain that, for a constant $C > 0$,
\begin{equation}\label{eq:exp1}
|\tilde U'| \leq C  e^{- \frac{\gmm}{10} (s-\tilde{s})} \abs{U'(\tilde{s}, \tilde{y})} + C A e^{-\frac{\gmm}{10} s}.
\end{equation}
In case $i = 2$, we have $(\tilde{s}, \tilde{y}) = (\tilde{s}, 4)$, and the desired bound \eqref{eq:decayup} with $r = \frac{\gmm}{20 b}$ follows from \eqref{eq:unifdissip}, \eqref{eq:exp0} and \eqref{eq:exp1}. On the other hand, in case $i = 1$, the desired bound with $r = \frac{\gmm}{20 b}$ follows from \eqref{eq:U'-ini}, \eqref{eq:exp0} and \eqref{eq:exp1}, where we simply bound $e^{-\frac{\gmm}{10}(s-\sgm_{0})} \abs{\tilde{y}}^{-\frac{2k}{2k+1}} \leq C \abs{y}^{-r}$ and $e^{-\frac{\gmm}{10}(s-\sgm_{0})} e^{-\sgm_{0}} \eps_{0} \leq C A e^{-r s}$. \qedhere
\end{proof}

\subsection{Estimates on \texorpdfstring{$\p_y^2 U$}{U''}}
As a preparation for closing the bootstrap assumption on $\nrm{\rd_{y}^{2k+3} U}_{L^{2}}$, we first prove a uniform bound for $\nrm{\rd_{y}^{2} U}_{L^{2}}$. The key ingredients are the method of characteristics in the region close to $y = 0$, as well as the a-priori pointwise bounds for $U'$ proved in Lemma~\ref{lem:U1}.

\begin{lemma}\label{lem:uii}
There exist $\eps_{0}$, $A$, $y_0$, $\gmm$, $\sgm_{0}$ such that if the initial data conditions~\eqref{eq:idu}--\eqref{eq:idw4} are satisfied and the bootstrap assumptions \eqref{eq:bst1}--\eqref{eq:bootstraplow2} hold for $s \in [\sgm_{0}, \sgm_1]$, then the following inequality holds true for all $s \in [\sgm_0, \sgm_1]$:
\begin{equation}\label{eq:estuii}
\| \p^{2}_y U(s, \cdot)\|_{L^2} \leq C.
\end{equation}
Here, $C > 0$ is a constant independent of $A$, $y_0$, and $\sgm_{0}$.
\end{lemma}

\begin{proof}[Proof of Lemma~\ref{lem:uii}]

We begin by recalling \eqref{eq:commutu} with $j = 2$ for $U''$, which we rewrite as follows:
\begin{equation}\label{eq:uii}
\begin{aligned}
	&\rd_{s} U^{''} +  (1+b + 3U')U'' + (by+{(1+e^s \tau_s)}U) \rd_{y} U'' \\
	& = -{ 3 e^{s} \tau_{s} U' U''}+ \left(e^{b s} \xi_{s} - (1+e^{s} \tau_{s}) e^{(b-1) s} \kpp \right) U'''\\
	&\peq + (1+e^{s} \tau_{s}) \left( e^{-\mu s} \calH (U'') + {e^{-s}} \rd_{y}^{2} \calL(U + e^{(b-1)s} \kpp) \right)=: E^{(2)}. 
\end{aligned}
\end{equation}
Upon restricting to small $y_0$ and consequently to large $\sgm_0$, the following properties are a consequence of Taylor expansion about $0$, the Taylor coefficients of the profile at $y =0$, and the hypotheses on initial data at $s = \sgm_{0}$ (exactly as in the reasoning around~\eqref{eq:expansiony0}):
\begin{align}
&\left| U''(s,  \pm y_0) \mp 2k y_0^{2k-1} \right|\leq  y_0^{2k-1}  \qquad \text{for all } s \geq \sgm_{0}, \label{eq:iid1}\\
&\left| U''(s,  y) - 2k y^{2k-1} \right|\leq  |y|^{2k-1} \qquad \text{for all } y \in (-1/4, y_0) \cup (y_0, 1/4). \label{eq:iid2}
\end{align}
We remark that, in order to ensure this condition, we first need to first choose $y_0$ to be small, and as a function of that, we need to choose $\sgm_0$ to be large, cf.~Remark~\ref{rem:deps}.

Let us recall the definition of the Lagrangian trajectories $X_1$ and $X_2$ from~\eqref{eq:lag1}--\eqref{eq:lag2}. Let us first focus on $X_1$, and we define $\tilde U'' := U''(s, X_1(s, \tilde y))$. Assume that $\tilde y$ is such that $y_0 \leq \tilde y \leq \frac{1}{4}$, the negative $\tilde y$ case being analogous. We easily have, from~\eqref{eq:uii} (using moreover~\eqref{eq:bst3}) that
\begin{equation}
\rd_{s} \tilde U'' +  (b -2)\tilde U''   \leq E^{(2)}(s, X_1(s, \tilde y)). 
\end{equation}
We now notice, similarly to~\eqref{eq:E1bound}, that
\begin{equation}\label{eq:E2bound}
	\nrm{E^{(2)}(s, y)}_{L^{\infty}} \leq C A e^{-\gmm s}.
\end{equation}
Hence, restricting to $\sgm_0$ possibly larger (and treating the terms on the RHS as perturbative), we have, for $s \geq \sgm_{0}$, upon integration
\begin{equation}\label{eq:uiibound}
0 \leq U''(s, X_1(s, \tilde y)) \leq (2k+2) e^{(2-b)(s-\sgm_{0} )}\tilde y^{2k-1}.
\end{equation}
We notice that we have the following easy consequence of~\eqref{eq:bst3} in the region $y \in (y_0, 2)$, $s \in [\sgm_0, \sgm_{1}]$ (upon choosing $\sgm_0$ large as a function of $A$ and $y_0$):
\begin{equation}
U(s,y) \geq  -(b-1) y.
\end{equation}
We now go back to the definition of the Lagrangian trajectories~\eqref{eq:lag1}, and we immediately deduce that, for $\tilde y \in (y_0, \frac{1}{4})$, as long as $|X_1(s, \tilde y)| \leq 2$,
\begin{equation}
X_1(s, \tilde y) \geq e^{(b-1)(s-\sgm_0)} \tilde y.
\end{equation}
We now have that, for $s$ such that $s - \sgm_{0}= - 2k \log \tilde y$, the following holds:
\begin{equation}\label{eq:laguii}
X_1(s, \tilde y) \geq  1.
\end{equation}
Inequality~\eqref{eq:uiibound}, now gives that, for all $s$ such that $s - \sgm_{0} \leq - 2k \log \tilde y$, and all $\tilde y \in (y_0, \frac {1} {4})$:
\begin{equation}\label{eq:uiifinal1}
|U''(s, X_1(s, \tilde y))| \leq (2k+2) e^{-(2-b)2k \log \tilde y}\tilde y^{2k-1} \leq 2k +2,
\end{equation}
since $(2-b)2k = 2k -1$.

Similarly, turning now our attention to $X_2$, we have, for $s \geq \tilde s$, and $\tilde y$ such that $y_0 \leq \tilde y \leq \frac{1}{4}$:
\begin{equation}\label{eq:uiibound2}
0 \leq U''(s, X_2(s, \tilde s)) \leq (2k+2) e^{(b-2)(s-\tilde s )}y^{2k-1}_0.
\end{equation}
We moreover have that, for $s$ such that $s - \tilde s = - 2k \log \tilde y_0$, the following holds:
\begin{equation}\label{eq:laguii2}
X_2(s, \tilde s) \geq y_0 e^{-(b-1)2k \log y_0} = 1.
\end{equation}
Combined with inequality~\eqref{eq:uiibound2}, this implies that, for all $s$ such that $s - \sgm_{0} \leq - 2k \log y_1$,
\begin{equation}\label{eq:uiifinal2}
|U''(s, X_2(s, \tilde s))| \leq 2k +2.
\end{equation}
Combining previous inequalities \eqref{eq:laguii}, \eqref{eq:uiifinal1},  \eqref{eq:laguii2}, and \eqref{eq:uiifinal2}, we conclude that 
\begin{equation}\label{eq:bounds2interior}
|U''(s,y)| \leq 2k+2
\end{equation} for $|y| \leq \frac 14$, and $s \in [\sgm_0, \sgm_1]$. This concludes the bounds in the ``near'' region.

We now proceed to show an $L^2$ bound in the ``intermediate'' region: $y: \frac 1 4 \leq |y| \leq y_2$, where $y_2$ is chosen depending on $\zeta > 0$ using Lemma~\ref{lem:U1} in a way that, for all $y: |y| \geq y_2$, $s \in [\sgm_0, \sgm_1]$,
\begin{equation}
|U'(s,y)| \leq \zeta.
\end{equation}
We are going to first show a weighted $L^2$ estimate on $U''$, where the weight is exponentially decaying in $y$. Although the estimates for this part are carried out on the whole real line, one should think of them as just useful to the ``intermediate region'' ($|y| \leq y_2$). In the final part of the proof of this lemma, we are going to deal with the ``far'' region using the smallness arising from point 3.~in lemma~\ref{lem:U1}.

We now multiply equation~\eqref{eq:uii} by the weight $e^{- \lambda y}U''$, to obtain a weighted $L^2$ estimate in the region $|y| \geq \frac 14$. We integrate by parts on the set $S = [- 1/4, 1/4]^c$. We have, denoting by $\| f\|_{w} := \|e^{- \frac \lambda 2 y} f(y) \|_{L^2(S)}$,
\begin{equation}
\begin{aligned}
&\frac 12 \p_s \|U''(s,\cdot)\|^2_w + \int_S \big(1+ \frac b 2+ \frac 52 U'- \frac{1}{2} e^s \tau_s U'(s,y) \big) (e^{-\frac \lambda 2 y} U''(s,y))^2 \, \ud y\\
&+ \frac{\lambda}{2} \int_S (by+(1+e^s \tau_s) U(s,y)) (e^{-\frac \lambda 2 y} U''(s,y))^2  \, \ud y\\
& \leq 2e^{- \frac \lambda 4} \Big( \frac b 4 + \frac 14\Big) (2k+2)^2 + \|U''(s,\cdot)\|_w \|E^{(2)}(s,\cdot)\|_w ,
\end{aligned}
\end{equation}
where we used the bootstrap assumptions and the bounds \eqref{eq:bounds2interior} on $U''$ to control the boundary terms. Now, for $y \geq \frac 14$, possibly choosing $\sgm_0$ to be large, $by - U \geq \frac {b-1} 8$, and $|U'(s,y)| \leq 1$. Hence, it suffices to choose $\lambda$ to satisfy $(16)^{-1}\lambda(b-1)-2 \geq 1$ to obtain (after an application of the Cauchy--Schwarz inequality on the RHS, and restricting to $\sgm_0$ large):
\begin{equation}
\frac 12 \p_s \|U''(s,\cdot)\|^2_w + \|U''(s,\cdot)\|^2_w \leq 2(2k+2)^2 + \|E^{(2)}(s,\cdot)\|^2_w,
\end{equation}
Together with the assumptions on initial data, and the bounds~\eqref{eq:E2bound}, this implies  $\|U''(s,\cdot)\|^2_w \leq C_k$ for all $s \geq \sgm_{0}$, where $C_k$ is a constant depending only on $k$. Possibly redefining the constant $C_k$, we also have the unweighted bound
\begin{equation}\label{eq:L2ii}
\|U''(s,\cdot) \|_{L^2([-y_2, y_2])} \leq C_k.
\end{equation}
(recall the definition of $y_2$ from lemma~\ref{lem:U1}).

We finally perform estimates in the ``far away'' region $|y| \geq y_2$. Again we consider equation~\eqref{eq:uii}, we multiply by $U''$ and integrate by parts. We have, adding a multiple of the bound \eqref{eq:L2ii}, and letting $S_2 =\R \setminus [-y_2, y_2]$, possibly redefining the constant $C_k$,
\begin{equation}
\begin{aligned}
&\frac 12 \p_s \|U''(s,\cdot)\|^2_{L^2(\R)} + \int_{S_2} \big(1+ \frac b 2+ \frac 52 U'- \frac{1}{2} e^s \tau_s U'(s,y) \big) ( U''(s,y))^2 \, \ud y\\
& \leq C_k + \|U''(s,\cdot)\|_{L^2(\R)} \|E^{(2)}(s,\cdot)\|_{L^2(\R)} ,
\end{aligned}
\end{equation}
Recall now that, choosing $\zeta$ appropriately, in particular $|U'(s,y)| \leq \frac 14$ for $|y| \geq y_2$, so that  
$$
1+ \frac b 2+ \frac 52 U'- \frac{1}{2} e^s \tau_s U'(s,y) - \frac{1}{4}\geq \frac{1}{4}.
$$
Adding to the above inequality a multiple of $\| U''(s, \cdot)\|_{L^2([-y_2, y_2])}$, we have
\begin{equation}\label{eq:U2almost}
\frac 12 \p_s \|U''(s,\cdot)\|_{L^2(\R)}^2 +\frac{1}{4} \|U''(s,\cdot)\|_{L^2(\R)}^2  \leq  C_k +  4\|E^{(2)}(s,\cdot)\|^2_{L^2(\R)}.
\end{equation}
Finally, by \eqref{eq:bst1}--\eqref{eq:bst3} for $U'$, \eqref{eq:bst2} and \eqref{eq:U'-L2} for $U''$ and $U'''$, \eqref{eq:bootstrapmod} for the modulation terms and Lemma~\ref{lem:forcing} for the forcing terms, we have
\begin{equation} \label{eq:E2bound-L2}
	\nrm{E^{(2)}}_{L^{2}} \leq C A e^{-\gmm s}.
\end{equation}
Combining the above bounds with inequality~\eqref{eq:U2almost}, we obtain (possibly choosing $\sgm_0$ large as a function of $A$) \eqref{eq:estuii} on $\bbR$ as desired. \qedhere

\end{proof}

\subsection{Top order \texorpdfstring{$L^2$}{L2} estimate}
We are now ready to close the bootstrap assumption \eqref{eq:bst2} on the top order $L^{2}$ norm.
\begin{lemma}\label{lem:high}
There exist $\eps_{0}$, $A$, $y_0$, $\gmm$, $\sgm_{0}$ such that if the initial data conditions~\eqref{eq:idu}--\eqref{eq:idw4} are satisfied and the bootstrap assumptions \eqref{eq:bst1}--\eqref{eq:bootstraplow2} hold for $s \in [\sgm_{0}, \sgm_{1}]$, we have
\begin{equation} \label{eq:esthigh}
\|\rd_{y}^{2k+3} U(s, \cdot)\|_{L^2} \leq C,
\end{equation}
where $C > 0$ is a constant independent of $A$, $y_{0}$ and $\sgm_{0}$. Hence,~\eqref{eq:ibst2} holds.
\end{lemma}

\begin{proof}[Proof of Lemma~\ref{lem:high}]
Let $j = 2k+3$ and consider equation~\eqref{eq:commutu}. We multiply this equation by $\p_y^{(j)}U$ and integrate by parts. We obtain, using the fact that the each of the operators $e^{-s} \Gmm(e^{bs} D_{y}) e^{bs} \rd_{y}$ and $e^{-s} \Ups(e^{bs} D_{y})$ on the RHS (cf.~\eqref{eq:H-def}--\eqref{eq:L-def}) either is anti-symmetric or has a good sign,
\begin{align}
&\frac 12 \p_s \|U^{(j)}\|_{L^2(\R)}^2 +\int_{\R}\left( (1 -b) + jb+(j+1)(1+e^{s} \tau_{s})U' \right) (U^{(j)})^2  \,  \ud y \nonumber\\
& -\frac 12\int_{\R} (b  + (1+e^s \tau_s)U')( U^{(j)})^2 \, \ud y \label{eq:highenergy}\\
& \leq (1+e^{s} \tau_{s}) \|M^{(j)}\|_{L^2(\R)} \|U^{(j)}\|_{L^2(\R)}.\nonumber 
\end{align}
This implies
\begin{align}
&\frac 12 \p_s \|U^{(j)}\|_{L^2(\R)}^2 +\int_{\R}\left( \left(1 - \frac{3}{2} b\right) + jb+\left(j+\frac 1 2\right)(1+e^{s} \tau_{s})U' \right) (U^{(j)})^2 \, \ud y\\
& \leq \|M^{(j)}\|_{L^2(\R)} \|U^{(j)}\|_{L^2(\R)}.
\end{align}
Note the following inequalities, valid for $b > 2$, $a \geq 2$:
\begin{align}
&\|U^{(b)}\|_{L^2(\R)} \lesssim \|U''\|_{L^2(\R)}^{1-\theta_1} \|U^{(j)}\|_{L^2(\R)}^{\theta_1},\\
&\|U^{(a)}\|_{L^\infty} \lesssim \|U''\|_{L^2(\R)}^{1-\theta_2} \|U^{(j)}\|_{L^2(\R)}^{\theta_2},
\end{align}
where $\theta_1 = (b-2)/(j-2)$ and $\theta_2 = (a-3/2)/(j-2)$.

The general form of a term appearing in $M^{(j)}$ is $U^{(a)} U^{(b)}$, with $a,b \geq 1$ and $a + b = j+1$. We estimate, when $a \leq b$,
\begin{equation}
\int_{\R} |U^{(a)} U^{(b)} U^{(j)}| dy \leq C_j \|U''\|^{1-\theta_1-\theta_2}_{L^2(\R)}  \|U^{(j)}\|_{L^2(\R)}^{1+ \theta_1 + \theta_2} \leq C_j \|U^{(j)}\|_{L^2(\R)}^{2- \frac 1{2(j-2)}},
\end{equation}
where $C_j$ is a constant which depends only on $j$ (here, we used the bound $\|U ''\|_{L^2(\R)} \leq C$ from Lemma~\ref{lem:uii}). Combining the previous inequalities, and substituting $2k+3$ for $j$,
\begin{equation*}
\begin{aligned}
&\frac 12 \p_s \|U^{(2k+3)}\|_{L^2(\R)}^2 + \int_{\R}(j+\frac 1 2)e^{s} \tau_{s} U' (U^{(2k+3)})^2 \, \ud y\\
&+\left(\frac{b-1} 2 - C_k \|U^{(2k+3)}\|_{L^2(\R)}^{ -\frac 1{4k+2}} \right) \|U^{(2k+3)}\|_{L^2(\R)}^2 \leq 0.
\end{aligned}
\end{equation*}
Here, $C_k$ is a constant that depends only on $k$. It then follows, using the bootstrap assumptions~\eqref{eq:bootstrapmod}, as well as choosing $\sgm_0$ sufficiently large, that
\begin{equation}\label{eq:highint}
\frac 12 \p_s \|U^{(2k+3)}\|_{L^2(\R)}^2 +\left(\frac{b-1} 4 - C_k \|U^{(2k+3)}\|_{L^2(\R)}^{ -\frac 1{4k+2}} \right) \|U^{(2k+3)}\|_{L^2(\R)}^2 \leq 0.
\end{equation}
From this inequality and the assumptions on initial data it follows that, for all $s \geq \sgm_{0}$, the following bound is propagated:
\begin{equation} \label{eq:esthigh-pre}
\|U^{(2k+3)}\|_{L^2(\R)} \leq 2 \Big(\frac{4C_k}{b-1} \Big)^{4k}.
\end{equation}
For $A$ large, this proves the improved bound \eqref{eq:ibst2}.
\end{proof}

\subsection{Weighted \texorpdfstring{$L^{2}$}{L2} estimates on \texorpdfstring{$U$}{U}}
Finally, we improve the bootstrap assumptions \eqref{eq:bst-w-1} and \eqref{eq:bst-w-high} concerning weighted $L^{2}$ estimates on $\rd_{y} U$ and $\rd_{y}^{2k+3} U$, respectively.
\begin{lemma}\label{lem:weighted2}
There exist $\eps_{0}$, $A$, $y_0$, $\gmm$, $\sgm_{0}$ such that if the initial data conditions~\eqref{eq:idu}--\eqref{eq:idw4} are satisfied and the bootstrap assumptions \eqref{eq:bst1}--\eqref{eq:bootstraplow2} hold for $s \in [\sgm_{0}, \sgm_{1}]$, the improved inequalities~\eqref{eq:ibst-w-1} and~\eqref{eq:ibst-w-high} hold true. 
\end{lemma}

\begin{proof}[Proof of Lemma~\ref{lem:weighted2}]
We begin the proof with a basic, abstract computation. Consider a first-order operator of the form
\begin{equation*}
	\calT = \rd_{s} + v \rd_{y} + q.
\end{equation*}
We decompose $\calT$ into its anti-symmetric and symmetric parts, i.e., $\calT = \calT^{a} + \calT^{s}$, where
\begin{equation*}
	\calT^{a} = \tfrac{1}{2} (\calT - \calT^{\dagger}) = \rd_{s} + v \rd_{y} + \tfrac{1}{2} (\rd_{y} v), \qquad 
	\calT^{s} = \tfrac{1}{2} (\calT + \calT^{\dagger}) = q - \tfrac{1}{2} (\rd_{y} v).
\end{equation*}
Let $\varpi^{2} = \varpi^{2}(s, y)$ be a nonnegative weight. If we multiply $\calT V$ by $\varpi^{2} V$ and integrate over $[s_{1}, s_{2}] \times \bbR$, we have
\begin{equation} \label{eq:w-energy}
\begin{aligned}
	&\int_{s_{1}}^{s_{2}} \int (\calT V) (\varpi^{2} V) \, \ud y \ud s \\
	&= \frac{1}{2} \left. \int \varpi^{2} V^{2} \, \ud y \right|_{s=s_{1}}^{s_{2}} + \frac{1}{2} \int_{s_{1}}^{s_{2}} \int (\varpi^{2} \calT V) V  \, \ud y \ud s
	+ \frac{1}{2} \int_{s_{1}}^{s_{2}} \int V \calT^{\dagger} (\varpi^{2} V)  \, \ud y \ud s \\
	&= \frac{1}{2} \left. \int \varpi^{2} V^{2} \, \ud y \right|_{s=s_{1}}^{s_{2}} + \frac{1}{2} \int_{s_{1}}^{s_{2}} \int ([\varpi^{2}, \calT^{a}] V) V  \, \ud y \ud s
	+ \frac{1}{2} \int_{s_{1}}^{s_{2}} \int (\varpi^{2} \calT^{s} + \calT^{s} \varpi^{2}) V V  \, \ud y \ud s \\
	&= \frac{1}{2} \left. \int \varpi^{2} V^{2} \, \ud y \right|_{s=s_{1}}^{s_{2}} 
	+ \int_{s_{1}}^{s_{2}} \int  \left( - \rd_{s} \varpi - v \rd_{y} \varpi  -\frac{1}{2} (\rd_{y} v) \varpi + q \varpi \right) \varpi  V^{2}  \, \ud y \ud s.
\end{aligned}
\end{equation}
Motivated by the last expression, we introduce the operator
\begin{equation*}
	\breve{\calT} := \rd_{s} \varpi + v \rd_{y} \varpi  + \frac{1}{2} (\rd_{y} v) \varpi - q \varpi.
\end{equation*}
As in the proof of Lemma~\ref{lem:mult-w}, we introduce a nonnegative smooth partition of unity $\set{\eta_{j}}_{j \in \bbZ}$ on $\bbR$ subordinate to the open cover $\set{A_{j} = \set{y \in \bbR : 2^{j-3} < \abs{y} < 2^{j+2}}}_{j \in \bbZ}$, and also the shorthand $\eta_{\geq j} = \sum_{j' \geq j} \eta_{j'}$. 

\smallskip
\noindent {\it Step~1.} Our goal is to prove \eqref{eq:ibst-w-1} concerning $U'$. In view of Lemma~\ref{lem:U1}, it suffices to bound the expression 
\begin{equation} \label{eq:U'-w-outside}
	\sup_{j \in \bbZ, \, 2^{j} < e^{bs}}\left( \int_{2^{j-1} < \abs{y} < 2^{j}}  \eta_{\geq i_{0}}^{2} \big(  \abs{y}^{\frac{1}{b}-\frac{1}{2}} U'(s, y) \big)^{2}\, \ud y \right)^{\frac{1}{2}}
	+ e^{(1-\frac{1}{2} b ) s} \left( \int_{\abs{y} > \frac{e^{bs}}{2}} \eta_{\geq i_{0}}^{2} U'(s, y)^{2} \, \ud y \right)^{\frac{1}{2}},
\end{equation}
where the cutoff parameter $i_{0} > 0$ is to be determined below. As a preparation for the proof, we introduce 
\begin{equation} \label{eq:T1}
	\calT_{1} = \rd_{s} + v \rd_{y} + q_{1},
\end{equation}
where
\begin{equation} \label{eq:T1-coeff}
	 v = \left( b y + (1 + e^{s} \tau_{s}) U - e^{b s} \xi_{s} + (1+e^{s} \tau_{s}) e^{(b-1) s} \kpp \right), \qquad
	q_{1} = \left(1 + (1 + e^{s} \tau_{s}) U' \right).
\end{equation}
Recall, from \eqref{eq:commutu-1}, that $U'$ obeys
\begin{equation}\label{eq:u1recall}
\begin{aligned}
	\calT_{1} U' & = (1+e^{s} \tau_{s}) \left( e^{-\mu s} \calH (U') + e^{-s} \rd_{y} \calL(U + e^{(b-1)s} \kpp) \right).
\end{aligned}
\end{equation}
Let $j_{1} := \lfloor b \sgm_{0} \log 2  \rfloor$. For each integer $j \leq j_{1}$, we introduce the weight
\begin{equation} \label{eq:w-1}
	\varpi_{j}(s, y) = e^{(1 - \frac{1}{2} b) (s - \sgm_{0})} \varpi_{j, 0}(y e^{-b(s-\sgm_{0})}), \qquad
	\varpi_{j, 0} = \begin{cases}
	2^{(\frac{1}{b} - \frac{1}{2}) j} \eta_{j}  & \hbox{ for } j < j_{1}, \\
	2^{(\frac{1}{b} - \frac{1}{2}) j_{1}} \eta_{\geq j_{1}} & \hbox{ for } j = j_{1}.
	\end{cases}
\end{equation}
Observe that each $\varpi_{j}$ solves the equation
\begin{equation*}
	\left( \rd_{s} + by \rd_{y} + \frac{1}{2} b - 1 \right) \varpi_{j} = 0,
\end{equation*}
where, as we will see, the LHS is a good approximation of $\breve{\calT}_{1} \varpi_{j}$. 

We are now ready to begin the proof of \eqref{eq:ibst-w-1} in earnest. For each $j \leq j_{1}$, we apply \eqref{eq:w-energy} with $\varpi = \eta_{\geq i_{0}} \varpi_{j}$, which leads to
\begin{align*}
	\frac{1}{2} \int \eta_{\geq i_{0}}^{2} \varpi_{j}^{2} U'(s, y)^{2} \, \ud y
	= \frac{1}{2} \int \eta_{\geq i_{0}}^{2} \varpi_{j}^{2} U'(\sgm_{0}, y)^{2} \, \ud y + \int_{\sgm_{0}}^{s} \int \left(\breve{\calT}_{1} (\eta_{\geq i_{0}} \varpi_{j}) U' + \eta_{\geq i_{0}} \varpi_{j} \calT_{1} U' \right) \eta_{\geq i_{0}} \varpi_{j} U' \ud y \ud s.
\end{align*}
Thanks to \eqref{eq:w-1}, if we take the supremum in $j \leq j_{1}$, then the contribution of the LHS is equivalent to \eqref{eq:U'-w-outside}, whereas that of the first term on the RHS is bounded by a constant $C$ as a consequence of \eqref{eq:U-w-ini}. It is then straightforward to derive the estimate
\begin{equation*}
	\sup_{s \in [\sgm_{0}, \sgm_{1}]} \sup_{j \leq j_{0}} \left( \int \eta_{\geq i_{0}}^{2} \varpi_{j}^{2} U'(s, y)^{2} \, \ud y \right)^{\frac{1}{2}}
	\leq C + C \sup_{j \leq j_{0}} \int_{\sgm_{0}}^{\sgm_{1}} \left( \nrm{\eta_{\geq i_{0}} \varpi_{j} \calT_{1} U'}_{L^{2}} + \nrm{\breve{\calT}_{1} (\eta_{\geq i_{0}} \varpi_{j}) U'}_{L^{2}}\right) \, \ud s.
\end{equation*}
We claim that, for some $c > 0$ and $\sgm_{0}$ sufficiently large depending on $A$,
\begin{equation*}
	\sup_{j \leq j_{0}} \int_{\sgm_{0}}^{\sgm_{1}} \left( \nrm{\eta_{\geq i_{0}} \varpi_{j} \calT_{1} U'}_{L^{2}} + \nrm{\breve{\calT}_{1} (\eta_{\geq i_{0}} \varpi_{j}) U'}_{L^{2}}\right) \, \ud s
	\aleq 2^{\frac{3}{2} i_{0}} + 2^{-c i_{0}} A + e^{-c \sgm_{0}} A.
\end{equation*}
From this claim, \eqref{eq:ibst-w-1} would follow by taking $i_{0}$, $A$, and $\sgm_{0}$ large enough (in this order).

To bound the contribution of $\eta_{\geq i_{0}} \varpi_{j} \calT_{1} U'$, we use \eqref{eq:u1recall}. Using \eqref{eq:bootstrapmod} for $e^{s} \tau_{s}$, \eqref{eq:calH-w} in Lemma~\ref{lem:mult-w} for $e^{-\mu s} \calH(U')$ (with $\varpi = \eta_{\geq i_{0}} \varpi_{j}$, $\ell = 0$ and $\nu = \frac{1}{2} - \frac{1}{b}$) with \eqref{eq:bst-w-1}--\eqref{eq:bst-w-high}, and \eqref{eq:forcing-L-j-L2} for $e^{-s} \rd_{y} \calL(U + e^{(b-1) s} \kpp)$, we have
\begin{align*}
	\nrm{\eta_{\geq i_{0}} \varpi_{j} \calT_{1} U'}_{L^{2}}
	&\aleq (1 + e^{-\gmm s} A) \left( e^{-\mu s} A + 2^{(\frac{1}{b} - \frac{1}{2}) j} e^{(1-\frac{1}{2} b)(s - \sgm_{0})} (1+\kpp_{0}) e^{-(2-\frac{1}{2} b) s} \right) \\
	&\aleq  (1 + e^{-\gmm s} A) \left( e^{-\mu s} A + (1+\kpp_{0}) e^{-s} \right),
\end{align*}
where on the last line, we used the fact that $2^{j} \aleq 2^{j_{1}} \aeq e^{b \sgm_{0}}$. The integral of the last expression on $[\sgm_{0}, \sgm_{1}]$ is clearly acceptable.

Next, we bound the contribution of $\breve{\calT}_{1} (\eta_{\geq i_{0}} \varpi_{j}) U'$. We compute
\begin{equation*}
	\breve{\calT}_{1} (\eta_{\geq i_{0}} \varpi_{j}) = 
	v \eta_{\geq i_{0}}' \varpi_{j} + \eta_{\geq i_{0}} (v - by) \varpi_{j}' - \frac{1}{2} \eta_{\geq i_{0}} (1+e^{s} \tau_{s}) U' \varpi_{j}.
\end{equation*}
To proceed, note that, by the transport property of $\varpi_{j}$, the $s$-support of $\nrm{\eta_{\geq i_{0}}' \varpi_{j}}_{L^{\infty}}$ is at most of length $O(1)$ (independent of $i_{0}$ and $j$). Moreover, by \eqref{eq:decayup} and $U(s, 0) = 0$,
\begin{equation} \label{eq:decayup-int}
	\abs{U'(s, y)} + \abs{y}^{-1} \abs{U(s, y)} \aleq \max\set { (1+\abs{y})^{-r}, A e^{-r s} }.
\end{equation}
Then, using also \eqref{eq:bootstrapmod},
\begin{equation*}
	\abs{v - by} = \abs{(1 + e^{s} \tau_{s}) U - e^{b s} \xi_{s} + (1+e^{s} \tau_{s}) e^{(b-1) s} \kpp}
	\aleq \abs{y} \max\set { (1+\abs{y})^{-r}, A e^{-r s} } + e^{-\gmm s} A.
\end{equation*}
Finally, observe that $\eta_{\geq i_{0}} \varpi_{j}$ and $\eta_{\geq i_{0}} \abs{y} \varpi_{j}'$ are supported in $\set{ \abs{y} \ageq 2^{j} e^{b(s-\sgm_{0})}} \cap \set{\abs{y} \ageq 2^{i_{0}}}$ and are bounded by $C 2^{(\frac{1}{b} - \frac{1}{2}) j} e^{(1 - \frac{1}{2} b)(s-\sgm_{0})} \mathbf{1}_{> 2^{i_{0}-5}}(2^{j} e^{b(s-\sgm_{0})})$, where $\mathbf{1}_{> 2^{i_{0}-5}}$ is the characteristic function of $\set{\abs{y} > 2^{i_{0} - 5}}$. Putting all these together, we may arrive at 
\begin{align*}
\int_{\sgm_{0}}^{\sgm_{1}} \nrm{\breve{\calT}_{1} (\eta_{\geq i_{0}} \varpi_{j}) U'}_{L^{2}} \, \ud s
&\aleq 2^{(\frac{1}{b}+\frac{1}{2}) i_{0}} \nrm{U'}_{L^{2}(\supp \eta_{\geq i_{0}}')} \\
&\pleq + \int_{\sgm_{0}}^{\sgm_{1}} \left(2^{-r j} e^{-rb(s-\sgm_{0})} \mathbf{1}_{> 2^{i_{0}-5}}(2^{j} e^{b(s-\sgm_{0})}) + A e^{-r s} \right)\nrm{U}_{\dot{\calH}_{<e^{bs}}^{1}} \, \ud \sgm \\
&\aleq 2^{(\frac{1}{b}+1) i_{0}} + \left(2^{-r i_{0}} + A e^{-r \sgm_{0}} \right) A,
\end{align*}
where we used \eqref{eq:ibst1}, \eqref{eq:ibst3} and \eqref{eq:bst-w-1} on the last line. This proves the desired claim.

\smallskip
\noindent {\it Step~2.} Now we turn to the proof of \eqref{eq:ibst-w-high} concerning $U^{(2k+3)}$. To simplify the notation, let us write $\calU = U^{(2k+3)}$. In view of Lemma~\ref{lem:high}, it suffices to bound the expression 
\begin{equation} \label{eq:U-high-w-outside}
	\sup_{j \in \bbZ, \, 2^{j} < e^{bs}} \left( \int_{2^{j-1} < \abs{y} < 2^{j}} \eta_{\geq i_{0}}^{2} \big( \abs{y}^{\frac{1}{b} + 2k+\frac{3}{2}}  \calU(s, y)\big)^{2} \, \ud y \right)^{\frac{1}{2}}
	+ e^{(1+(2k+\frac{3}{2}) b) s} \left( \int_{\abs{y} > \frac{e^{bs}}{2}} \eta_{\geq i_{0}}^{2} \calU(s, y)^{2} \, \ud y \right)^{\frac{1}{2}},
\end{equation}
where the cutoff parameter $i_{0} > 0$ is to be determined below. Like in Step~1, we introduce
\begin{equation} \label{eq:T-high}
	\calT_{2k+3} = \rd_{s} + v \rd_{y} + q_{2k+3},
\end{equation}
where $v$ is as before and
\begin{equation} \label{eq:T-high-coeff}
	q_{2k+3} = \left(1 + (2k+2) b + (2k+4) (1 + e^{s} \tau_{s}) U' \right).
\end{equation}
Recall, from \eqref{eq:commutu}, that $\calU$ obeys
\begin{equation}\label{eq:u-high-recall}
\begin{aligned}
	\calT_{2k+3} \calU & = - (1+e^{s} \tau_{s}) M^{(2k+3)} + (1+e^{s} \tau_{s}) \left( e^{-\mu s} \calH (\calU) + e^{-s} \rd_{y}^{2k+3} \calL(U + e^{(b-1)s} \kpp) \right).
\end{aligned}
\end{equation}
Let $j_{1} := \lfloor b \sgm_{0} \log 2  \rfloor$ and for each integer $j \leq j_{1}$, we introduce the weight
\begin{equation}
	\varpi_{j}(s, y) = e^{(1 + (2k+\frac{3}{2}) b) (s - \sgm_{0})} \varpi_{j, 0}(y e^{-b(s-\sgm_{0})}), \qquad
	\varpi_{j, 0} = \begin{cases}
	2^{(\frac{1}{b} + 2k+\frac{3}{2}) j} \eta_{j}  & \hbox{ for } j < j_{1}, \\
	2^{(\frac{1}{b} + 2k+\frac{3}{2}) j_{1}} \eta_{\geq j_{1}} & \hbox{ for } j = j_{1}.
	\end{cases}
\end{equation}

For each $j \leq j_{1}$, we apply \eqref{eq:w-energy} with $\varpi = \eta_{\geq i_{0}} \varpi_{j}$, which leads to
\begin{align*}
	\frac{1}{2} \int \eta_{\geq i_{0}}^{2} \varpi_{j}^{2} \calU(s, y)^{2} \, \ud y
	= \frac{1}{2} \int \eta_{\geq i_{0}}^{2} \varpi_{j}^{2} \calU(\sgm_{0}, y)^{2} \, \ud y + \int_{\sgm_{0}}^{s} \int \left(\breve{\calT}_{2k+3} (\eta_{\geq i_{0}} \varpi_{j}) U' + \eta_{\geq i_{0}} \varpi_{j} \calT_{2k+3} \calU \right) \eta_{\geq i_{0}} \varpi_{j} \calU \ud y \ud s.
\end{align*}
To avoid the derivative loss, we further write out the contribution of $\eta_{\geq i_{0}} \varpi_{j} \calT_{2k+3} \calU$ as follows:
\begin{align*}
	\int \eta_{\geq i_{0}} \varpi_{j} \calT_{2k+3} \calU \eta_{\geq i_{0}} \varpi_{j} \calU \, \ud y
	&= \int (1+e^{s} \tau_{s}) e^{-\mu s} \eta_{\geq i_{0}} \varpi_{j} \calH(\calU) \eta_{\geq i_{0}} \varpi_{j} \calU \, \ud y \\
	&\peq + \int (1+e^{s} \tau_{s}) e^{-s} \eta_{\geq i_{0}} \varpi_{j} \rd_{y}^{2k+3} \calL(U + e^{(b-1) s} \kpp) \eta_{\geq i_{0}} \varpi_{j} \calU\, \ud y \\
	&\peq - \int (1+e^{s} \tau_{s}) \eta_{\geq i_{0}} \varpi_{j} M^{(2k+3)} \eta_{\geq i_{0}} \varpi_{j} \calU \, \ud y \\
	&= (1+e^{s} \tau_{s}) e^{-\mu s} \int \calH (\eta_{\geq i_{0}} \varpi_{j} \calU) \eta_{\geq i_{0}} \varpi_{j} \calU \, \ud y \\
	&\peq + \int (1+e^{s} \tau_{s}) e^{-\mu s} [\eta_{\geq i_{0}} \varpi_{j}, \calH] \calU \eta_{\geq i_{0}} \varpi_{j} \calU \, \ud y \\
	&\peq + \int (1+e^{s} \tau_{s}) e^{-s} \eta_{\geq i_{0}} \varpi_{j} \rd_{y}^{2k+3} \calL(U + e^{(b-1) s} \kpp) \eta_{\geq i_{0}} \varpi_{j} \calU\, \ud y \\
	&\peq - \int (1+e^{s} \tau_{s}) \eta_{\geq i_{0}} \varpi_{j} M^{(2k+3)} \eta_{\geq i_{0}} \varpi_{j} \calU \, \ud y.
\end{align*}
Observe that the first term on the far RHS is nonnegative, by the dispersive/dissipative property of $\Gmm$/$\Ups$, respectively. Returning to the weighted energy identity, taking the supremum in $j \leq j_{1}$ and in $s \in [\sgm_{0}, \sgm_{1}]$, then proceeding as before, we arrive at
\begin{equation} \label{eq:U-high-w-en}
	\sup_{s \in [\sgm_{0}, \sgm_{1}]} \sup_{j \leq j_{0}} \left( \int \eta_{\geq i_{0}}^{2} \varpi_{j}^{2} \calU(s, y)^{2} \, \ud y \right)^{\frac{1}{2}}
	\leq C + C \sup_{j \leq j_{0}} \int_{\sgm_{0}}^{\sgm_{1}} \left( \mathrm{I} + \mathrm{II} + \mathrm{III} + \mathrm{IV} \right) \, \ud s,
\end{equation}
where 
\begin{align*}
	\mathrm{I} &= (1+e^{s} \tau_{s}) e^{-\mu s} \nrm{[\eta_{\geq i_{0}} \varpi_{j}, \calH] \calU}_{L^{2}}, \\
	\mathrm{II} &=  (1+e^{s} \tau_{s}) e^{-s}  \nrm{\eta_{\geq i_{0}} \varpi_{j} \rd_{y}^{2k+3} \calL(U + e^{(b-1) s} \kpp)}_{L^{2}} \\
	\mathrm{III} &= (1+e^{s} \tau_{s}) \nrm{\eta_{\geq i_{0}} \varpi_{j} M^{(2k+3)}}_{L^{2}}, \\
	\mathrm{IV} &= \nrm{\breve{\calT}_{2k+3} (\eta_{\geq i_{0}} \varpi_{j}) U'}_{L^{2}}.
\end{align*}
We claim that, for some $c > 0$ and $\sgm_{0}$ sufficiently large depending on $A$,
\begin{equation*}
	\sup_{j \leq j_{0}} \int_{\sgm_{0}}^{\sgm_{1}}  \left( \mathrm{I} + \mathrm{II} + \mathrm{III} + \mathrm{IV} \right) \, \ud s
	\aleq 2^{\frac{3}{2} i_{0}} + 2^{-c i_{0}} A + e^{-c \sgm_{0}} A.
\end{equation*}
Since the LHS of \eqref{eq:U-high-w-en} is equivalent to \eqref{eq:U-high-w-outside}, \eqref{eq:ibst-w-high} would follow from the claim by taking $i_{0}$, $A$, and $\sgm_{0}$ large enough (in this order).

The contributions of $\mathrm{I}$ and $\mathrm{II}$ are bounded using Lemma~\ref{lem:mult-w} and \eqref{eq:forcing-L-j-L2} (as well as \eqref{eq:bst-w-1}, \eqref{eq:bst-w-high}, and \eqref{eq:bootstrapmod}) by $e^{-\mu \sgm_{0}} A$ and $e^{-\sgm_{0}} (1+\kpp_{0})$, respectively. To treat $\mathrm{III}$, we begin by noting that
\begin{equation*}
\begin{aligned}
&\nrm{\eta_{\geq i_{0}} \varpi_{j} M^{(2k+3)}}_{L^{2}} \\
&\aleq \nrm{U'}_{L^{\infty}(\tilde{A}_{j})} \left( 2^{(2k+\frac{3}{2} + \frac{1}{b}) j} 2^{((2k+\frac{3}{2})b - \frac{1}{2}) (s-\sgm_{0})}\nrm{U^{(2k+3)}}_{L^{2}(\tilde{A}_{j})} + 2^{(\frac{1}{b} - \frac{1}{2}) j} 2^{(1 - \frac{1}{2} b) (s-\sgm_{0})} \nrm{U'}_{L^{2}(\tilde{A}_{j})} \right),
\end{aligned}
\end{equation*}
where $\tilde{A}_{j}$ is a slight enlargement of $\supp \varpi_{j}$. Indeed, this inequality is proved by harmlessly substituting $U'$ by $\tilde{\eta}_{j} U'$, where $\supp \tilde{\eta}_{j} \subseteq \tilde{A}_{j}$ and $\tilde{\eta}_{j} = 1$ on $\supp \varpi_{j}$, then applying the usual Gagliardo--Nirenberg inequalities. Then using \eqref{eq:decayup-int}, \eqref{eq:bst-w-1}, and \eqref{eq:bst-w-high} (as well as \eqref{eq:bootstrapmod}), we obtain
\begin{equation*}
	\int_{\sgm_{0}}^{\sgm_{1}} \mathrm{III} \, \ud s \aleq (2^{-r i_{0}} + A e^{-r \sgm_{0}})A,
\end{equation*}
which is acceptable. Finally, the contribution of $\mathrm{IV}$ is handled similarly as the term $\breve{\calT}_{1} (\eta_{\geq i_{0}} \varpi_{j})$ in Step~1, where we use Lemma~\ref{lem:high} and \eqref{eq:bst-w-1}--\eqref{eq:bst-w-high} instead of \eqref{eq:ibst1}--\eqref{eq:ibst3} and \eqref{eq:bst-w-1}, respectively. We may show that
\begin{equation*}
	\int_{\sgm_{0}}^{\sgm_{1}} \mathrm{IV} \, \ud s \aleq 2^{(\frac{1}{b}+2k+3) i_{0}} + (2^{-r i_{0}} + A e^{-r \sgm_{0}}) A,
\end{equation*}
which completes the proof. \qedhere \end{proof}

\section{Estimates on modulation parameters and unstable coefficients}\label{sec:unstable}
In this section, we analyze the ODE's satisfied by the modulation parameters and the coefficients $W^{(j)}(s, 0)$. We prove the bootstrap assumptions in Lemma~\ref{lem:main} involving the modulation ODE's and $W^{(2k+1)}(s, 0)$, as well as Lemma~\ref{lem:top} in its entirety.

\subsection{Control of the modulation parameters and \texorpdfstring{$w_{2k+1}$}{w(2k+1)}} 
We establish sharp bounds on the ODE's for modulation parameters, which improve \eqref{eq:bootstrapmod}.
\begin{lemma}[Control of the modulation parameters]\label{lem:modulation}
Assume the hypotheses of Lemma~\ref{lem:main}. Then we have
\begin{align} 
	\abs{e^{s} \tau_{s}} &\leq C_{A} e^{-\min\set{2 \gmm, \mu} s}, \label{eq:modul-tau-est} \\
	\abs{e^{bs} \xi_{s} - (1+e^{s} \tau_{s}) e^{(b-1) s} \kpp } &\leq C_{A} e^{- \min\set{2\gmm, \mu} s}, \label{eq:modul-xi-est} \\
	\abs{e^{(b-1) s} \kpp_{s}} &\leq C_{A} e^{-\min\set{2\gmm, \mu_{0}} s}. \label{eq:modul-kpp-est}
\end{align}
In particular, if $\sgm_{0}$ is sufficiently large (depending only on $A$, $y_{0}$, $\kpp_{0}$), \eqref{eq:ibootstrapmod} holds.
\end{lemma}
\begin{proof}
For $\tau_{s}$, by \eqref{eq:modul-tau}, we have
\begin{align*}
	\abs{e^{s} \tau_{s}} 
	\leq \frac{\abs{F^{(1)}(s, 0)} + \abs{w_2(e^{bs}\xi_s - (1+e^{s}\tau_s)e^{(b-1)s}\kpp )}}{1 - \abs{F^{(1)}(s, 0)}}.
\end{align*}
By \eqref{eq:bootstrapmod}, \eqref{eq:trapped} (in the case $k \geq 2$ for $w_{2}$), \eqref{eq:forcing-H-j} and \eqref{eq:forcing-L-j} with $j = 1$, $\abs{F^{(1)}(s, 0)} \leq C_{A} \left( e^{- \mu s} + e^{- 2\gmm s}\right)$ (since $0 < \mu < 2$), which is acceptable.

For $\xi_{s}$, we use \eqref{eq:modul-xi} to bound
\begin{align*}
	\abs{e^{bs} \xi_{s} - (1+e^{s} \tau_{s}) e^{(b-1) s} \kpp }
	\leq \frac{1}{(2k)!-1} \left(\abs{\left. N(2k) \right|_{y=0}} + (1+\abs{e^{s} \tau_{s}}) \abs{F^{(2k)}(s, 0)} \right),
\end{align*}
where we crucially used \eqref{eq:bootstraplow2} to ensure that $(2k)!+w_{2k+1} \geq (2k)!-1 > 0$ in the denominator.
By \eqref{eq:trapped} and \eqref{eq:orth-cond}, we have $\abs{\left. N(2k) \right|_{y=0}} \aleq e^{-8 \gmm s}$. On the other hand, by \eqref{eq:forcing-H-j}, \eqref{eq:forcing-L-j} with $j = 2k$, and \eqref{eq:modul-tau-est}, we have $\abs{F^{(2k)}(s, 0)} \leq C_{A} e^{- \mu s}$ (since $0 < \mu < 2$). At this point, \eqref{eq:modul-kpp-est} follows.

Finally, for $\kpp_{s}$, we use \eqref{eq:modul-kpp} to bound
\begin{align*}
\abs{e^{(b-1) s} \kpp_{s}} \leq  \abs{e^{bs} \xi_{s} - (1+e^{s} \tau_{s}) e^{(b-1) s} \kpp} + (1+\abs{e^{s} \tau_{s}}) \abs{F^{(0)}(s, 0)}.
\end{align*}
By \eqref{eq:forcing-H-0} and \eqref{eq:forcing-L-j} with $j = 0$, we have $\abs{F^{(0)}(s, 0)} \leq C e^{- \min\set{\mu_{0}, 2-b} s} = C e^{-\mu_{0} s}$ (since $2-b = \frac{2k-1}{2k}$). Combined with the previous bounds \eqref{eq:modul-tau-est} and \eqref{eq:modul-xi-est} for $\abs{e^{s} \tau_{s}}$ and $\abs{e^{bs} \xi_{s} - (1+e^{s} \tau_{s}) e^{(b-1) s} \kpp}$, respectively, \eqref{eq:modul-kpp-est} follows. 

Finally, since $\gmm < \min\set{2 \gmm, \mu, \mu_{0}}$ by \eqref{eq:gmm-def}, \eqref{eq:ibootstrapmod} follows from \eqref{eq:modul-tau-est}--\eqref{eq:modul-kpp-est} provided that $\sgm_{0}$ sufficiently large. \qedhere
\end{proof}

Next, we study the ODE satisfied by $W^{(2k+1)}(s, 0)$ and improve~\eqref{eq:bootstraplow2}. 
\begin{lemma} [Control of the stable coefficient] \label{lem:bootstraplow2}
Assume the hypotheses of Lemma~\ref{lem:main}. Then \eqref{eq:ibootstraplow2} holds true for $\sgm_{0}$ sufficiently large.
\end{lemma}
\begin{proof}
Using equation~\eqref{eq:modul-gen} with $j = 2k+1$, we have the equation for $w_{2k+1} = W^{(2k+1)}(s,0)$ (since $\bar U^{(2k+1)}(0) = (2k)!$ and $\bar U^{(2k+2)}(0) = 0$):
\begin{equation} \label{eq:modul-2k+1}
\begin{aligned}
	&\rd_{s} w_{2k+1} 
	+ (1+e^{s} \tau_{s}) \left. N(2k+1) \right|_{y = 0} \\
	&+ \left(- e^{b s} \xi_{s} + (1+e^{s} \tau_{s}) e^{(b-1) s} \kpp \right) w_{2k+2}
	- e^{s} \tau_{s} \left( (2k)! - (2k+2) w_{2k+1} \right) \\
	& =(1+e^{s} \tau_{s}) F^{(2k+1)}(s, 0).
\end{aligned}
\end{equation}
Note the crucial vanishing of the constant in front of the term $w_{2k+1}$. 
Hence all terms other than $\rd_{s} w_{2k+1}$ decay exponentially as $s$ increases, thanks to the trapping assumption~\eqref{eq:trapped}, the improved bound \eqref{eq:ibootstrapmod}, and the bounds $|W^{(2k+1)}(s,0)| \leq 1$ as well as $|W^{(2k+2)}(s,0)| \leq C A$ (this last bound follows from the estimate~\eqref{eq:uii} in Lemma~\ref{lem:uii} and~\eqref{eq:bst2} plus Sobolev embedding). Moreover, by \eqref{eq:ibootstrapmod}, \eqref{eq:forcing-H-j} and  \eqref{eq:forcing-L-j}, we have $\abs{(1+e^{s} \tau_{s}) F^{(2k+1)}(s, 0)} \aleq_{A} e^{-\mu s}$. Finally, recall from \eqref{eq:w2k+1-ini} that the initial data at $\sgm_{0}$ are such that $|\p^{(2k+1)}_y W(\sgm_{0}, 0) |\leq C \eps_{0} e^{-(1+2kb)\sgm_{0}}$. Therefore, we can prove the improved bootstrap bound ~\eqref{eq:ibootstraplow2} upon choosing $\sgm_{0}$ sufficiently large, as desired. \qedhere
\end{proof}

\subsection{Control of the unstable coefficients: proof of Lemma~\ref{lem:top}} 
The purpose of this subsection is to prove Lemma~\ref{lem:top} (shooting lemma), which is relevant when $k > 1$. We start by establishing the key \emph{outgoing property} of the unstable ODE near the boundary of the trapped region:

\begin{lemma} \label{lem:outgoing}
Under the hypotheses of Lemma~\ref{lem:main}, there exists $c_{0} > 0$ such that, for $\sgm_{0} $ sufficiently large (depending only on $k$, $\mu$, $\gmm$ and $A$), the following holds. For any $s \in [\sgm_{0}, \sgm_{1}]$ such that
\begin{equation} \label{eq:near-bdry}
\frac{1}{2} e^{- \gmm s} < \vec{w}(s) < e^{- \gmm s}.
\end{equation}
we have
\begin{equation} \label{eq:outgoing}
	\rd_{s} \abs{\vec{w}(s)}^{2} > 2 c_{0} \abs{\vec{w}(s)}^{2}.
\end{equation}
\end{lemma}
\begin{proof}
We recall the vector $\vec{w}(s) = (w_2, w_1, \ldots, w_{2k-1})(s)$, which satisfies the following system of ODEs: 
\begin{equation}\label{eq:odesysrec}
  \p_s \vec{w}(s) - D \vec{w}(s) + (1+e^{s} \tau_{s}) \mathcal{N}(\vec{w}(s)) = M \vec{w}(s) + \vec{f}(s),
\end{equation}
Here $D = \mathrm{diag}\,\left( \lmb_{2}, \ldots, \lmb_{2k-1} \right)$ with $\lmb_{j} = 1 - \frac{j-1}{2k}$, so that $1 > \lmb_{2} > \ldots > \lmb_{2k-1} > 0$. We put
\begin{equation*}
	c_{0} = \frac{1}{2} \lmb_{2k-1}.
\end{equation*}
We now evaluate $\rd_{s} \abs{\vec{w}(s)}^{2}$ using \eqref{eq:odesysrec}. The contribution of $D \vec{w}(s)$ gives the main positive term $4 c_{0} \abs{\vec{w}}^{2}$. We claim that the contribution of the remaining terms are bounded below by $- 2 c_{0}  \abs{\vec{w}}^{2}$ if $\sgm_{0}$ is sufficiently large. First, for $\mathcal{N}$, we have, by \eqref{eq:trapped}
\begin{equation}
|\mathcal{N}(\vec{w}(s))| \leq C_k |\vec{w}(s)|^2 \leq C_{k} e^{-\gmm s} \abs{\vec{w}(s)},
\end{equation}
which is acceptable. Next, $M \vec{w}(s)$ is acceptable thanks to
\begin{equation} \label{eq:mod-matrix}
	\abs{M} \leq C_{k} e^{-\gmm s},
\end{equation}
which follows from \eqref{eq:ibootstrapmod}. Finally, by \eqref{eq:forcing-H-j}, \eqref{eq:forcing-L-j} for $j \geq 2$, \eqref{eq:ibootstrapmod} and \eqref{eq:near-bdry}, we have
\begin{equation*}
	\abs{\vec{f}} \leq C_{k, A} e^{-\mu s} \leq 2C_{k, A} e^{-(\mu - \gmm) s} \abs{\vec{w}(s)}.
\end{equation*}
Recalling from \eqref{eq:gmm-def} that $\gmm < \mu_{0} \leq \mu$, the contribution of $\vec{f}(s)$ is acceptable. \qedhere
\end{proof}
Finally, we are ready to prove Lemma~\ref{lem:top}.
\begin{proof}[Proof of Lemma~\ref{lem:top}]
For each $\abs{\vec{w}_{0}} \leq e^{- \gmm \sgm_{0}}$, denote by $U_{\vec{w}_{0}}(s, y)$ the solution with initial data at $s = \sgm_{0}$ induced by $\vec{w}_{0}$ and $W_{0}$, and write $\vec{w}_{\vec{w}_{0}}(s)$ for the vector $(\rd_{y}^{2} U_{\vec{w}_{0}}(s,0), \ldots, \rd_{y}^{2k-1} U_{\vec{w}_{0}}(s,0))$. For the purpose of contradiction, suppose that for all $\vec{w}_{0}$ satisfying $\abs{\vec{w}_{0}} \leq e^{- \gmm \sgm_{0}}$, $U_{\vec{w}_{0}}$ does not remain trapped forever. By Lemma~\ref{lem:main} and a standard bootstrap argument, there exists a unique $\sgm_{trap}(\vec{w}_{0}) > \sgm_{0}$ such that $\abs{\vec{w}_{\vec{w}_{0}}(\sgm_{trap}(\vec{w}_{0}))} = e^{- \gmm \sgm_{trap}(\vec{w}_{0})}$ while $\abs{\vec{w}_{\vec{w}_{0}}(s)} < e^{- \gmm s}$ for all $\sgm_{0} \leq s < \sgm_{trap}(\vec{w}_{0})$ (see the discussion following Definition~\ref{def:trapped}). The key step in the proof is to establish the following:

\smallskip
\noindent {\bf Claim.} The map $H: B_{0}(e^{-\gmm \sgm_{0}}) \to \rd B_{0}(1)$, $\vec{w}_{0} \mapsto e^{-\gmm \sgm_{trap}(\vec{w}_{0})} \vec{w}_{\vec{w}_{0}}(\sgm_{trap}(\vec{w}_{0}))$ is continuous.

\smallskip
Assuming the claim, we first conclude the proof of the lemma. Note, first, that for $\vec{w}_{0} \in \rd B_{0}(e^{-\gmm \sgm_{0}})$, we trivially have $\sgm_{trap} = \sgm_{0}$ and $\vec{w}_{\vec{w}_{0}}(\sgm_{trap}) = \vec{w}_{0}$; hence $H$ is equal to the identity when restricted to the boundary $\rd B_{0}(e^{-\gmm\sgm_{0}})$. Hence, by composing with $B_{0}(1) \to B_{0}(e^{-\gmm \sgm_{0}})$, $\vec{v} \mapsto e^{-\gmm \sgm_{0}} \vec{v}$, we obtain a continuous map from $B_{0}(1)$ into $\rd B_{0}(1)$ that is equal to the identity map on $\rd B_{0}(1)$ (i.e., a continuous retraction $B_{0}(1) \to \rd B_{0}(1)$). As is well-known (cf.~proof of Brouwer's fixed point theorem), such a map does not exist, which is a contradiction.

It remains to establish the claim. Fix $\vec{w}_{0} \in B_{0}(e^{-\gmm \sgm_{0}})$. By definition, \eqref{eq:trapped} holds for $s \in [\sgm_{0}, \sgm_{trap}(\vec{w}_{0})]$, so $U_{\vec{w}_{0}}$ obeys \eqref{eq:ibst1}--\eqref{eq:ibootstraplow2} on $s \in [\sgm_{0}, \sgm_{trap}(\vec{w}_{0})]$. By a standard argument involving analysis of the linearized system, it can be shown that $\vec{v} \mapsto U_{\vec{v}}$ is Lipschitz continuous\footnote{We use Lipschitz continuity here as it is easier to observe for \eqref{eq:fkdv}, which is quasilinear. The minor price we have to pay is that we cannot work with the highest order topology $H^{2k+3}$, but rather with the lower order topology $H^{2k+2}$.} near $\vec{w}_{0}$ in $C_{s} (I; H^{2k+2})$, where $I$ is a fixed open interval containing $[\sgm_{0}, \sgm_{trap}(\vec{w}_{0})]$. By Sobolev embedding, it follows that $\vec{w}_{\vec{v}}(s)$ depends continuously on $(s, \vec{v}) \in I \times B_{\vec{w}_{0}}(\dlt)$ for some $\dlt > 0$. To establish the continuity of $H$ at $\vec{w}_{0}$, it therefore only remains to show that   that $\vec{v} \mapsto \sgm_{trap}(\vec{v})$ is continuous at $\vec{v} = \vec{w}_{0}$. 

To prove the continuity of $\sgm_{trap}$, we begin by using the outgoing property near the boundary (Lemma~\ref{lem:outgoing}) to make the following observation: for an arbitrary sufficiently small number $\eps > 0$, if $U_{\vec{v}}$ is trapped on $[\sgm_{0}, \sgm_{1})$ and $\vec{w}_{\vec{v}}(\sgm_{1}) > e^{-\gmm \sgm_{1} - c_{0} \eps}$ then $\abs{\sgm_{trap}(\vec{v}) - \sgm_{1}} < \eps$. When $\vec{w}_{0} \in \rd B_{0}(e^{-\gmm \sgm_{0}})$, the continuity of $\sgm_{trap}$ at $\vec{w}_{0}$ follows immediately by applying the preceding statement with $\sgm_{1} = \sgm_{0}$, since $\eps > 0$ may be arbitrarily small. When $\vec{w}_{0} \not \in \rd B_{0}(e^{-\gmm \sgm_{0}})$, we have $\sgm_{trap}(\vec{w}_{0}) > \sgm_{0}$. Clearly, there exists $\sgm_{1} \in (\sgm_{0}, \sgm_{trap}(\vec{w}_{0}))$ such that $e^{-\gmm \sgm_{1} - c_{0} \eps} < \vec{w}_{\vec{w}_{0}}(\sgm_{1}) < e^{-\gmm \sgm_{1}}$ and $\abs{\sgm_{1} - \sgm_{trap}(\vec{w}_{0})} < \eps$. By Lemma~\ref{lem:main} and the strict inequality $\vec{w}_{\vec{w}_{0}}(\sgm_{1}) < e^{-\gmm \sgm_{1}}$, for $\vec{v}$ sufficiently close to $\vec{w}_{0}$, the corresponding solution $U_{\vec{v}}$ is trapped on $[\sgm_{0}, \sgm_{1}]$ and obeys $e^{-\gmm \sgm_{1} - c_{0} \eps} < \vec{w}_{\vec{v}}(\sgm_{1}) < e^{-\gmm \sgm_{1}}$. Hence, $\abs{\sgm_{trap}(\vec{v}) - \sgm_{trap}(\vec{w}_{0})} \leq \abs{\sgm_{trap}(\vec{v}) - \sgm_{1}} + \abs{\sgm_{trap}(\vec{w}_{0}) - \sgm_{1}}< 2 \eps$, which implies the desired continuity of $\sgm_{trap}$. \qedhere
\end{proof}

\bibliographystyle{abbrv}
\bibliography{fkdv1.bib}

\begin{thebibliography}{10}

\bibitem{AlDrVo}
N.~Alibaud, J.~Droniou, and J.~Vovelle.
\newblock Occurrence and non-appearance of shocks in fractal {B}urgers
  equations.
\newblock {\em J. Hyperbolic Differ. Equ.}, 4(3):479--499, 2007.

\bibitem{BuIy}
T.~Buckmaster and S.~Iyer.
\newblock Formation of unstable shocks for 2{D} isentropic compressible
  {E}uler.
\newblock ArXiv preprint: \url{https://arxiv.org/abs/2007.15519}, 2020.

\bibitem{BuShVi11}
T.~Buckmaster, S.~Shkoller, and V.~Vicol.
\newblock Formation of point shocks for 3{D} compressible {E}uler.
\newblock ArXiv preprint: \url{https://arxiv.org/abs/1912.04429}, 2019.

\bibitem{BuShVi1}
T.~Buckmaster, S.~Shkoller, and V.~Vicol.
\newblock Formation of shocks for 2{D} isentropic compressible {E}uler.
\newblock ArXiv preprint: \url{https://arxiv.org/abs/1907.03784}, 2019.

\bibitem{BuShVi2}
T.~Buckmaster, S.~Shkoller, and V.~Vicol.
\newblock Shock formation and vorticity creation for 3{D} {E}uler.
\newblock ArXiv preprint: \url{https://arxiv.org/abs/2006.14789}, 2020.

\bibitem{castro2010}
A.~Castro, D.~C\'{o}rdoba, and F.~Gancedo.
\newblock Singularity formations for a surface wave model.
\newblock {\em Nonlinearity}, 23(11):2835--2847, 2010.

\bibitem{CMVP}
K.~R. Chickering, R.~C. Moreno-Vasquez, and G.~Pandya.
\newblock {A}symptotically self-similar shock formation for 1{D} fractal
  {B}urgers equation.
\newblock ArXiV Preprint: \url{https://arxiv.org/pdf/2105.15128.pdf}, 2021.

\bibitem{christodoulou1}
D.~Christodoulou.
\newblock {\em The formation of shocks in 3-dimensional fluids}.
\newblock EMS Monographs in Mathematics. European Mathematical Society (EMS),
  Z\"{u}rich, 2007.

\bibitem{christodoulou2}
D.~Christodoulou.
\newblock {\em The shock development problem}.
\newblock EMS Monographs in Mathematics. European Mathematical Society (EMS),
  Z\"{u}rich, 2019.

\bibitem{Collot2018}
C.~Collot, T.-E. Ghoul, and N.~Masmoudi.
\newblock Singularity formation for {B}urgers equation with transverse
  viscosity.
\newblock ArXiv preprint: \url{https://arxiv.org/abs/1803.07826}, Mar. 2018.

\bibitem{CoEs}
A.~Constantin and J.~Escher.
\newblock Wave breaking for nonlinear nonlocal shallow water equations.
\newblock {\em Acta Math.}, 181(2):229--243, 1998.

\bibitem{CoWu}
P.~Constantin and J.~Wu.
\newblock Regularity of {H}\"{o}lder continuous solutions of the supercritical
  quasi-geostrophic equation.
\newblock {\em Ann. Inst. H. Poincar\'{e} Anal. Non Lin\'{e}aire},
  25(6):1103--1110, 2008.

\bibitem{CoMaMe}
R.~C\^{o}te, Y.~Martel, and F.~Merle.
\newblock Construction of multi-soliton solutions for the {$L^2$}-supercritical
  g{K}d{V} and {NLS} equations.
\newblock {\em Rev. Mat. Iberoam.}, 27(1):273--302, 2011.

\bibitem{DoDuLi}
H.~Dong, D.~Du, and D.~Li.
\newblock Finite time singularities and global well-posedness for fractal
  {B}urgers equations.
\newblock {\em Indiana Univ. Math. J.}, 58(2):807--821, 2009.

\bibitem{EgFo}
J.~Eggers and M.~A. Fontelos.
\newblock {\em Singularities: formation, structure, and propagation}.
\newblock Cambridge Texts in Applied Mathematics. Cambridge University Press,
  Cambridge, 2015.

\bibitem{Elgindi}
T.~Elgindi.
\newblock {F}inite-{T}ime {S}ingularity {F}ormation for
  $\mathcal{C}^{1,\alpha}$ {S}olutions to the {I}ncompressible {E}uler
  {E}quations on $\mathbb{R}^3$.
\newblock ArXiv preprint: \url{https://arxiv.org/abs/1904.04795}, 2019.

\bibitem{Hu2017}
V.~M. Hur.
\newblock Wave breaking in the {W}hitham equation.
\newblock {\em Adv. Math.}, 317:410--437, 2017.

\bibitem{HuTa}
V.~M. Hur and L.~Tao.
\newblock Wave breaking for the {W}hitham equation with fractional dispersion.
\newblock {\em Nonlinearity}, 27(12):2937--2949, 2014.

\bibitem{Ibdah}
H.~Ibdah.
\newblock Lipschitz continuity of solutions to drift-diffusion equations in the
  presence of nonlocal terms.
\newblock ArXiV Preprint: \url{https://arxiv.org/abs/2006.01859.pdf}, 2020.

\bibitem{KiNaSh}
A.~Kiselev, F.~Nazarov, and R.~Shterenberg.
\newblock Blow up and regularity for fractal {B}urgers equation.
\newblock {\em Dyn. Partial Differ. Equ.}, 5(3):211--240, 2008.

\bibitem{klein2018}
C.~Klein, F.~Linares, D.~Pilod, and J.-C. Saut.
\newblock On {W}hitham and related equations.
\newblock {\em Stud. Appl. Math.}, 140(2):133--177, 2018.

\bibitem{KlSa}
C.~Klein and J.-C. Saut.
\newblock A numerical approach to blow-up issues for dispersive perturbations
  of {B}urgers' equation.
\newblock {\em Phys. D}, 295/296:46--65, 2015.

\bibitem{KleinSautWang}
C.~Klein, J.-C. Saut, and Y.~Wang.
\newblock On the modified fractional {K}orteweg-de {V}ries and related
  equations.
\newblock ArXiv preprint: \url{https://arxiv.org/abs/2010.05081}, 2020.

\bibitem{LuSp}
J.~Luk and J.~Speck.
\newblock Shock formation in solutions to the 2{D} compressible {E}uler
  equations in the presence of non-zero vorticity.
\newblock {\em Invent. Math.}, 214(1):1--169, 2018.

\bibitem{MRRS-NLS}
F.~Merle, P.~Rapha\"el, I.~Rodnianski, and J.~Szeftel.
\newblock On blow up for the energy super critical defocusing non linear
  {S}chr\"odinger equations.
\newblock ArXiv preprint: \url{https://arxiv.org/abs/1912.11005}, 2019.

\bibitem{MRRS-E}
F.~Merle, P.~Rapha\"el, I.~Rodnianski, and J.~Szeftel.
\newblock On smooth self similar solutions to the compressible {E}uler
  equations.
\newblock ArXiv preprint: \url{https://arxiv.org/abs/1912.10998}, 2019.

\bibitem{MRRS-NS}
F.~Merle, P.~Rapha\"el, I.~Rodnianski, and J.~Szeftel.
\newblock On the implosion of a three dimensional compressible fluid.
\newblock ArXiv preprint: \url{https://arxiv.org/abs/1912.11009}, 2019.

\bibitem{NaSh}
P.~I. Naumkin and I.~A. Shishmar\"{e}v.
\newblock {\em Nonlinear nonlocal equations in the theory of waves}, volume 133
  of {\em Translations of Mathematical Monographs}.
\newblock American Mathematical Society, Providence, RI, 1994.
\newblock Translated from the Russian manuscript by Boris Gommerstadt.

\bibitem{Rimah}
A.~Rimah~Said.
\newblock On the {C}auchy problem of dispersive {B}urgers type equations.
\newblock ArXiv preprint: \url{https://arxiv.org/abs/2006.03803}, 2021.

\bibitem{SautWangMod}
J.-C. Saut and Y.~Wang.
\newblock {G}lobal dynamics of small solutions to the modified fractional
  {K}orteweg-de {V}ries and nonlinear {S}chr\"odinger equations.
\newblock ArXiv preprint: \url{https://arxiv.org/abs/2010.05081}, 2020.

\bibitem{SautWang}
J.-C. Saut and Y.~Wang.
\newblock The wave breaking for {W}hitham-type equations revisited.
\newblock ArXiv preprint: \url{https://arxiv.org/abs/2103.03588}, 2020.

\bibitem{Sel}
R.~L. Seliger.
\newblock A note on the breaking of waves.
\newblock {\em Proceedings of the Royal Society of London. Series A,
  Mathematical and Physical Sciences}, 303(1475):493--496, 1968.

\bibitem{Sil1}
L.~Silvestre.
\newblock H\"{o}lder estimates for advection fractional-diffusion equations.
\newblock {\em Ann. Sc. Norm. Super. Pisa Cl. Sci. (5)}, 11(4):843--855, 2012.

\bibitem{Sil2}
L.~Silvestre.
\newblock On the differentiability of the solution to an equation with drift
  and fractional diffusion.
\newblock {\em Indiana Univ. Math. J.}, 61(2):557--584, 2012.

\bibitem{whitham}
G.~B. Whitham.
\newblock Variational methods and applications to water waves.
\newblock {\em Proceedings of the Royal Society of London. Series A,
  Mathematical and Physical Sciences}, 299(1456):6--25, 1967.

\bibitem{Whi}
G.~B. Whitham.
\newblock {\em Linear and nonlinear waves}.
\newblock Pure and Applied Mathematics (New York). John Wiley \& Sons, Inc.,
  New York, 1999.
\newblock Reprint of the 1974 original, A Wiley-Interscience Publication.

\bibitem{Yang2020}
R.~Yang.
\newblock Shock formation for the {B}urgers-{H}ilbert equation.
\newblock ArXiv preprint: \url{https://arxiv.org/abs/2006.05568}, 2020.

\end{thebibliography}

\end{document}